\documentclass[reqno,11pt]{amsart}
\pdfoutput=1

\usepackage{eucal}
\usepackage{outlines}
\usepackage{bbm}
\usepackage{amssymb,amsfonts}
\usepackage{stmaryrd}

\usepackage{autobreak}
\usepackage{enumerate}
\usepackage{enumitem}
\usepackage{mathrsfs}
\usepackage{tikz-cd}
\usepackage{mathtools}
\usepackage{yhmath}
\usepackage{quiver}
\usepackage{bbm}
\usepackage{comment}
\usepackage{bm}
\usepackage{quiver} 
\usepackage{scalerel}
\usepackage{stackengine,wasysym}
\usepackage{csquotes}
\usepackage{xcolor}
\usepackage{amsmath,amsthm}
\usepackage{etoolbox}
\usepackage{calligra,mathrsfs}
\usepackage[style=alphabetic,
minalphanames=1,
maxalphanames=20,
maxbibnames=99]{biblatex}
\usepackage[colorlinks,allcolors=blue]{hyperref}
\usepackage{zref-clever}
\zcsetup{cap,nameinlink}

\newcommand{\cref}[1]{\zcref{#1}}
\newcommand{\Cref}[1]{\zcref[S]{#1}}

\newcommand{\labelcref}[1]{\zcref[noname]{#1}}

\ExplSyntaxOn
\newcommand{\zcrefglobalstringname}[2]{
	\__zrefclever_opt_varname_lang_type:nnnn{\languagename}{#1}{#2}{tl}
}
\newcommand{\zcreflocalstringname}[2]{
	\__zrefclever_opt_varname_type:een{#1}{#2}{tl}
}
\newcommand{\zcrefgetstring}[2]{
	\__zrefclever_provide_langfile:e { \languagename }
	\__zrefclever_process_language_settings:
	\ifcsvoid{\zcreflocalstringname{#1}{#2}}
	{\csname \zcrefglobalstringname{#1}{#2} \endcsname}
	{\csname \zcreflocalstringname{#1}{#2} \endcsname}
}
\ExplSyntaxOff
\newcommand{\zcreftypename}[1]{\zcrefgetstring{#1}{Name-sg}}

\newtheorem{sharedtheoremcounter}{}[section]
\undef\sharedtheoremcounter 

\NewDocumentCommand{\newzctheorem}{mO{#1}O{\zcreftypename{#2}}}{%
    \newtheorem{#1}[sharedtheoremcounter]{#3}%

    \zcsetup{countertype={#1=#2}}%

    \AtBeginEnvironment{#1}{%
        \zcsetup{countertype={sharedtheoremcounter=#2}}%
    }%
}
\zcRefTypeSetup{convention}{
	Name-sg = Convention ,
	name-sg = convention ,
	Name-pl = Conventions ,
	name-pl = conventions ,
}
\zcRefTypeSetup{keythm}{
	Name-sg = Key Theorem ,
	name-sg = key theorem ,
	Name-pl = Key Theorems ,
	name-pl = key theorems ,
}
\zcRefTypeSetup{mainthm}{
	Name-sg = Main Theorem ,
	name-sg = main theorem ,
	Name-pl = Main Theorems ,
	name-pl = main theorems ,
}
\zcRefTypeSetup{exmp}{
	Name-sg = Example ,
	name-sg = example ,
	Name-pl = Examples ,
	name-pl = examples ,
}
\zcRefTypeSetup{exmps}{
	Name-sg = Examples ,
	name-sg = examples ,
	Name-pl = Examples ,
	name-pl = examples ,
}
\zcRefTypeSetup{notn}{
	Name-sg = Notation ,
	name-sg = notation ,
	Name-pl = Notations ,
	name-pl = notations ,
}
\zcRefTypeSetup{notns}{
	Name-sg = Notations ,
	name-sg = notations ,
	Name-pl = Notations ,
	name-pl = notations ,
}
\zcRefTypeSetup{disc}{
	Name-sg = Discussion ,
	name-sg = discussion ,
	Name-pl = Discussions ,
	name-pl = discussions ,
}
\zcRefTypeSetup{addm}{
	Name-sg = Addendum ,
	name-sg = addendum ,
	Name-pl = Addenda ,
	name-pl = addenda ,
}
\zcRefTypeSetup{exer}{
	Name-sg = Exercise ,
	name-sg = exercise ,
	Name-pl = Exercises ,
	name-pl = exercises ,
}
\zcRefTypeSetup{property}{
	Name-sg = Property ,
	name-sg = property ,
	Name-pl = Properties ,
	name-pl = properties ,
}
\zcRefTypeSetup{properties}{
	Name-sg = Properties ,
	name-sg = properties ,
	Name-pl = Properties ,
	name-pl = properties ,
}
\zcRefTypeSetup{warn}{
	Name-sg = Warning ,
	name-sg = warning ,
	Name-pl = Warnings ,
	name-pl = warnings ,
}
\zcRefTypeSetup{conv}{
	Name-sg = Convention ,
	name-sg = convention ,
	Name-pl = Conventions ,
	name-pl = conventions ,
}
\zcRefTypeSetup{assumption}{
	Name-sg = Assumption ,
	name-sg = assumption ,
	Name-pl = Assumptions ,
	name-pl = assumptions ,
}
\zcRefTypeSetup{interlude}{
	Name-sg = Interlude ,
	name-sg = interlude ,
	Name-pl = Interludes ,
	name-pl = interludes ,
}
\zcRefTypeSetup{importantrem}{
	Name-sg = Important Remark ,
	name-sg = important remark ,
	Name-pl = Important Remarks ,
	name-pl = important remarks ,
}
\zcRefTypeSetup{rems}{
	Name-sg = Remarks ,
	name-sg = remarks ,
	Name-pl = Remarks ,
	name-pl = remarks ,
}
\zcRefTypeSetup{warning}{
	Name-sg = Warning ,
	name-sg = warning ,
	Name-pl = Warnings ,
	name-pl = warnings ,
}
\zcRefTypeSetup{sch}{
	Name-sg = Scholium ,
	name-sg = scholium ,
	Name-pl = Scholia ,
	name-pl = scholia ,
}
\zcRefTypeSetup{dig}{
	Name-sg = Digression ,
	name-sg = digression ,
	Name-pl = Digressions ,
	name-pl = digressions ,
}
\zcRefTypeSetup{con}{
	Name-sg = Construction ,
	name-sg = construction ,
	Name-pl = Constructions ,
	name-pl = constructions ,
}
\zcRefTypeSetup{irem}{
	Name-sg = Important Remark ,
	name-sg = important remark ,
	Name-pl = Important Remarks ,
	name-pl = important remarks ,
}
\zcRefTypeSetup{tempnotation}{
	Name-sg = Temporary Notation ,
	name-sg = temporary notation ,
	Name-pl = Temporary Notations ,
	name-pl = temporary notations ,
}
\zcRefTypeSetup{conventions}{
	Name-sg = Conventions ,
	name-sg = conventions ,
	Name-pl = Conventions ,
	name-pl = conventions ,
}

\zcRefTypeSetup{defns}{
	Name-sg = Definitions ,
	name-sg = definitions ,
	Name-pl = Definitions ,
	name-pl = definitions ,
}

\zcRefTypeSetup{claim}{
	Name-sg = Claim ,
	name-sg = claim ,
	Name-pl = Claims ,
	name-pl = claims ,
}

\zcRefTypeSetup{notation}{
	Name-sg = Notation ,
	name-sg = notation ,
	Name-pl = Notations ,
	name-pl = notations ,
}

\zcRefTypeSetup{notnn}{
	Name-sg = Notation ,
	name-sg = notation ,
	Name-pl = Notations ,
	name-pl = notations ,
}

\zcRefTypeSetup{warning}{
	Name-sg = Warning ,
	name-sg =  warning ,
	Name-pl = Warnings ,
	name-pl = warnings ,
}

\zcRefTypeSetup{quest}{
	Name-sg = Question ,
	name-sg = question ,
	Name-pl = Questions ,
	name-pl = questions ,
}

\theoremstyle{plain}

\newzctheorem{theorem}
\newzctheorem{keythm}[keythm][Key Theorem]
\newzctheorem{lem}[lemma]
\newzctheorem{thm}[theorem]
\newzctheorem{thmm}[theorem]
\newzctheorem{cor}[corollary]
\newzctheorem{prop}[proposition]
\newzctheorem{conj}[conjecture]
\newzctheorem{quest}[quest]
\newzctheorem{claim}[claim][Claim]
\newzctheorem{mainthm}[mainthm][Main Theorem]

\theoremstyle{definition}

\newzctheorem{defi}[definition]
\newzctheorem{convention}
\newzctheorem{defn}[definition][Definition]	
\newzctheorem{defns}[defns][Definitions]
\newzctheorem{exmp}[exmp][Example]
\newzctheorem{exmps}[exmps][Examples]
\newzctheorem{notn}[notn][Notation]
\newzctheorem{notns}[notns][Notations]
\newzctheorem{disc}[disc][Discussion]
\newzctheorem{addm}[addm][Addendum]
\newzctheorem{exer}[exer][Exercise]
\newzctheorem{property}[property][Property]
\newzctheorem{properties}[properties][Properties]
\newzctheorem{warn}[warn][Warning]
\newzctheorem{remark}[remark][Remark]
\newzctheorem{assumption}[assumption][Assumption]
\newzctheorem{interlude}[interlude][Interlude]
\newzctheorem{importantrem}[importantrem][Important Remark]
\newzctheorem{con}[con][Construction]

\theoremstyle{remark}

\newzctheorem{rems}[rems][Remarks]
\newzctheorem{warning}[warning][Warning]
\newzctheorem{sch}[sch][Scholium]	
\newzctheorem{rem}[remark][Remark]
\newzctheorem{dig}[dig][Digression]
\newzctheorem{notation}[notation][Notation]
\newzctheorem{irem}[irem][Important Remark]
\newzctheorem{tempnotation}[tempnotation][Temporary Notation]
\newzctheorem{conventions}[conventions][Conventions]
\newzctheorem{notnn}[notnn][Notation]
\newzctheorem{conv}[conv][Convention]

\makeatletter
\let\c@equation\c@sharedtheoremcounter

\makeatother

\newcommand{\beqn}{\begin{equation}}
	\newcommand{\eeqn}{\end{equation}}
\newcommand{\bclaim}{\begin{claim}}
	\newcommand{\eclaim}{\end{claim}}
\newcommand{\blem}{\begin{lem}}
	\newcommand{\elem}{\end{lem}}
\newcommand{\bproof}{\begin{proof}}
	\newcommand{\eproof}{\end{proof}}
\newcommand{\bdef}{\begin{defn}}
	\newcommand{\edefn}{\end{defn}}
\newcommand{\bprop}{\begin{prop}}
	\newcommand{\eprop}{\end{prop}}
\newcommand{\bthm}{\begin{thm}}
	\newcommand{\ethm}{\end{thm}}
\newcommand{\brem}{\begin{rem}}
	\newcommand{\erem}{\end{rem}}
\newcommand{\bcor}{\begin{cor}}
	\newcommand{\ecor}{\end{cor}}

\newcommand{\bbC}{\mathbb{C}}

\newcommand{\bbE}{\mathbb{E}}

\newcommand{\bbG}{\mathbb{G}}

\newcommand{\bbN}{\mathbb{N}}
\newcommand{\bbO}{\mathbb{O}}

\newcommand{\bbT}{\mathbb{T}}

\newcommand{\bbZ}{\mathbb{Z}}

\newcommand{\calB}{\mathcal{B}}

\newcommand{\calD}{\mathcal{D}}

\newcommand{\calF}{\mathcal{F}}
\newcommand{\calG}{\mathcal{G}}
\newcommand{\calH}{\mathcal{H}}

\newcommand{\calL}{\mathcal{L}}
\newcommand{\calM}{\mathcal{M}}

\newcommand{\calO}{\mathcal{O}}
\newcommand{\calP}{\mathcal{P}}

\newcommand{\calS}{\mathcal{S}}
\newcommand{\calT}{\mathcal{T}}
\newcommand{\calU}{\mathcal{U}}

\newcommand{\calW}{\mathcal{W}}
\newcommand{\calX}{\mathcal{X}}
\newcommand{\calY}{\mathcal{Y}}
\newcommand{\calZ}{\mathcal{Z}}

\newcommand{\scrA}{\mathscr{A}}
\newcommand{\scrB}{\mathscr{B}}
\newcommand{\scrC}{\mathscr{C}}
\newcommand{\scrD}{\mathscr{D}}
\newcommand{\scrE}{\mathscr{E}}

\newcommand{\scrR}{\mathscr{R}}

\newcommand{\frakm}{\mathfrak{m}}
\newcommand{\fraki}{\mathfrak{i}}
\newcommand{\frakX}{\mathfrak{X}}

\newcommand{\frakY}{\mathfrak{Y}}

\newcommand{\frakp}{\mathfrak{p}}

\renewcommand{\a}{\alpha}
\renewcommand{\b}{\beta}

\renewcommand{\i}{\iota}

\newcommand{\on}{\operatorname}
\newcommand{\wit}{\widetilde}

\newcommand{\wih}{\widehat}

\newcommand{\hook}{\hookrightarrow}
\newcommand{\F}{\mathcal{F}}

\newcommand{\kb}{k[\b^{\pm 1}]}

\newcommand{\mathendash}{\text{\textendash}}
\renewcommand{\mod}{\mathendash\mathrm{mod}}

\mathchardef\mhyphen="2D

\renewcommand{\bmod}{\mathendash\mathbf{mod}}
\newcommand{\comod}{\mathendash\mathrm{comod}}

\newcommand{\heart}{\heartsuit}

\newcommand{\oblv}{\mathrm{oblv}}

\newcommand{\Cat}{\mathcal{C}\mathrm{at}}
\newcommand{\id}{\mathrm{id}}
\newcommand{\Map}{\calM\mathrm{ap}}
\newcommand{\End}{\mathrm{End}}
\newcommand{\cl}{\mathrm{cl}}

\newcommand{\sm}{\mathrm{sm}}

\newcommand{\Alg}{\mathrm{Alg}}
\newcommand{\CAlg}{\mathrm{CAlg}}

\newcommand{\hL}{\widehat{\mathcal{L}}}

\newcommand{\dgcat}{\mathbf{dgcat}}

\newcommand{\QCoh}{\on{\mathsf{QCoh}}}
\newcommand{\IndCoh}{\on{\mathsf{IndCoh}}}
\newcommand{\Vect}{\mathsf{Vect}}

\newcommand{\Coh}{\on{\mathsf{Coh}}}
\newcommand{\Perf}{\on{\mathsf{Perf}}}
\newcommand{\ShvCat}{\on{\mathsf{ShvCat}}}
\newcommand{\act}{\mathsf{act}}
\newcommand{\Ind}{\on{\mathsf{Ind}}}

\newcommand{\Rep}{\on{\mathsf{Rep}}}

\newcommand{\pr}{\on{pr}}

\newcommand{\Fun}{\on{Fun}}

\newcommand{\op}{\on{op}}
\newcommand{\pt}{\on{pt}}

\newcommand{\st}{\on{st}}

\newcommand{\DGCat}{\mathrm{DGCat}}

\newcommand{\res}{\on{res}}

\newcommand{\Tot}{\on{Tot}}

\newcommand{\bCoh}{\mathbf{Coh}}
\newcommand{\bReg}{\mathbf{Reg}}

\newcommand{\Aut}{\on{Aut}}

\DeclareMathOperator{\colim}{colim}

\DeclareMathOperator{\Mod}{Mod}

\DeclareMathOperator{\Sym}{Sym}
\DeclareMathOperator{\fib}{fib}

\DeclareMathOperator{\Spec}{Spec}
\DeclareMathOperator{\Hom}{Hom}

\DeclareMathOperator{\sHom}{\mathscr{H}\text{\kern -3pt {\calligra\large om}}\,}
\DeclareMathOperator{\Spf}{Spf}
\DeclareMathOperator{\Ann}{Ann}

\renewcommand{\Pr}{\calP\mathrm{r}}

\newcommand{\Ch}{\mathsf{Ch}}

\newcommand{\Grp}{\mathrm{Grp}}

\newcommand{\im}{\on{im}}

\newcommand{\ev}{\mathrm{ev}}

\newcommand{\Mor}{\calM\mathrm{or}}

\renewcommand{\sHom}{\calH\mathrm{om}}

\renewcommand{\res}{\mathrm{res}}

\newcommand{\ku}{k[u^{\pm 1}]}

\newcommand{\idm}{\mathrm{idm}}

\newcommand{\enh}{\mathrm{enh}}

\newcommand{\Stk}{\mathsf{Stk}}

\newcommand{\Grpd}{\mathsf{Grpd}}

\newcommand{\laft}{\mathrm{laft}}

\newcommand{\dR}{\mathrm{dR}}

\newcommand{\tT}{t\mathrm{Tate}}
\newcommand{\Pot}{\calD^{\mathrm{pot}}}

\newcommand{\Crys}{\mathsf{Crys}}
\newcommand{\red}{\mathrm{red}}

\newcommand{\bbone}{\bm{1}}

\newcommand{\Ani}{\mathsf{Ani}}
\newcommand{\triv}{\mathrm{triv}}
\newcommand{\Ob}{\calO\mathrm{b}}
\newcommand{\MapStk}{\mathrm{Map}}

\newcommand{\pot}{\mathrm{pot}}
\newcommand{\sfMap}{\mathsf{Map}}
\newcommand{\PreStk}{\mathsf{PreStk}}
\newcommand{\git}{\mathbin{
  \mathchoice{/\mkern-6mu/}
    {/\mkern-6mu/}
    {/\mkern-5mu/}
    {/\mkern-5mu/}}}
\newcommand{\dgCat}{\mathrm{dgCat}}
\newcommand{\bLoc}{\mathbf{Loc}}

\newcommand{\Spc}{\mathsf{Spc}}
\newcommand{\cn}{\mathrm{cn}}
\newcommand{\St}{\mathrm{St}}
\newcommand{\calAut}{\mathcal{A}\mathrm{ut}}
\newcommand{\Act}{\mathrm{Act}}

\DeclareFontFamily{U}{mathx}{\hyphenchar\font45}
\DeclareFontShape{U}{mathx}{m}{n}{
	<5> <6> <7> <8> <9> <10>
	<10.95> <12> <14.4> <17.28> <20.74> <24.88>
	mathx10
}{}
\DeclareSymbolFont{mathx}{U}{mathx}{m}{n}
\DeclareFontSubstitution{U}{mathx}{m}{n}
\DeclareMathAccent{\widecheck}{0}{mathx}{"71}
\DeclareMathAccent{\wideparen}{0}{mathx}{"75}

\renewcommand{\colim}{\operatornamewithlimits{colim}}

\makeatletter
\newcommand*\bigcdot{\mathpalette\bigcdot@{.5}}
\newcommand*\bigcdot@[2]{\mathbin{\vcenter{\hbox{\scalebox{#2}{$\m@th#1\bullet$}}}}}
\makeatother

\usepackage{geometry}
\makeatletter
\renewcommand{\widehat}{%
  \mathaccent"0\hexnumber@\symyhlargesymbols62\relax
}
\makeatother

\title{Kashiwara's Lemma for potent $\calD$-modules}

\author{Jacob Erlikhman and Kendric Schefers}

\begin{document}
\raggedbottom
\begin{abstract}
On stacks there is a richer version of $\calD$-modules---called \textit{potent $\calD$-modules}---which decompletes at equivariant parameters, in the same sense that genuine equivariant $K$-theory decompletes equivariant cohomology. In this paper, we prove a ``genuine'' version of Kashiwara's lemma for $\calD$-modules. Namely, we show that for a broad class of geometric stacks $X$,
if $Z \subset X$ is a closed substack, the category
$\Pot(Z)$ of potent $\calD$-modules on $Z$ is equivalent
under pushforward to the category $\Pot _Z(X)$ of potent $\calD$-modules
on $X$ with support contained in $Z$. 
\end{abstract}

\maketitle

\tableofcontents


\section{Introduction}
\subsection{Statement of main result}
Kashiwara's Lemma for $\mathcal{D}$-modules asserts
that, for a closed embedding of smooth schemes $Z \xhookrightarrow{i} X$ 
over an algebraically closed characteristic 0 field $k$,
the pushforward $i_+: \calD_Z\mod \to \calD_X\mod$\footnotemark
\footnotetext{$\calD_X\mod$ here denotes the (derived)
$\infty$-category of $\calD_X$-modules.}
is fully faithful, and
identifies its essential image as $\calD_X\mod_Z$, the full subcategory of $\calD_X\mod$ spanned
by $\mathcal{D}$-modules set-theoretically supported on $Z$ (cf. \cite[Theorem 1.6.1 and Corollary 1.6.2]{Hotta}). Since the category
$\calD_X\mod_Z$ is insensitive to the scheme structure, Kashiwara's Lemma in particular shows
that the derived category of $\calD$-modules is invariant under nilthickening.
In \cite{BZNLoops}, Ben-Zvi--Nadler demonstrate that the ($2$-periodicized) category
of $\calD$-modules on $X$
can be recovered, using Koszul duality and the Hochschild--Kostant---Rosenberg (HKR) isomorphism, as
the category of cyclic sheaves on the derived loop stack of $X$\footnotemark
	\footnotetext{The derived loop stack of $X$ is the derived stack
	$\calL X := \Map(S^1, X) \simeq X \times_{X \times X} X$,
	where $S^1$ is the algebraic stack $B\bbZ$, or, equivalently,
	the constant derived stack with value $S^1 \in \Spc$.}---more precisely,
they establish an equivalence 
\[\IndCoh(\calL X)^{\tT} \simeq \calD_X\mod_{\bbZ/2},\footnotemark\]
	\footnotetext{The category $\IndCoh(\calL X)^{\tT}$ is the regularized
	Tate invariants, defined in \zcref{sssec:Tate invariants} below, with respect to the $S^1$-action on $\IndCoh(\calL X)$
	given by loop rotation. Concretely, it is the ind-completion
	of the ordinary Tate invariants, $\Coh(\calL X)^{tS^1}$, with respect to the $S^1$-action on $\Coh(\calL X)$
	(appearing as $\simeq \Perf_X(\calL X)$ in \cite{BZNLoops}).}
where $\calD_X\mod_{\bbZ/2}$ denotes the $2$-periodicization of $\calD_X\mod$.
When $X$ is a geometric stack, rather than a scheme, the natural HKR morphism $\bbT_X[-1] \to \calL X$ furnishes
an isomorphism only upon taking the formal completion, along the zero section $X \hook \bbT_X[-1]$
and along the constant loops $X \hook \calL X$, respectively, on each side: $\wih{\bbT}_X[-1] \xrightarrow{\simeq} \wih{\calL}X$. 
In the same paper, Ben-Zvi--Nadler show (\cite[Theorem 6.13]{BZNLoops}) that for a smooth geometric
stack $X$, the ($2$-periodicized) category of $\calD$-modules on $X$ can be recovered,
using this equivalence of completions, as the category of cyclic sheaves on $\wih{\calL} X$.

Thus, the category of cyclic sheaves $\mathsf{IndCoh}(\mathcal{L}X)^{\tT}$ on the \emph{full} 
loop space $\calL X$ is a distinct entity---a kind of enlargement of the ($2$-periodicized) 
category of $\calD$-modules on $X$.\footnotemark
\footnotetext{If $X$ has affine diagonal, the constant loop morphism $X\to \mathcal{L}X$ is a closed immersion. Hence, by \zcref{lem:support-is-formal-completion} below, pushforward
along the inclusion $\wih{\calL}X \hook \calL X$ induces
a fully faithful inclusion of $\IndCoh(\wih{\calL} X)^{\tT}$ 
into $\IndCoh(\calL X)^{\tT}$.}
This category, dubbed the category of
\emph{potent $\calD$-modules} on $X$,
and denoted by $\Pot(X)$, is a central object
in the program of potent
categorical representation theory proposed in 
the recent ICM address of Ben-Zvi--Nadler (\cite{BZPotent}).
In the philosophy expounded therein, and explained
below in \zcref{ssec: potent philosophy}, $\Pot(X)$
is a \emph{decompletion} of the ordinary $2$-periodicized
category of $\calD$-modules.

For a map of stacks $f: X \to Y$, there is an induced $S^1$-equivariant map
of loop stacks, $\calL f: \calL X \to \calL Y$. If $X$ and $Y$ are reasonable, one can show that
the $!$-pullback $\calL f^!: \IndCoh(\calL Y) \to \IndCoh(\calL X)$ induces a map
on regularized Tate invariants, giving a functor
$f^!: \Pot(Y) \to \Pot(X)$. In forthcoming work \cite{ES-6FF}, 
we construct a six functor
formalism extending the assignment
	\[(X \xrightarrow{f} Y) \leadsto (\Pot(Y) \xrightarrow{f^!} \Pot(X)),\] 
and compatible with the six functor formalism on
the full subcategory of ordinary $\calD$-modules.
In particular, given a closed immersion of stacks,
$i: Z \hook X$, there is a well-defined pushforward,
$i_*: \Pot(Z) \to \Pot(X)$, as well
as an appropriate notion of objects in $\Pot(X)$ supported on $Z$.
The aim of this paper is to show that, for a broad
class of stacks, $\Pot$ satisfies the
same invariance under nilthickening as its completed
counterpart.

\begin{thm}[``Kashiwara's Lemma for $\Pot$" (\zcref{thm: Kashiwara's lemma for scalloped stacks})]
	\label{thm:intro kashiwara's lemma for scalloped stacks}
	Let $X$ be a linearly scalloped stack almost of finite type and with affine diagonal and $i:Z \hookrightarrow X$ a closed immersion almost of finite presentation.
	Then $i_*: \Pot(Z) \to \Pot(X)$ is fully faithful 
	and its essential image is $\Pot_Z(X)$, the full subcategory
	of $\Pot(X)$ spanned by objects set-theoretically supported on $Z$.
\end{thm}  

\begin{rem}
	The class of linearly scalloped stacks, defined in \cite[Definition A.1]{KhanRavi}, encompasses many of the most
	prominently featured finite-dimensional stacks in geometric representation theory.
	For example, any stack of the form $Y/H$, where $H$ is a subgroup of a reductive group $G$ and 
	$Y$ is an $H$-quasi-projective derived scheme, is a linearly scalloped stack. 
	In particular, the Hecke stack $B\backslash G/B$ is linearly scalloped.
	Roughly speaking, they are stacks built successively using Nisnevich squares out of global quotient 
	stacks of the form $Y/G$ for $G$ a reductive group and $Y$ a $G$-quasi-projective scheme.
\end{rem}

\subsection{Motivation}
\label{ssec: potent philosophy}

\subsubsection{Explanation of the name ``potent $\calD$-modules"}
The name ``potent $\calD$-module" is inspired by the following
observation.\footnotemark
\footnotetext{Any unjustified assertion in the following may be found in \cite{BZNLoops}.} 
Given a stack $X$, the unipotent loop space
of $X$, denoted $\calL^u X$, is defined to be the mapping stack $\Map(B\bbG_a, X)$.
Precomposition with the affinization map $S^1 \to B\bbG_a$ gives an inclusion
$\calL^u X \hook \calL X$. The inclusion $\wih{\calL} X \hook \calL X$
factors through this latter map, giving an inclusion $\wih{\calL}X \hook \calL^u X$.
By a theorem of \cite{BZCHN2} (recalled as \zcref{thm: BZCHN Tate sheaves on unipotent equals formal} below), when $X$ is a geometric stack, pushforward along
this inclusion induces an equivalence
\[\mathsf{IndCoh}(\wih{\mathcal{L}}X)^{\tT} \xrightarrow{\simeq} \mathsf{IndCoh}(\mathcal{L}^uX)^{\tT}.\]
For a smooth stack $X$, the left-hand category is identifed with 
$\calD_X\mod_{\bbZ/2}$, as explained in the previous section (and, for a general
stack $X$, it is identified with the $2$-periodicized 
category of (right) crystals, $\Crys(X)_{\bbZ/2}$, by \cite{Preygel}).
The right-hand category, on the other hand, is obtained by
taking the regularized Tate invariants $\tT$ of the category $\IndCoh(\calL^u X)$, whose $S^1$-action factors through
the affinization $S^1 \to B\bbG_a$. In \zcref{prop: BG_a-actions are unipotent} below,
it is shown that the automorphism of a compact object in $\IndCoh(\calL^u X)$ induced by its $S^1$-action
is unipotent. Hence, $\calD_X\mod_{\bbZ/2}$ is obtained as $\tT$
of a \emph{unipotent} categorical $S^1$-action.
In contrast, the automorphism of an object of
$\IndCoh(\calL X)$ induced by its $S^1$-action
is not necessarily \emph{uni}potent---the loop rotation $S^1$-action
on $\IndCoh(\calL X)$ is ``potent." As $\Pot(X)$ is obtained
as $\tT$ of a category with ``potent" categorical $S^1$-action,
its objects are called potent $\calD$-modules.

\subsubsection{Decompletion and genuine equivariance}
\label{sssec: decompletion}
One way to motivate the invention of genuine equivariant spectra
is the following.
Given a compact Lie group $G$ acting on a CW complex $X$, 
one might study $G$-equivariant vector bundles on $X$.
In particular, one might want to consider the Grothendieck group
of such, $K_0(\Vect_G(X))$. In order to do so,
it is natural to consider the complex $K$-theory of the
Borel construction $|X/G|$---the so-called Borel $G$-equivariant
$K$-theory of $X$---but this unfortunately gives
the wrong object. Indeed, while $K_0(\Vect_G(\ast))$ is 
the representation ring of $G$, $R(G)$, $KU^0$ of the classifying space $BG$
is the \emph{completion} of the representation ring $R(G)$
at the augmentation ideal $I_G = \ker(R(G) \to \bbZ)$.

After some reflection, one realizes the issue is that
$KU$, as a spectrum, is determined by descent from its value on CW complexes.
Instead, one should consider an enriched object whose
value is specified for objects like $BG$---orbifold test objects of the form 
$G/H$---by hand.
The appropriate definition for such an object is that
of a genuine equivariant spectrum.
The genuine $G$-equivariant $K$-theory spectrum $KU_G$
is such an object, and its value on the $G$-homogeneous space
$G/H$ is rigged precisely so that $KU^0_G(G/H) \simeq R(H)$.
More generally, the Atiyah--Segal completion theorem
\cite{AtiyahSegal} tells us that the completion of genuine
equivariant $K$-theory at the augmentation ideal $I_G$ is
isomorphic to Borel equivariant $K$-theory.

The category of $\calD$-modules (or crystals) on a stack suffers from the same kind of deficiency.
Given a stack $X$, the category of crystals on $X$ is obtained by descent from
affines:
\[\Crys(X) \xrightarrow{\simeq} \lim_{S \to X} \Crys(S).\]
The result is a category which sees only the formal completion of $\calL X$
at constant loops; indeed (cf. \cite{Preygel, BZNLoops}), for a sufficiently
nice stack, we have
\[\Crys(X)_{\bbZ/2} \simeq \IndCoh(\wih{\calL} X)^{\tT}.\]
Our perspective---the one expounded in \cite{BZPotent}---is that
one should rather consider a category which accesses the full geometry of the
loop space $\calL X$. Namely, one should consider potent $\calD$-modules on $X$.

The analogy between Borel vs. genuine equivariant $K$-theory on the one hand, and ordinary vs.
potent $\calD$-modules on the other, becomes more precise when $X$ is the global quotient of a
derived scheme $U$ by a reductive group $G_{\mathbb{C}}$. Then $\calL(U/G_{\mathbb{C}})$ admits a canonical map $\mu: \calL(U/G_{\mathbb{C}}) \to G_{\mathbb{C}}/\!\!/G_{\mathbb{C}}$
to the GIT quotient $G_{\mathbb{C}}/\!\!/G_{\mathbb{C}}$. In this case, the unipotent loop space $\calL^u(U/G_{\mathbb{C}}) \coloneq \Map(B\bbG_a, U/G_{\mathbb{C}})$
has an alternative expression as the formal completion of $\calL(U/G_{\mathbb{C}})$ along the fiber of $\mu$ over the conjugacy class
of the identity $[1_{G_{\mathbb{C}}}]$. Since $U/G_{\mathbb{C}}$ is a geometric stack, pullback along the canonical map $\wih{\calL}(U/G_{\mathbb{C}}) \hook \calL^u(U/G_{\mathbb{C}})$
induces an equivalence,
\[\IndCoh(\calL^u(U/G_{\mathbb{C}}))^{\tT} \xrightarrow{\simeq} \IndCoh(\wih{\calL}(U/G_{\mathbb{C}}))^{\tT},\]
so we have, in a literal sense (cf. the proof of \zcref{lem:support-is-formal-completion}), 
that objects of $\Crys(U/G_{\mathbb{C}})_{\bbZ/2}$ live on the formal completion at $[1_{G_{\mathbb{C}}}] \in G_{\mathbb{C}}/\!\!/G_{\mathbb{C}}$.
On the other hand, given a compact Lie group $G$,
the complexification of the representation ring $R(G)$---which we idenitified with
$KU_G^0(\ast)$ above---is isomorphic to algebraic global functions 
on the GIT quotient $G_{\bbC}/\!\!/G_{\bbC}$
of the complexification $G_{\bbC}$. Because
$KU^0(BG)$ is a formal completion of $K^0_G(\ast)$,
we say that $K^0_G(\ast)$ is a \emph{decompletion}
of $KU^0(BG)$. In this way, we see that the natural functor 
$\Crys(U/G_{\mathbb{C}})_{\bbZ/2} \to \Pot(U/G_{\mathbb{C}})$
exhibits $\Pot(U/G_{\mathbb{C}})$ as a decompletion of $\Crys(U/G_{\mathbb{C}})_{\bbZ/2}$ along the
map from genuine to Borel equivariant $K$-theory, 
\[KU^0_G(\ast) \simeq \calO(G_{\mathbb{C}}/\!\!/G_{\mathbb{C}}) \to \wih{\calO(G_{\mathbb{C}}/\!\!/G_{\mathbb{C}})_{I_G}} \xrightarrow{\simeq} KU^0(BG).\]

\begin{rem}
    In light of the preceding discussion, one way to understand \zcref{thm:intro kashiwara's lemma for scalloped stacks} 
    is as a genuine equivariant version of Kashiwara's Lemma for $\calD$-modules.
\end{rem}

\begin{rem}
Points of $G/\!\!/G$ are often called ``equivariant parameters."
By working with $\Pot(X)$, we remember information about the
loop space at every equivariant parameter $[\a] \in G/\!\!/G$, rather than just $[1_G]$.
These parameters appear naturally in many different contexts, such as in the study
of 3d mirror symmetry and geometric Langlands, as explained in \cite[\S\S 3, 4]{BZPotent}.
The proof of our main theorem, outlined below in \zcref{ssec: outline of proof},
makes extensive use of the ability to consider the restriction
of potent $\calD$-modules to individual equivariant parameters.
\end{rem}

\begin{rem}
Another sense in which $\Pot(U/G_{\mathbb{C}})$ is a decompletion of $\Crys(U/G_{\mathbb{C}})_{\bbZ/2}$
is furnished by the observation that the functor $\Crys(U/G_{\mathbb{C}})_{\bbZ/2} \to \Pot(U/G_{\mathbb{C}})$
is obtained as the induced map on regularized
Tate invariants given by pushforward along the ``decompletion" $\wih{\calL}(U/G_{\mathbb{C}}) \hook \calL(U/G_{\mathbb{C}})$
of the formal loop space.
Decompletions of the formal loop, however, are not unique.
For example, $\calL X$ and $\bbT_X[-1]$ have 
equivalent formal completions $\wih{\bbT}_X[-1] \xrightarrow{\simeq} \wih{\calL}X$.
In this sense, each of $\calL X$ and $\bbT_X[-1]$ can be seen
as particular decompletions of the formal stack $\wih{\calL}X$.
\end{rem}

\subsubsection{Decompletion and genuine $S^1$-equivariance}
\label{sssec: intro S^1 decompletion}
We are actually performing two decompletions by considering potent $\mathcal{D}$-modules: In the previous section, we explained decompletion with respect to the equivariant parameters of an algebraic group action. In this section, we will explain how the passage from ordinary (homotopy) $S^1$-invariants (resp. Tate invariants) to regularized $S^1$-invariants (resp. $\tT$ invariants) is a decompletion of the action of the topological circle $S^1$.

Let $\mathscr{C}$ be a stable, presentable $\infty$-category with regular $t$-structure\footnote{See \cite[Section 4.3]{Preygel} for what this means.} and equipped with an $S^1$-action. For example, we could take $\mathscr{C}=\mathsf{IndCoh}(X)$, where $X$ is a finite type algebraic stack with affine diagonal and $S^1$-action. Denote the regularization $\mathscr{D}\coloneq \mathscr{R}(\mathscr{C}^{S^1})\coloneq \mathsf{Ind}(\mathsf{Coh}(\mathscr{C}^{S^1}))$. There is a pair of adjoint functions
\begin{align*}
I: \mathscr{C}\rightleftarrows \mathscr{C}^{S^1}:U,	
\end{align*} 
where $U$ is the forgetful functor (forgetting $S^1$-equivariance) and $I$ is the induction functor. If $c\in \mathscr{C}$ is compact, then $I(c)$ is also compact in $\mathscr{C}^{S^1}$, because $U$ preserves colimits. Moreover, the objects $I(c)$ for compact $c\in \mathscr{C}$ generate $\mathscr{C}^{S^1}$, since $U$ is conservative. Thus, this defines a fully faithful and continuous inclusion
\begin{align}\label{eq: inclusion of C^S^1 into R(C^S^1)}
	i: \mathscr{C}^{S^1}\hookrightarrow \mathscr{D}.
\end{align} 
Now, all objects $I(c)$ are $u$-torsion with respect to the $k[u]$-module structure on $\mathscr{C}^{S^1}$, where $u$ is in cohomological degree $2$. Indeed, the $u$-action is induced by tensoring an object $c\in \mathscr{C}^{S^1}$ with the degree 2 map
\begin{align*}
	u:k\to k[2]
\end{align*} 
in $\mathsf{Vect}^{S^1}\simeq k[u]\mod_{u-\text{tors}}$. Since forgetting the equivariance is compatible with the tensor action by $\mathsf{Vect}^{S^1}$ and sends $u$ to a nullhomotopic map in $\mathsf{Vect}$, we find that the $u$-action on $U(\tilde c)$ is nullhomotopic for any $\tilde c\in \mathscr{C}^{S^1}$. Suppose now that $\tilde c = I(c)$ for some $c\in \mathscr{C}$. By adjunction,
\begin{align*}
	\Hom_{\mathscr{C}^{S^1}}(I(c),I(c)[2])\simeq\Hom_{\mathscr{C}}(c,UI(c)[2]),
\end{align*} 
and the map $u_{I(c)}:I(c)\to I(c)[2]$ corresponds to the composition
\begin{align*}
	c\to UI(c)\xrightarrow{U(u_{I(c)})}U(I(c))[2],
\end{align*} 
which we know is nullhomotopic. Hence, $u_{I(c)}$ is nullhomotopic as well. Conversely, if $\tilde c\in \mathscr{C}^{S^1}$ is an arbitrary $S^1$-equivariant object, then
\begin{align*}
	\ker(u:\tilde c\to \tilde c[2])
\end{align*} 
is precisely $IU(\tilde c)$, up to a shift, and these generate the full subcategory of $u$-torsion objects of $\mathscr{D}$:
\begin{align*}
	\mathscr{D}_{u\text{-tors}}\coloneq \ker(\mathscr{D}\to \mathscr{D}[u^{-1}] \simeq \mathscr{C}^{\tT}).
\end{align*} 
Thus, we see that the image of the functor $i$ of \zcref[noname]{eq: inclusion of C^S^1 into R(C^S^1)} is precisely the $u$-torsion subcategory. Moreover, the natural functor of \cite[Proposition 4.3.2]{Preygel} is in fact right adjoint to $i$,\footnote{This can be checked by comparing mapping complexes out of compact objects.} and it exhibits $\mathscr{D}_{u\text{-tors}}\simeq \mathscr{C}^{S^1}$ as a completion of $\mathscr{D}$. In other words, we are decompleting with respect to the equivariant parameter $u$, where $\mathscr{C}^{S^1}$ is supported at $u=0$ and $\mathscr{D}$ is its decompletion away from 0.

\subsubsection{Nil-invariance and Tate fixed points}
    We note that the analogue of \zcref{thm:intro kashiwara's lemma for scalloped stacks}
    is not true for $\IndCoh(\calL(-))$ (in place of $\Pot(-)$), or even for regularized $S^1$-invariants,
    $\Ind(\Coh(\calL(-))^{S^1})$. This furnishes an instance of a recurring pattern 
    in algebraic geometry wherein the Tate fixed points operation 
    produces objects which are ``topological" in nature---that is,
    depending only on underlying reduced structures---from objects which are not.
    We recall the primary example of this pattern below.

    If $X$ is a smooth scheme, $\calO(\calL X)$\footnotemark
    \footnotetext{i.e. derived global functions on the loop space of $X$} 
    may be identified under the Hochschild--Kostant--Rosenberg isomorphism
    with the algebraic de Rham complex $\Omega^{\bullet}_X$. Moreover, this isomorphism
    intertwines the canonical $S^1$-action on
    $\calO(\calL X)$ given by loop rotation with the
    action of the de Rham differential on $\Omega_X^{\bullet}$, letting us further
    identify $\calO(\calL X)^{S^1}$ with $\Omega^{\bullet, \cl}_X$,
    the complex of closed forms on $X$,
    and $\calO(\calL X)^{tS^1}$ with the de Rham cohomology of $X$.
    But while both $\Omega^{\bullet}_X$ and $\Omega^{\bullet, \cl}_X$ are 
    sensitive to nilthickenings of $X$, the de Rham cohomology
    of $X$ is not, with $H^{\bullet}_{\text{dR}}(X)$ depending 
    only on the underlying reduced scheme of $X$.

\subsubsection{Extended example}
In order to anchor the reader's understanding of the text, we
sketch a toy model illustrating the aspects of the theory
found in the preceding discussion.

\begin{exmp}
	\label{ex: extended example}
	Let $G=\mathbb{G} _m$, and consider $X=B\mathbb{G} _m$. 
	Then $\calL X \simeq \mathbb{G} _m \times B\mathbb{G} _m$,
	and $\wih{\mathcal{L}}X \simeq \wih{(\mathbb{G} _m)_{1}}\times B\mathbb{G} _m$.
	Under these identifications, we see that the pushforward 
    \[\QCoh(\wih{\mathcal{L}}X) \hookrightarrow \QCoh(\calL X)\]
    along the inclusion $\wih{\calL}X \hook \calL X$ corresponds to
	the inclusion of graded $k[t,t^{-1}]$-modules supported at $1 \in \bbG_m$ into the category
    of all graded $k[t,t^{-1}]$-modules.

    Applying Tate invariants to the small categories 
	$\mathsf{Coh}(\wih{\mathcal{L}}X)$ and $\mathsf{Coh}(\mathcal{L}X)$ respectively, we obtain (cf. \cite[Example A.13]{BZChenHelmNadler})
	\begin{align*}
		\mathsf{Coh}(\mathcal{L}X)^{tS^1} &\simeq \mathsf{Coh}(k[t,t ^{-1}])\otimes \ku\\
		\mathsf{Coh}(\wih{\mathcal{L}}X)^{tS^1} &\simeq \mathsf{Coh}(\Spf k[\![s]\!])\otimes \ku,
	\end{align*} 
	where $s\coloneq t-1$.\footnote{One reason it is necessary to pass to the ind-completion to obtain $\mathcal{D}_X\mod$ is that $k[\![s]\!]\notin \mathsf{Coh}(\Spf k[\![s]\!])$.} Thus, we see that we have an inclusion
	\begin{align*}
		\mathsf{Coh}(\wih{\mathcal{L}}X)^{tS^1}\hookrightarrow \mathsf{Coh}(\mathcal{L}X)^{tS^1},
	\end{align*} 
	whose image is precisely those modules over $k[t,t ^{-1}]$ which are 
	supported complete with respect to the ideal $(t-1)$, exhibiting a concrete
    instance of ``decompletion," in this instance at the parameter $[1_{\bbG_m}] \in \bbG_m\git\bbG_m= \bbG_m$. 
    
    As explained in \cite{BZChenHelmNadler}, the category $\mathsf{Coh}(\mathcal{L}X)$ admits a 
    direct sum decomposition parametrized by the character lattice $X^*(\mathbb{G}_m)=\mathbb{Z}$:
	\begin{align*}
		\mathsf{Coh}(\mathcal{L}(B\mathbb{G} _m)) \simeq\bigoplus _{\lambda\in \mathbb{Z} }\mathsf{Coh}(\mathbb{G} _m).
	\end{align*} 
    The $S^1$-action on this category is a natural transformation of the identity whose
    component on an object of the $\lambda$-summand $\Coh(B\bbG_m)$ 
    is given by multiplication by $t^\lambda\in k[t,t ^{-1}]$ at the point 
    $t \in \mathbb{G} _m$. Thus, one can think of points in 
    $\mathbb{G} _m$ as equivariant parameters, and taking $\mathsf{Coh}(\hL X)^{tS^1}$ corresponds 
    to completing at the equivariant parameter $1\in\mathbb{G}_m$. If we were instead to consider 
    a torus $T$ of rank $r$ and $X=BT$, then we would be completing at the ideal $(s_1,\ldots,s_r)$, 
    where the $s_i\in \mathcal{O}(T)$ are equivariant parameters.

The inclusion $\wih{\mathcal{L}}X\hookrightarrow \mathcal{L}X$ induces a pullback functor
	\begin{align*}
		\Pot(B\mathbb{G}_m)\coloneq \mathsf{Ind}(\mathsf{Coh}(\mathcal{L}B\mathbb{G}_m)^{tS^1})\to \mathsf{Ind}(\mathsf{Coh}(\wih{\mathcal{L}}B\mathbb{G}_m)^{tS^1})\simeq\mathcal{D}_{B\mathbb{G}_m}\mod_{\bbZ/2}.
	\end{align*}
	Thus, taking endomorphisms of the unit
 in both categories induces a map
	\begin{align*}
		k[t ^{\pm 1}, u^{\pm 1}]\to k[\![s]\!][u^{\pm 1}],
	\end{align*} 
	which exhibits $k[t ^{\pm 1},u^{\pm 1}]$ as a decompletion of $k[\![s]\!][u^{\pm 1}]$. This is analogous to, for example, 
    viewing the genuine $U(1)$-equivariant $K$-theory of a point as a decompletion of the complex $K$-theory of the classifying space $BU(1)\simeq \mathbb{CP}^\infty$.

        Since the functor $\mathsf{Ind}$ preserves colimits, 
    \[
	    \mathsf{Ind}(\mathsf{Coh}(\mathcal{L}X)^{tS^1})\simeq \mathsf{Ind}(\mathsf{Coh}(\mathcal{L}X)^{S^1})\otimes _{k[u]}\ku.
    \]
    In particular, we can consider 
    \begin{align*}
	    \mathsf{Ind}(\mathsf{Coh}(\mathcal{L}X)^{S^1})\simeq k[t,t ^{-1},u]\mod\quad\text{and}\quad
	    \mathsf{IndCoh}(\mathcal{L}X)^{S^1}\simeq k[t,t ^{-1},u]\mod_{u\text{-tors}},
    \end{align*} 
    before inverting $u$. Moreover, by \zcref{sssec: intro S^1 decompletion}, there are adjoint functors
    \begin{align*}
	    k[t,t ^{-1},u]\mod_{u\text{-tors}}\rightleftarrows k[t,t ^{-1},u]\mod,
    \end{align*} 
    which exhibit $k[t,t ^{-1},u]\mod$ as the decompletion of $k[t, t ^{-1},u]\mod_{u\text{-tors}}$ away from $u=0$.
\end{exmp}

\subsection{Outline of proof}
\label{ssec: outline of proof}
Throughout this outline, we assume that all stacks and schemes are of finite type and have affine diagonal. 
\subsubsection{Initial reduction}
Any closed immersion of stacks $Z \hook X$, factors through the formal completion as
\[Z \xrightarrow{i} \wih{X_Z} \xrightarrow{\wih{i}} X.\]
Our proof strategy starts with the
observation that, as for $\IndCoh$ and $\QCoh$,
pushforward along $\wih{i}$ induces an equivalence,
\[\wih{i}_*: \Pot(\wih{X_Z}) \xrightarrow{\simeq} \Pot_Z(X),\]
with adjoint equivalence $\wih{i}^!$ (\zcref{lem:support-is-formal-completion}).
This observation reduces the problem to showing that
the pushforward $i_*: \Pot(Z) \to \Pot(\wih{X_Z})$
is an equivalence---in other words, showing nil-invariance
for $\Pot$---which we do in this paper for $X$ a linearly scalloped stack.

\subsubsection{Reduction to quotient stacks}
In order to show our main result for a linearly scalloped stack $X$,
we reduce to the case when $X$ has the form of a global
quotient $U/G$, where $G$ is a reductive group and $U$ is a derived scheme.

To start this reduction, we show that if $X$ is an algebraic stack, 
then $\Pot(-)$ satisfies Nisnevich descent on the category $\Rep_{/X}$ of stacks $Y$ over $X$
whose structure morphism $Y\to X$ is representable by 
algebraic spaces (\zcref{prop:potent-d-modules-nis-descend}). 
If $X$ is linearly scalloped, this property of $\Pot(-)$ allows us to argue inductively 
using the scallop decomposition of $X$ from the case
where $X$ has the form of a global quotient of a $G$-quasi-projective derived scheme 
(\zcref{thm: Kashiwara's lemma for scalloped stacks}).\footnote{A $G$-(quasi-)projective scheme $U$ is a scheme $U$ with $G$-action such that the map $U /G\to BG$ is (quasi-)projective.}

\subsubsection{$\a$-potent loop spaces}
Now let $X = U/G$, where $G$ is a reductive
group and $U$ is a derived $G$-scheme.
In this case, the loop space $\calL X$ admits a ``Jordan decomposition,"
described in \cite{ChenEquiv,BZChenHelmNadler}, into certain $S^1$-invariant formal substacks,
denoted $\calL^{\alpha} X$, parametrized by closed points $\a \in G/\!\!/G$,\footnotemark
\footnotetext{Here, $G/\!\!/G$ denotes the GIT quotient for the adjoint action of $G$ on itself.}
which correspond to conjugacy classes of semisimple elements in $G$.
For a given closed point $\a \in G/\!\!/G$, the space 
$\calL^{\alpha} X$ is defined as the formal completion of
the fiber of the canonical map $\mu: \calL(U/G) \to G/\!\!/G$ over $\a$\footnotemark
\footnotetext{$\mu$ is defined as the composition of the induced
	map $\calL X = \calL(U/G) \to \calL(BG) \simeq G/G$
	with the affinization $G/G \to G/\!\!/G$.}
and is called the \emph{$\a$-potent loop space} of $X$.\footnote{This
teminology is justified by the fact that taking $\a = [1_G]$ recovers
the 1-potent, i.e. \textit{uni}potent, loop space $\calL^u X$ discussed above. In \cite{ChenEquiv,BZChenHelmNadler}, this space is called the $\alpha$-unipotent loop space.}

The reasons for introducing the spaces $\calL^{\alpha} X$ are:
\begin{enumerate}
	\item \label{first item above}
	$\Coh(\calL^{\alpha}(U/G))^{tS^1}$ is nil-invariant in the sense that
pushforward along the inclusion $i: Z/G \to U/G$ of any closed immersion\footnotemark
\footnotetext{It is shown in \zcref{lem:closed-substack-is-quotient} that
any closed substack of $U/G$ has the form $Z/G$ for a closed subscheme $Z \subset U$.}
induces an equivalence (\zcref{prop: Kashiwara's lemma for unipotent loop spaces}),
\begin{equation}
	\label{eqn: intro eqn}
	(i_*)^{tS^1}: \mathsf{Coh}(\calL^{\alpha}(Z/G))^{tS^1} \xrightarrow{\simeq} \mathsf{Coh}(\calL^{\alpha}(\wih{U_Z}/G))^{tS^1},
\end{equation}
and
	\item \label{second item above}
	the nil-invariance of $\Pot(X) := \IndCoh(\calL(U/G))^{\tT}$ can be deduced from
	knowing the equivalence \zcref[noname]{eqn: intro eqn} for each closed point $\a \in G/\!\!/G$. 
\end{enumerate}

\subsubsection{Essential Surjectivity}
It turns out that in order to show item (\ref{second item above}) above, it is sufficient to show that the pushforward
\begin{align*}
	(i_*)^{tS^1}: \mathsf{Coh}(\mathcal{L}(Z /G))^{tS^1}\to \mathsf{Coh}(\mathcal{L}(\wih{U_Z}/G))^{tS^1}
\end{align*} 
is fully faithful: In \zcref{prop: essential surjectivity in projective case}, we show that the image of $(i_*)^{tS^1}$ generates $\Coh(\wih{U_Z}/G))^{tS^1}$, which, combined with fully faithfulness, proves the theorem after ind-extending (\zcref{thm: Kashiwara's lemma for G-projective quotients}).
The proof of generation reduces to showing the corresponding statement on the hearts
of the $t$-structures on each of the source and target of $(i_*)^{tS^1}$,
which we show by a classical argument involving filtrations defined
by powers of the ideal sheaf of $Z^{\cl}$.

\subsubsection{Leveraging the $\a$-potent loop space}
To show fully faithfulness (\zcref{prop: fully faithfulness in projective case}), we use item (\ref{first item above}). We first reduce, by \zcref{lem: fully faithful iff fully faithful on pure tensors}, to showing that $(i_*)^{tS^1}$ is fully faithful on objects which are in the image of
\begin{align*}
	\mathsf{Coh}(\mathcal{L}(Z /G))^{S^1}\to \mathsf{Coh}(\mathcal{L}(Z /G))^{tS^1}.
\end{align*} 
Denoting two such objects by $\mathcal{F}$ and $\mathcal{G}$, we thus reduce to showing that the following fiber is 0
\begin{align*}
	F\coloneq \ker\left(\Hom_{\mathsf{Coh}(\mathcal{L}(Z /G))^{tS^1}}(\mathcal{F},\mathcal{G})\to \Hom_{\mathsf{Coh}(\mathcal{L}(\wih{U_Z} /G))^{tS^1}}((i_*)^{tS^1}\mathcal{F},(i_*)^{tS^1}\mathcal{G})\right).
\end{align*} 
Our strategy is then to use a theorem of Neeman \cite[Lemma 2.12]{neeman-bockstedt}, which states that a (possibly unbounded) complex of $A$-modules $M$ is 0 if and only if all base changes $M\otimes _Ak(\mathfrak{p})\simeq 0$ for all prime ideals $\mathfrak{p}\in\Spec A$, where $k(\mathfrak{p})$ is the residue field of $\mathfrak{p}$. Since the categories in question are all $\mathcal{O}(G/\!\!/G)$-linear, we aim to show, for $\Spec A \coloneq  G/\!\!/G$, that $F\otimes_A k(\mathfrak{p})\simeq 0$ for all prime ideals $\mathfrak{p}\subset A$.\footnote{In fact, the Tate invariants of an $A$-linear category is $A[u^{\pm 1}]$-linear. We are thus considering $F$ as an $A$-module by restricting the fiber in $A[u^{\pm 1}]\mod$ along $A\to A[u^{\pm 1}]$.} After choosing a prime ideal $\mathfrak{p}\subset A$, we obtain a closed point $\alpha'$ of $\Spec A_{k(\frakp)}\coloneq \Spec (A\otimes _kk(\mathfrak{p}))$ by using the point $A\to k(\mathfrak{p})$. However, the closed point $\alpha'$ is not geometric; hence, it is not necessarily represented by a semisimple element of $G$. In order to apply \zcref{prop: Kashiwara's lemma for unipotent loop spaces}, we require a semisimple representative. Thus, to correct this, we pass to the algebraic closure $\overline{k(\frakp)}$, and consider instead the geometric point $\alpha$ determined by $A\to \overline{k(\frakp)}$. We then base change all our geometric objects to $\overline{k(\frakp)}$. Denoting $(-)_{\overline{k(\mathfrak{p})}}\coloneq -\otimes _k \overline{k(\mathfrak{p})}$, we next show that Hom complexes appearing in the base-changed fiber $F_{\overline{k(\frakp)}}\in A_{\overline{k(\frakp)}}\mod$ can be written as Hom complexes on the base changed stacks in question; in other words, all calculations can now be performed relative to $ \overline{k(\frakp)}$ as the base algebraically closed field.

Now, in order to leverage item (\ref{first item above}), we tensor the base changed fiber with $-\otimes _{A_{\overline{k(\frakp)}}}P\otimes _{A_{\overline{k(\frakp)}}}P^\vee$, where $P$ is the Koszul complex of the maximal ideal of $A_{\overline{k(\frakp)}}$ corresponding to the closed point $\alpha$. The purpose of this maneuver is to use \zcref{lem: tensoring Homs with perfect complex commutes with S^1-invariants}, which allows us to rewrite 
	 $F\otimes _{A_{\overline{k(\frakp)}}}P\otimes _{A_{\overline{k(\frakp)}}}P^\vee$
	 as a kernel of Hom complexes between sheaves on the $\a$-potent loop spaces, since, by an argument similar to that in \zcref{lem:support-is-formal-completion}, an $S^1$-equivariant sheaf supported on the inverse image of $\alpha$ is an $S^1 $-equivariant sheaf on the formal completion at that inverse image, which is precisely the $\a$-potent loop space. Then, we use item (\ref{first item above}) to conclude that 
	 \begin{align*}
	 	F_{\overline{k(\frakp)}}\otimes _{A_{\overline{k(\frakp)}}}P\otimes _{A_{\overline{k(\frakp)}}}P^\vee\simeq 0.
	 \end{align*} 
Finally, after base changing this complex along $A_{\overline{k(\mathfrak{p})}}\to \overline{k(\mathfrak{p})}$, the Koszul differential becomes 0 in the resulting complex; thus, by perfectness of $P$, it is a direct sum of copies of shifts of $F\otimes_A \overline{k(\mathfrak{p})}$. In particular, each of these is 0, so that we have $F\otimes _A\overline{k(\mathfrak{p})}\simeq 0$, hence $F\otimes_A k(\frakp)\simeq 0$, as desired.

\subsection{Structure of paper}

In \zcref{sec:conventions and notation}, we establish the basic conventions
and notation used throughout the paper. In particular, this section contains a
glossary of categories that appear in the work. We encourage the reader to consult
this section first whenever unfamiliar notation appears.

In \zcref{sec:preliminaries}, we collect the foundational definitions and results 
on group actions, loop spaces, and formal completions used throughout the text.
We advise the reader to skip this section at first and return to it as necessary 
while reading the main text. 

In \zcref{sec:potent D-modules}, we develop the theory
of potent $\calD$-modules necessary to formulate and
prove our main result. We also recall various elements
of the six functor formalism on potent $\calD$-modules
constructed in the forthcoming work \cite{ES-6FF}.

In \zcref{sec:nil-invariance of cyclic sheaves on the a-unipotent loop space}, 
we discuss the $\a$-potent loop space $\calL^{\alpha}$, $A$-trivializations,
and the $\alpha$-shifted action on the ordinary unipotent loop space. 
Much of the content in this section are recollections of the material
in \cite[Appendix A]{BZChenHelmNadler}, as well as
compatibility statements between operations like formal completion, fixed points,
and $A$-trivialization. The main result of this section is
the nil-invariance of cyclic sheaves on the $\a$-potent loop space.

In \zcref{sec:the global quotient case} we prove Kashiwara's Lemma 
(\zcref{thm:intro kashiwara's lemma for scalloped stacks}) holds for 
quotient stacks.
We start by proving some algebraic properties of the interaction between $S^1$-invariants and Hom complexes. Using these comparisons,
we prove Kashiwara's Lemma for global quotient stacks.

Finally, in \zcref{sec:the linearly scalloped case}, we use the results of the previous section to
prove Kashiwara's Lemma for all linearly scalloped stacks.

\subsection{Acknowledgements}
JE would first like to thank his advisor, David Nadler, for his continued support and encouragement during this project, 
and Zhiwei Yun for related discussions which lead to the method of proof of \zcref{lem: formal completion is quotient}. 

KS thanks Rok Gregoric for illuminating discussions on the topics of homotopy theory and categorical group actions.

The authors together would like to thank Pavel Safronov, Sanath Devalapurkar, Swapnil Garg, Yuji Okitani, Jeremy Taylor, Jens Eberhardt, Spencer Tamagni, Germ\'an Stefanich, Sam Gunningham, Sam Raskin, and Justin Hilburn for discussions related to the subject of this work.  

The authors thank David Nadler for discussions that led to the creation of this paper. The athors also thank Harrison Chen for
answering numerous questions about his previous work, and for generously sharing work-in-progress. Lastly, we thank David Ben-Zvi for drawing our attention to the Jordan loop space decomposition, as well as for patiently explaining the philosophy of decompletion to us.

\subsubsection{AI use statement}
AI was used for literature search throughout this project, and for locating errors as the project neared completion. It was also used to fix some issues with LaTeX packages, and to help with typesetting.

\section{Conventions and notation}
\label{sec:conventions and notation}

\subsubsection{}
Fix an algebraically closed field $k$ of characteristic zero, once and for all.

\subsubsection{}
Given a functor $F: \scrC \to \scrD$ between $\infty$-categories,
$\im(F)$ denotes the essential image of $F$.

\subsubsection{Hom-objects and mapping spaces}
The reader will encounter $\infty$-categories in
this paper equipped with a variety of different structures,
each of which we use at different times. Here we
establish notation for the various hom-objects and
mapping spaces associated with these structures.

Let $\scrC$ be an $\infty$-category, and let $a,b \in \scrC$
be two objects. We adopt the following notation.
\begin{enumerate}
	\item \label{hom notation 1}
	$\sfMap_{\scrC}(a,b) \in \Grpd_{\infty}$ denotes the mapping
	space from $a$ to $b$. We drop the subscript ``$\scrC$" 
	when the category is clear from context.

	\item \label{hom notation 2}
	If $(\scrC, \otimes)$ is closed symmetric monoidal, then $\sHom_{\scrC}(a,b) \in \scrC$ denotes
	the internal hom-object from $a$ to $b$, i.e. the object determined by,
	\[\sfMap_{\scrC}(c, \sHom_{\scrC}(a,b)) \simeq \Map_{\scrC}(c \otimes a, b),\]
	for $c \in \scrC$. As above, we drop the subscript ``$\scrC$" when the
	category is clear from context.

	\item \label{hom notation 3} 
	If $\scrC$ is a module category over a
	rigid symmetric monoidal category $\scrR$, 
	then $\Hom_{\scrR}(a,b) \in \scrR$
	is the object determined by
	\[\sfMap_{\scrR}(V, \Hom_{\scrR}(a,b)) \simeq \sfMap_{\scrC}(V \otimes a, b),\]
	for $V \in \scrR$, $a, b \in \scrC$.

	When $\scrR$ is $R\mod$ for a ring spectrum $R$, we alternatively write
	$\Hom_R(a,b)$. For example, if $\scrC$ is an object of $\dgcat_k$, 
	then $\Hom_k(a,b)$ is the usual $k$-module spectrum of maps between $a$ and $b$.

	\item
	The space of $n$-morphisms in $\scrC$
	is given by $\sfMap_{\Cat_{\infty}}([n], \scrC)$
	where $[n]$ denotes the $n$-simplex category.
	We denote by $\Ob(\scrC)$, $\Mor(\scrC)$, and
	$\Mor_n(\scrC)$, respectively, the spaces $0$-morphisms (i.e. objects),
	$1$-morphisms, and $n$-morphisms, respectively.

\end{enumerate}

\begin{warning}
Given two objects $X,Y \in \Stk_k$, we will denote the internal mapping object
of $\Stk_k$ between $X$ and $Y$, i.e. the mapping stack, by
$\Map(X, Y)$.
\end{warning}

We will mainly be interested in categories of sheaves in this paper,
for which we have the following alternative notation.

Let $f: X \to S$ be a derived stack over $S$. Pullback along $f$ gives 
a natural symmetric monoidal functor $f^*:\QCoh(S) \to \QCoh(X)$,
making $\QCoh(X)$ into an algebra over $\QCoh(S)$. Composing
$f^*$ with the canonical symmetric monoidal embedding $\Upsilon: \QCoh(X) \to \IndCoh(X)$
obtains a symmetric monoidal functor $\QCoh(S) \to \IndCoh(X)$,
making $\IndCoh(X)$ into an algebra over $\QCoh(S)$, as well.

\begin{enumerate}
	
	\item We use $\sHom_X(-,-)$ to denote the hom-object
	(\labelcref{hom notation 2}) for $(\scrC, \otimes)$ taken to be either
	$(\QCoh(X), \otimes)$ or $(\IndCoh(X), \overset{!}{\otimes})$.
	When there is potential for confusion, we write $\sHom_{\QCoh(X)}(-,-)$
	or $\sHom_{\IndCoh(X)}(-,-)$.
	
	\item We use $\Hom_S(-,-)$ to denote the hom-object of 
	(\labelcref{hom notation 3}) above for $\scrR = \QCoh(S)$
	and $\scrC$ taken to be either $\QCoh(X)$ or $\IndCoh(X)$.
\end{enumerate}

\brem
There is a useful relationship between $\sHom_{\QCoh(X)}(-,-)$
and $\Hom_S(-,-)$:
\[f_*\sHom_{\QCoh(X)}(-,-) \simeq \Hom_S(-,-).\]
When $S$ is smooth, and $f$ is schematic, then
we similarly have
\[f_*\sHom_{\IndCoh(X)}(-,-) \simeq \Hom_S(-,-).\]

Note that if $X$ is taken to be a stack over itself, i.e.
$S = X$ and $f = \id_X$, then $\Hom_S(-,-) = \sHom_X(-,-)$.
\erem

\brem
In the case when the base stack $S$ is an affine scheme $\Spec R$, 
we also write $\Hom_R(-,-)$ in place of $\Hom_S(-,-)$, by slight abuse of
notation. Indeed note that an action of $\QCoh(\Spec R)$ is an action of $R\mod$,
and that $\Gamma(\Hom_{\Spec R}(-,-)) \simeq \Hom_R(-,-)$.
\erem

\subsubsection{Inclusions}
In order to minimize verbiage, we will use
$i_{?}$ as standing, all-purpose notation 
to denote the inclusion of some space (``$?$")
into another, whenever said inclusion is clear
from context.

\subsubsection{Tate invariants}
\label{sssec:Tate invariants}
Let $\scrC$ be a $k$-linear $\infty$-category equipped with an $S^1$-action.
The invariant category $\scrC^{S^1}$ is linear over $k[u]$, where $u$ has cohomological
degree $2$. We set 
\[\scrC^{tS^1} := \scrC^{S^1} \otimes_{k[u]} \ku,\] 
and call the resulting category the \emph{Tate invariants} of $\scrC$.

When $\scrC$ is a large category, the Tate invariants might not be the correct notion
for a given purpose, so we often use the following construction of Preygel instead.
Assume that $\scrC$ is a $k$-linear stable $\infty$-category with a coherent $t$-structure 
(\cite[Definition 4.2.7]{Preygel}). Let $\Coh(\scrC) \subset \scrC$ be the full subcategory
spanned by bounded almost perfect objects of $\scrC$. Then Preygel defines the \emph{$t$Tate invariants}
of $\scrC$ to be the ind-completion of the Tate invariants of the induced $S^1$-action on $\Coh(\scrC)$:
\[\scrC^{\tT} := \Ind(\Coh(\scrC)^{tS^1}).\]

\subsubsection{Stacks}
\label{sssec:stacks conventions}
We use the notions of geometric and ind-geometric stacks of \cite{CautisWilliams}. In particular, if $X$ is a coherent ind-geometric stack with $S^1$-action, it makes sense to write
\begin{align*}
    \mathsf{IndCoh}(X)^{\tT}
\end{align*}
by \cite[Proposition 5.30]{CautisWilliams} and the results of \zcref{ssec: group actions} below. In fact, all our stacks will be coherent ind-geometric stacks locally almost of finite type. By Remark 5.20 of \textit{op. cit.}, for such stacks the categories of ind-coherent sheaves defined in \cite{CautisWilliams} and in \cite{GaitsgoryIndCoh,GR17,GR17II} are the same.

By \textit{closed immersion}, we mean a closed immersion in the sense of \cite[Section 3.7]{CautisWilliams} that is also schematic.

\subsubsection{Symmetric monoidal categories}
If $(\scrC, \otimes)$ is a symmetric monoidal category, we
denote the tensor unit by $\bbone_{\scrC}$.

\subsection{Glossary of categories}
${}$ \\
\begin{outline}
	
	\1 $\Cat_{\infty}$ denotes the $\infty$-category of $\infty$-categories.   
        \2 $\Cat_{\infty}^t$ denotes the category of stable $\infty$-categories equipped with $t$-structures, whose $1$-morphisms are $t$-exact functors,
    and whose higher morphisms are as in $\Cat_{\infty}$.
        \2 $\Cat_{\infty}^{\sm}$ denotes the full subcategory of $\Cat_{\infty}$ spanned by small, idempotent-complete
        $\infty$-categories.
        \2 $\Cat_{\infty, \st}$ denotes the subcategory of $\Cat_{\infty}$ spanned by stable $\infty$-categories,
        whose $1$-morphisms are exact functors, and whose higher morphisms are as in $\Cat_{\infty}$.
            \3 $\Cat_{\infty, \st}^{\sm}$ denotes the subcategory of $\Cat_{\infty}$ given by $\Cat_{\infty, \st} \cap \Cat_{\infty}^{\st}$.
    \1 $\Pr^L_{\st}$ denotes the $\infty$-subcategory of $\Cat_{\infty}$ spanned by
       presentable, stable $\infty$-categories, whose $1$-morphisms are exact colimit-preserving functors, 
       and whose higher morphisms are as in $\Cat_{\infty}$.    
    \2 $(\Pr^L_{\st})^t$ denotes the $\infty$-category of objects in $\Pr^L_{\st}$
            equipped with $t$-structures, whose $1$-morphims are $t$-exact functors,
            and whose higher morphisms are as in $\Pr^L_{\st}$.
	\1 $\DGCat_k$ denotes $\Mod_{\Vect_k}(\Pr^L_{\st})$.
        \2 $\DGCat _k^t$ denotes the $\infty$-category of objects in $\DGCat_k$ equipped with
        $t$-structures, whose $1$-morphisms are $t$-exact functors, and whose higher morphisms are 
        as in $\DGCat_k$. Equivalently, $\DGCat_k^t \simeq \Mod_{(\Vect_k,t)}((\Pr^L_{\st})^t)$,
        where $(\Vect_k, t)$ denotes $\Vect_k$ equipped with its standard $t$-structure.
        There is an obvious forgetful functor $\oblv: \DGCat_k^t \to \DGCat_k$ which
        forgets the $t$-structure.
	\1 $\dgcat_k$ denotes $\Mod_{\Perf k}(\Cat_{\infty, \st}^{\sm})$.
    \2 $S^1\bmod$ denotes $\Fun(BS^1,\dgcat_k)$.
    \2 If $A$ is a commutative dg algebra, then $A\bmod$ denotes $A\mod(\dgcat_k)$.
	\1 $\Ch_k$ denote the dg-category of strict chain complexes over the field $k$. There is
	a localization map $\Ch_k \to \Vect_k$ inverting quasi-isomorphisms.
	\1 $\QCoh(X)$ denotes the $k$-linear (stable cocomplete) 
	$\infty$-category of quasicoherent complexes on a derived prestack $X$. 
	It is equipped with a natural t-structure, whose heart $\QCoh(X)^{\heart}$ 
	is equivalent to the (ordinary) category of quasicoherent sheaves on $\pi_0(X)$.
	\2 $\Perf(X) \subset \QCoh(X)$ is the full subcategory spanned
	by perfect complexes. If $X$ is perfect (e.g. a quasi-compact, quasi-separated
	derived scheme) in the sense of \cite{BZFN}, then $\QCoh(X) = \Ind(\Perf(X))$.
	\2 Let $X$ be a derived stack. Then $\Coh(X) \subset \QCoh(X)$ is the full 
	subcategory spanned by objects with locally bounded, coherent cohomology sheaves. 
	\1 For laft prestacks $X$, $\IndCoh(X)$ is defined in detail in \cite{GaitsgoryIndCoh}, but for $X$ as in \zcref{sssec:stacks conventions}, it is simply the ind-completion of $\Coh(X)$ defined above.
	\1 $\mathsf{Spc}$ denotes the $\infty$-category of $\infty$-groupoids, whose
	objects are alternatively called \emph{spaces} or \emph{animae}.
	\1 Given an $\infty$-category $\scrC$ with a t-structure, $\scrC^+$ will denote $\cup_{n \geq 0} \scrC^{\geq -n}$,
	and $\scrC^-$ will denote $\cup_{n \geq 0} \scrC^{\leq n}$.
	\1 We let $\bCoh_t$ denote the $\infty$-category whose objects are $\infty$-categories
	equipped with coherent $t$-structure, whose $1$-morphisms are continuous and
	left $t$-exact up to a shift functors, and whose higher morphisms are those in $\Cat_{\infty}$.\footnotemark
    \footnotetext{This differs from \cite[Definition 4.3.5(i)]{Preygel}, where the 1-morphisms there are required to be continuous and $t$-exact.}
	\1 We let $\bReg_t\subset \bCoh_t$ denote the full subcategory whose objects are $\infty$-categories with regular $t$-structures.

\end{outline}


\section{Preliminaries}
\label{sec:preliminaries}
\subsection{Group actions}
\label{ssec: group actions}
\begin{defn}
	\label{def: G-category}
Let $G$ be a group object in $\Ani$ (i.e. an $\infty$-group), 
and let $\scrC$ be an $\infty$-category. 
A $G$-action on an object $c \in \scrC$ is defined to be a functor\footnotemark $BG \to \scrC$ 
such that the composition $\ast \to BG \to \scrC$ is the object $c$.
\footnotetext{Here, and throughout the remainder of this paper when appropriate, we
	will consider anima as $\infty$-categories under the canonical inclusion $\Ani \xhookrightarrow{i} \Cat_{\infty}$.}
\end{defn}

In particular, one may take $\scrC$ to be $\Cat_{\infty}$ in the definition above 
to obtain the notion of an $\infty$-category equipped with a $G$-action,
which call a \emph{$G$-category}.
Given an object $c \in \scrC$ equipped with a $G$-action, we often
write $c: BG \to \scrC$ by abuse of notation.

Given a $G$-category $\scrC$, one may consider objects of $\scrC$ equipped with
$G$-actions which are compatible with that on $\scrC$.

\begin{defn}
\label{def: G-equivariant object}
Let $\scrC$ be a $G$-category. The space of \emph{$G$-equivariant} objects
of $\scrC$ is defined to be $\Map_G(\ast, \scrC) := \Map_{\Fun(BG,\Cat_{\infty})}(\ast, \scrC)$, i.e.
the space of natural transformations from the point (equipped with its only possible
$G$-action) to $\scrC: BG \to \Cat_{\infty}$.
\end{defn}

\begin{rem}
An object with $G$-action in a plain $\infty$-category $\scrC$
is equivalently a $G$-equivariant object of the trivial $G$-category $\scrC$.
\end{rem}

\subsubsection{The $G$-action on mapping spaces}
	\label{sssec: G-action on mapping spaces}
	Given a $G$-category $\scrC$, and a pair of $G$-equivariant objects $c, c' \in \scrC$,
	there is a canonical $G$-action on the mapping space 
	$\Map(c,c')$ between the underlying objects of $\scrC$ defined as follows.
	
	Given a $G$-category $\scrC: BG \to \Cat_{\infty}$, there is an induced action
	of $G^{\op}$ on the presheaf category $\calP(\scrC)$ given by the composition
	$BG^{\op} \xrightarrow{\scrC^{\op}} \Cat_{\infty}^{\op} \xrightarrow{\calP} \Cat_{\infty}$,
	where $\calP$ is the functor taking the category $\scrC$ to $\Fun(\scrC^{\op}, \Ani)$.
	This operation is functorial in $\scrC$,
	\[\Fun(BG,\Cat_{\infty})^{\op} \xrightarrow{(-)^{\op}} \Fun(BG^{\op},\Cat_{\infty}^{\op}) \xrightarrow{\calP \circ -} \Fun(BG^{\op}, \Cat_{\infty}),\]
	giving a functor $\wit{\calP}: (\Cat_{\infty}^G)^{\op} \to \Cat_{\infty}^{G^{\op}}$.
	
	Note that, if $\ast$ denotes the trivial $G$-category, $\wit{\calP}(\ast) = \Ani$, equipped with 
	the trivial $G^{\op}$-action. Consider a $G$-equivariant object in $\scrC$,
	$c: \ast \to \scrC$. Then $\wit{\calP}(c^{\op})$ gives a $G^{\op}$-equivariant functor $c^*: \calP(\scrC) \to \Ani$.
	In particular, if $c'$ is $G$-equivariant object of $\scrC$, we see that $c^*$ sends the presheaf represented by $c'$ to a
	$G^{\op}$-equivariant object of $\Ani$ whose underlying space is $\Map(c,c')$. Imitating this
	construction for copresheaves, we find a $G$-equivariant object of $\Ani$ whose 
	underlying space is also $\Map(c,c')$.
	
	Thus, given a pair of $G$-equivariant objects $c,c'$ in $\scrC$, there
	is a canonical $G^{\op} \times G$-action on $\Map(c,c')$. Restricting along
	the diagonal $(-)^{-1} \times \id_G: G \to G^{\op} \times G$, we obtain a canonical $G$-action on
	the mapping space $\Map(c,c')$.

\begin{rem}
\label{rem: mapping space G-action for trivial G-category}
The special case of this construction where $\scrC$ has
the trivial $G$-action is more straightforward.
In this case, a pair of $G$-equivariant objects is
a pair of functors $c: BG \to \scrC$
and $c': BG \to \scrC$. Then the $G^{\op} \times G$-action on the mapping
space $\Map(c(\ast), c'(\ast))$ is given by the composition
\[BG^{\op} \times BG \xrightarrow{c^{\op} \times c'} \scrC^{\op} \times \scrC \xrightarrow{\Map(-,-)} \Ani,\]
and corresponding $G$-action comes from precomposing this functor
with the diagonal map $BG \to BG^{\op} \times BG$.
\end{rem}

\subsubsection{Invariants and coinvariants}

\begin{defn}
	Let $G$ an $\infty$-group and $\scrC$ an $\infty$-category. Suppose that $c: BG \to \scrC$ is
	an object of $\scrC$ equipped with a $G$-action.
	Then the \emph{$G$-invariants} and \emph{$G$-coinvariants}
	of $c$ are defined to be
	\[c^G := \lim_{BG} \scrC \;\;\;\; \text{and} \;\;\;\;  c_G := \colim_{BG} \scrC,\]
	respectively.	
\end{defn}

\begin{rem}
	\label{rem: equivariant objects}
	Note that a $G$-equivariant object of a $G$-category $\scrC$ is
	the same thing as an object of $\scrC^G$. Indeed, for any $\infty$-category $\scrE$, 
	the functor of taking $G$-invariants, $(-)^G: \Fun(BG, \scrE) \to \scrE$
	is right-adjoint to the functor $\triv: \scrE \to \Fun(BG, \scrE)$ that equips
	objects with the trivial $G$-action. Taking $\scrE$ to be $\Cat_{\infty}$,
	we thus obtain,
	\[\scrC^G \simeq \Fun_{\Cat_{\infty}}(\ast, \scrC^G) \simeq \Fun_{\Fun(BG,\Cat_{\infty})}(\triv(\ast), \scrC) = \Fun_G(\ast, \scrC).\]
	In particular, under this equivalence, a morphism in $\scrC^G$ corresponds to
	a map between natural transformations from $\triv(\ast): BG \to \ast \to \Cat_{\infty}$
	to $\scrC: BG \to \Cat_{\infty}$.
\end{rem}

\begin{lem}
\label{lem: morphisms in C^G}
Suppose that $\scrC$ is a $G$-category (i.e. $\scrC: BG \to \Cat_{\infty}$),
and that $c, c' \in \Map_G(\ast,\scrC)$ are $G$-equivariant objects of $\scrC$.
Then there is a natural equivalence
\[\Map_{\scrC^G}(c,c') \simeq \Map_{\scrC}(c,c')^G,\]
where the $G$-action on $\Map_{\scrC}(c,c')$ is the one
defined in \zcref{sssec: G-action on mapping spaces}.
\end{lem}

\begin{proof}
	We begin by outlining some general facts about
	mapping spaces in limit categories.
    Consider a diagram of $\infty$-categories, $D: I \to \Cat_{\infty}$.
    Since mapping spaces commute with limits in the second argument,
    the objects and morphisms of the limit category $\lim_I D$ are
    given by $\lim_{I} \Map_{\Cat_{\infty}}(\ast, D(i))$,
    and $\lim_I \Map_{\Cat_{\infty}}([1], D(i))$, respectively.
    The universal property of the limit furnishes an equivalence of spaces between
    $\Map_{\Ani}(\ast, \lim_I \Ob(D(i)))$ and compatible collections of objects in $D(i)$ (here $i \in I$ denotes
    objects in the index category $I$).
    For a given $d \in \Map(\ast, \lim_I \Ob(D(i)))$, we write (somewhat abusively) $d = \lim d(i)$
    for $d(i)$ the corresponding object in $D(i)$.
	Putting this all together, we see that $\Map_{\lim_I D}(d,d') \simeq \lim_I \Map_{D(i)}(d(i), d'(i))$.

	Now consider the case of a given $G$-category $\scrC: BG \to \Cat_{\infty}$,
	and let $c, c' \in \scrC^G = \lim_{BG} \scrC$. Applying the above discussion to this case,
	we see that 
	\[\Map_{\scrC^G}(c,c')
	\simeq \Map_{\lim_{BG} \scrC}(c, c') \simeq \lim_{BG} \Map_{\scrC(i)}(c(i), c'(i)).\]
	Since $BG$ has only one object $\ast$ (up to isomorphism), this latter space
	is $\lim_{BG} \Map_{\scrC}(c,c')$, where we write $\scrC$, $c$, and $c'$ for
	$\scrC(\ast)$, $c(\ast)$, and $c'(\ast)$ by abuse of notation. Thus, $\Map_{\scrC^G}(c,c')$
	is obtained by taking the limit over a particular diagram $BG \to \Ani$ sending
	the unique object $\ast \in BG$ to $\Map_{\scrC}(c,c')$.
	Unravelling the construction of this diagram, we see that
	the functor $BG \to \Ani$ is precisely the one constructed in \zcref{sssec: G-action on mapping spaces}
	giving the $G$-action on $\Map_{\scrC}(c,c')$. The equivalence now
	follows from the definition of $\Map_{\scrC}(c,c')^G$, as the limit of this diagram.
\end{proof}

\begin{rem}
	Note that there are variants of the above proposition that replace
	$\Cat_{\infty}$ with some suitable category of categories, such as
	$\DGCat_k$.
\end{rem}

We record the following lemma for use later.

\begin{lem}
	\label{lem:fully-faithful-functors-are-ff-after-tate}
	Let $\scrC$ and $\scrD$ be $\Perf k$-linear $\infty$-categories equipped with $S^1$-actions, 
	and let $F: \mathscr{C}\to \mathscr{D}$ be a fully faithful $S^1$-equivariant functor. 
	Then the induced functor $F^{tS^1}: \scrC^{tS^1} \to \scrD^{tS^1}$ is fully faithful.	
\end{lem}

\begin{proof}
	We first show that the induced functor $F^{S^1}: \scrC^{S^1} \to \scrD^{S^1}$ is fully faithful.
	To see this, let $\wit{a},\wit{b} \in \scrC^{S^1}$, and let $a, c \in \scrC$ denote their images
	under the canonical functor $\scrC^{S^1} \to \scrC$. From the discussion in \zcref{sssec: G-action on mapping spaces},
	we know that the mapping complex $\Hom_{\scrC}(a, b)$ and
	$\Hom_{\scrD}(Fa, Fb)$ have canonical $S^1$-actions. Moreover, since
	$F: \scrC \to \scrD$ is $S^1$-equivariant, the natural map 
	$\Hom_{\scrC}(a, b) \to \Hom_{\scrD}(Fa, Fb)$ is $S^1$-equivariant.
	By \zcref{lem: morphisms in C^G}, in order to show that $F^{S^1}$ is fully faithful, it suffices
	to show that the induced morphism $\Hom_{\scrC}(a, b)^{S^1} \to \Hom_{\scrD}(Fa, Fb)^{S^1}$
	is an isomorphism.
	
	Let $H_{a,b}: BS^1 \to \Ani$ and $H_{Fa, Fb}: BS^1 \to \Ani$ denote the $S^1$-actions on
	$\Hom_{\scrC}(a,b)$ and $\Hom_{\scrD}(Fa, Fb)$ respectively.
	We note that the $S^1$-equivariant map $\Hom_{\scrC}(a,b) \to \Hom_{\scrD}(Fa, Fb)$ is
	a natural transformation $F_{a,b}: H_{a,b} \xLongrightarrow{} H_{Fa, Fb}$.
	Such a natural transformation is an object in the category,
	$\Fun([1], \Fun(BS^1, \Ani))$, which under adjunction
	defines an object of $\Fun(BS^1, \Fun([1], \Ani))$,
	which is precisely a $BS^1$-shaped diagram in $\Fun([1],\Ani)$
	The induced map on $S^1$-invariants is given by
	the limit of this diagram,
	\[\lim_{i \in BS^1} \Hom_{\scrC}(a,b)_i \to \Hom_{\scrD}(Fa,Fb)_i.\]
	Note that $\Hom_{\scrC}(a,b)_i \to \Hom_{\scrD}(Fa,Fb)_i$ is an isomorphism
	for each the unique object $i \in BS^1$.
	
	By \cite[\href{https://kerodon.net/tag/02YF}{Tag 02YF}]{kerodon}, the
	natural transformation is an equivalence if and only if $F_{a,b}$ exhibits
	$H_{a,b}$ as a right Kan extension of $H_{Fa, Fb}$ along the identity
	$BS^1 \to BS^1$.  This condition is satisfied precisely when the induced map
	$F_{a,b}(c): H_{a,b}(c) \to H_{Fa, Fb}(c)$ is an equivalence for all objects $c \in BS^1$.
	Since there is a unique point $\ast \in BS^1$, and since $F_{a,b}(\ast): \Hom_{\scrC}(a,b) \to \Hom_{\scrD}(Fa, Fb)$
	is an equivalence by the assumption that $F$ is fully faithful,
	we see that $F_{a,b}$ is actually a natural equivalence of $S^1$-equivariant animae.
	Thus, when we apply $\lim_{BS^1}: \Fun(BS^1, \Ani) \to \Ani$, we obtain
	an equivalence, as desired. Hence, the functor $F^{S^1}: \scrC^{S^1} \to \scrD^{S^1}$
	is fully faithful.
	
	It remains to show that if $\scrA$ and $\scrB$
	are small $k[u]$-linear categories, and if $F: \scrA \to \scrB$ is a fully faithful
	$k[u]$-linear functor, then $F \otimes_{k[u]} \ku$ is fully faithful.
	Since $\scrA$ is small, we have a natural equivalence,
	\[\Hom_{\scrA}(x,y) \otimes_{k[u]} \ku \xrightarrow{\simeq} \Hom_{\scrA \otimes_{k[u]} \ku}(x \otimes \ku, y \otimes \ku),\]
	for all $x,y \in \scrA$, and the same for $\scrB$.
	Thus, the induced map on $\Hom$s given by $F \otimes_{k[u]} \ku$
	is given by base change of the map $\Hom_{\scrA}(x,y) \to \Hom_{\scrB}(Fx, Fy)$,
	which by assumption is an isomorphism, so we are done.
\end{proof}

\begin{lem}\label{lem: fully faithful iff fully faithful on pure tensors}
	Let $\scrC$ and $\scrD$ be small, stable $\infty$-categories
    equipped with $S^1$-actions, and let $F: \mathscr{C}\to \mathscr{D}$ be an exact, $S^1$-equivariant functor. Then 
	\begin{align*}
		F^{tS^1}: \mathscr{C}^{tS^1}\to \mathscr{D}^{tS^1}
	\end{align*} 
	is fully faithful if and only if it is fully faithful on the image of the natural functor
	\begin{align*}
		\theta:\mathscr{C}^{S^1}\to \mathscr{C}^{tS^1}\coloneq \mathscr{C}^{S^1}\otimes _{k[u]}\ku.
	\end{align*}
\end{lem} 
\begin{proof}
	By \cite[Lemma 2.7]{Perry}, the image of $\theta$ generates $\mathscr{C}^{tS^1}$. Thus, it suffices to show that
	\begin{align*}
		\Hom_{\mathscr{C}^{tS^1}}(\theta c,\theta c')\to \Hom_{\mathscr{D}^{tS^1}}(F^{tS^1}(\theta c),F^{tS^1}(\theta c'))
	\end{align*} 
	is an isomorphism for all $c,c'\in \mathscr{C}^{S^1}$. For fixed $c\in \mathscr{C}^{S^1}$, the objects $c'_u\in \mathscr{C}^{tS^1}$ such that
	\begin{align*}
		\Hom(\theta c,c'_u)\to \Hom(F^{tS^1}(\theta c),F^{tS^1}(c'_u))
	\end{align*} 
	is an isomorphism form a thick subcategory. Making the same argument for fixed $c'_u$ extends the isomorphism from all $\theta c$ to all objects $c_u\in\mathscr{C}^{tS^1}$.
\end{proof}

\begin{lem}
	\label{lem: invariants of tensor product}
	Let $R$ be a $k$-algebra.
	Let $\scrC, \scrD \in S^1\mod(\DGCat_R)$ such that $\scrD$ is compactly
	generated (hence dualizable) and equipped with the trivial $S^1$-action. 
	Then $(\scrC \otimes_R \scrD)^{tS^1} \simeq
	\scrC^{tS^1} \otimes_{R[u^{\pm 1}]} \scrD[u^{\pm 1}]$.
\end{lem}

\begin{proof}
	The symmetric monoidal structure on $S^1\mod(\DGCat_R) := \Fun(BS^1, \DGCat_R)$
	is taken to be the pointwise one. Under this structure, the $S^1$-action on $\scrC \otimes_R \scrD$
	is given by the restriction of the diagram 
	\[\Pi: BS^1 \times BS^1 \xrightarrow{\scrC(-) \times \scrD(-)} \DGCat_R \times \DGCat_R \xrightarrow{\otimes_R} \DGCat_R\]
	along the diagonal functor $\Delta: BS^1 \to BS^1 \times BS^1$.
	Since $\scrD$ is equipped with the trivial $S^1$-action, the diagram $\Pi: BS^1 \times BS^1 \to \DGCat_R$
	factors through the map $BS^1 \times BS^1 \to BS^1 \times \ast \simeq BS^1$. Thus, 
	$(\scrC \otimes_R \scrD)^{S^1} \simeq \lim_{i \in BS^1} (\scrC(i) \otimes_R \scrD)$.
	Since $\scrD$ is dualizable, $\scrC \otimes_R \scrD \simeq \Fun_R(\scrD^{\vee}, \scrC)$,
	where $\scrD^{\vee}$ is the dual to $\scrD$.
	Since $\Fun_R(-,-)$ preserves limits in the second argument, we have
	\begin{align*}
		\lim_{BS^1} \Fun_R(\scrD^{\vee}, \scrC) &\simeq \Fun_R(\scrD^{\vee}, \lim_{BS^1} \scrC) \\
		&\simeq \Fun_R(\scrD^{\vee}, \scrC^{S^1}) \\
		&\simeq \scrC^{S^1} \otimes_R \scrD \\
		&\simeq \scrC^{S^1} \otimes_{R[u]} R[u]\mod \otimes_R \scrD \\
		&\simeq \scrC^{S^1} \otimes_{R[u]} \scrD[u],
	\end{align*}
	where we have used that the $R$-linear structure on $\scrC^{S^1}$
	factors through its $R[u]$-linear structure coming from the action of $(R\mod)^{S^1}$.
	
	By definition, $(\scrC \otimes_R \scrD)^{tS^1} := (\scrC \otimes_R \scrD)^{S^1} \otimes_{R[u]} R[u^{\pm 1}]$,
	so from the identifications above, we have,
	\begin{align*}
		(\scrC \otimes_R \scrD)^{tS^1} &\simeq \left(\scrC^{S^1} \otimes_{R[u]} \scrD[u]\right) \otimes_{R[u]} R[u^{\pm 1}] \\
															 &\simeq \left(\scrC^{S^1} \otimes_{R[u]} R[u^{\pm 1}]\right) \otimes_{R[u^{\pm 1}]} \left(\scrD[u] \otimes_{R[u]} R[u^{\pm 1}]\right) \\
															 &\simeq \scrC^{tS^1} \otimes_{R[u^{\pm 1}]} \scrD[u^{\pm 1}],
	\end{align*}
	where we have used the fact that the base change functor $\DGCat_{R[u]} \to \DGCat_{R[u^{\pm 1}]}$
	is symmetric monoidal (i.e. preserves relative tensor products).
\end{proof}

We record the following result of \cite{BZCHN2}.
\begin{thm}[\cite{BZCHN2}]\label{thm: tTate preserves localization sequences}
	Let
	\[
	\begin{tikzcd}
		A\arrow[r,shift left,"F"]&B\arrow[l,shift left,"F^R"]\arrow[r,shift left,"G"]&C\arrow[l,shift left,"G^R"]
	\end{tikzcd}
	\] 
		be a localization sequence in $\Fun(BS^1,\bCoh_t)$ such that the right adjoints $F^R$ and $G^R$ are $t $-bounded. Then the sequence after applying the $\tT$ construction is a localization sequence in $\ku\bmod$.
\end{thm}

\subsubsection{Group actions and $t$-structures}

\begin{defn}
\label{defn: compatible with t-structure}
Let $G\in\Grp(\mathsf{Spc})$, and suppose given a $G$-action on the underlying
$\infty$-category $\scrC$ of an object in $\Cat_{\infty}^t$. 
We say that the group action is \textit{compatible with the $t$-structure} on $\scrC$ if the functor
$\Act: BG \to \Cat_{\infty}$ encoding the $G$-action factors through the forgetful functor
$\oblv^t: \Cat_{\infty}^t \to \Cat_{\infty}$.
\end{defn}

In other words, the group action is compatible with the $t$-structure on $\scrC$
if $\Act(g)$ is a $t$-exact autoequivalence for each morphism $g \in BG$.

\begin{lem}
\label{lem:truncation functors are equivariant}
	Let $(\scrC, \scrC_{\leq 0}) \in \Cat_{\infty}^t$, i.e. a stable $\infty$-category
    $\scrC$ equipped with a $t$-structure. Let $\Act \in \Fun(BG,\Cat_{\infty})$ be a $G$-action on $\scrC$, i.e.
    $\Act(\ast) \simeq \scrC$. If the $G$-action on $\mathscr{C}$ is compatible with its $t$-structure
    in the sense of \zcref{defn: compatible with t-structure}, 
    then the truncation functors $\tau_{\leq n}$ and $\tau_{\ge n}$ are canonically $G$-equivariant for any $n$.
\end{lem}

\begin{proof}
Let $\Fun_{ff}([1], \Cat_{\infty})$ denote the full subcategory of $\Fun([1],\Cat_{\infty})$
spanned by the fully faithful functors. An object of this category may be seen as pair
$(\scrC, \scrC')$ consisting of a category $\scrC$ and a full subcategory $\scrC' \subset \scrC$.
Then $\Cat_{\infty}^t$ may be viewed as the full subcategory
of $\Fun_{ff}([1], \Cat_{\infty})$ spanned by 
pairs $(\scrC, \scrC')$ such that $\scrC$ is stable,
and $\scrC'$ together with the left orthogonal
complement of $\scrC'[-1]$ satisfy the properties of a $t$-structure
(\cite[Definition 1.2.1.1]{HA}) with $\scrC_{\leq 0} \coloneq \scrC'$
and $\scrC_{\geq 0} \coloneq {}^{\perp}\scrC'[-1]$.
From this description, it is clear
that the source and target functors $\ev_0: \Fun([1], \Cat_{\infty}) \to \Cat_{\infty}$
and $\ev_1: \Fun([1], \Cat_{\infty}) \to \Cat_{\infty}$ restricts 
to the functors,
\[\pr_{\leq n}: \Cat_{\infty}^t \to \Cat_{\infty} \; \text{and} \; \oblv: \Cat_{\infty}^t \to \Cat_{\infty}\]
which send a category $\scrC$ equipped with $t$-structure
to the subcategory of $n$-connective objects $\scrC_{\leq n}$ and
$\scrC$, respectively.

Denote the lift of $\Act$ to $\Cat_{\infty}^t$ by $\wit{\Act}: BG \to \Cat_{\infty}^t$, and define
$\Act_{\leq n} \coloneq \pr_{\leq n} \circ {\wit{\Act}}$.
Under the adjunction, 
\[\Fun(\Fun([1], \Cat_{\infty}) \times [1], \Cat_{\infty}) \simeq \Fun([1], \Fun(\Fun([1], \Cat_{\infty}), \Cat_{\infty})),\]
the evaluation functor $\ev: \Fun([1], \Cat_{\infty}) \times [1] \to \Cat_{\infty}$ corresponds
to a natural transformation $\ev_0 \Longrightarrow \ev_1$,
which restricts to a natural transformation $\pr_{\leq n} \xLongrightarrow{} \oblv$.
Whiskering this natural transformation with the functor $\wit{\Act}$,
we obtain a natural transformation $i_{\leq n}: \Act_{\leq n} \to \Act$, whose
component at the unique point $\ast \in BG$ is the inclusion $i_{\leq n}: \scrC_{\leq n} \xhookrightarrow{} \scrC$.
Taking right adjoints, we get a natural transformation $\tau_{\leq n}: \Act \xLongrightarrow{} \Act_{\leq n}$,
whose component at $\ast \in BG$ is the ordinary truncation functor $\tau_{\leq n}$. Thus, we have constructed
a canonical $G$-equivariant structure on $\tau_{\leq n}$ for any $n \in \bbZ$.

The proof for $\tau_{\geq n}$ proceeds in exactly the same manner, starting
with the data of the coconnective subcategory $\scrC_{\geq 0} \subset \scrC$,
rather than $\scrC_{\leq 0}$.
\end{proof}

Given $(\scrC, \scrC_{\leq 0}) \in \Cat_{\infty}^t$,
and a $G$-action on $\scrC$ compatible with its $t$-structure,
we may equip the $G$-invariants $\mathscr{C}^{G}\coloneq \lim_{BG} \Act$ with a canonical $t$-structure
(cf. \cite[Lemma 2.1.9(4)]{HLPreygel}). 
Indeed, let $U: \scrC^G \to \scrC$ be the canonical functor. 
We may define $\mathscr{C}^G_{\le 0}$ and $\mathscr{C}^G_{\ge 0}\subset \mathscr{C}^G$,
respectively, as the full subcategory of $\scrC^G$ spanned by objects $c\in \mathscr{C}^G$ such that $Uc \in \mathscr{C}_{\le 0}$ and $Uc\in \mathscr{C}_{\ge 0}$, respectively. 
Because morphisms in $BG$ act by $t$-exact autoequivalences, these subcategories
determine a $t$-structure on $\scrC^G$ such that the
projection functor $\scrC^G = \lim_{BG} \Act \to \scrC$ is $t$-exact.

In particular, we have that $(\mathscr{C}^{G})^\heart$
is the full subcategory of $\mathscr{C}^G$ spanned by objects $c$ such that $Uc\in \mathscr{C}^{\heart}$. 
On the other hand, one can consider the $G$-invariants $(\mathscr{C}^\heart)^G\coloneq \lim_{BG} \Act^\heart$, 
where $\Act^{\heart}$ is obtained from the pair $(\scrC, \scrC^{\heart})$ in the same manner as $\Act_{\leq n}$ in 
the proof of \zcref{lem:truncation functors are equivariant} above. 
We also obtain a natural transformation $\Act^{\heart} \xLongrightarrow{} \Act$ in the same manner
as we obtained the natural transformation $\Act_{\leq n} \xLongrightarrow{} \Act$ above,
whose component at $\ast \in BG$ is the inclusion $\scrC^{\heart} \xhookrightarrow{} \scrC$.
Taking $G$-invariants, we thus obtain from this natural transformation a functor,
\[(\scrC^{\heart})^G \to \scrC^G,\]
which, from our definition of the $t$-structure on $\mathscr{C}^G$, 
factors canonically through the inclusion $(\mathscr{C}^G)^\heart \xhookrightarrow{} \scrC^G$.

\begin{lem}
	\label{lem: S^1 invariants commutes with heart}
	The canonical functor $(\mathscr{C}^\heart)^G\to (\mathscr{C}^G)^\heart$ is an equivalence.
\end{lem}

\begin{proof}
By definition, we have a $G$-equivariant
equivalence, $\scrC^{\heart} \simeq \scrC_{\leq 0} \times_{\scrC} \scrC_{\geq 0}$.
The operation of taking $G$-invariants, being a limit, commutes with
products, so we have,
\[(\scrC^{\heart})^G \xrightarrow{\simeq} (\scrC_{\leq 0})^G \times_{\scrC^G} (\scrC_{\geq 0})^G.\]
On the other hand, we have commutative diagrams,
\[\begin{tikzcd}
	{\scrC^G} & \scrC && {\scrC^G} & \scrC \\
	{(\scrC_{\leq 0})^G} & {\scrC_{\leq 0}} && {(\scrC_{\geq 0})^G} & {\scrC_{\geq 0}}
	\arrow["U", from=1-1, to=1-2]
	\arrow["U", from=1-4, to=1-5]
	\arrow[hook, from=2-1, to=1-1]
	\arrow["{U_{\leq 0}}", from=2-1, to=2-2]
	\arrow[hook, from=2-2, to=1-2]
	\arrow[hook, from=2-4, to=1-4]
	\arrow["{U_{\geq 0}}", from=2-4, to=2-5]
	\arrow[hook, from=2-5, to=1-5]
\end{tikzcd}\]
implying that $(\scrC_{\leq 0})^G \xrightarrow{\simeq} U^{-1}(\scrC_{\leq 0}) =: (\scrC^G)_{\leq 0}$, and 
$(\scrC_{\geq 0})^G \xrightarrow{\simeq} U^{-1}(\scrC_{\geq 0}) =: (\scrC^G)_{\geq 0}$. Now, by the definition of the heart,
we also have $(\scrC^G)^{\heart} = (\scrC^G)_{\leq 0} \times_{\scrC^G} (\scrC^G)_{\geq 0}$. 
Combining these equivalences,
we thus obtain
\[(\scrC^{\heart})^G \xrightarrow{\simeq} (\scrC^G)^{\heart}.\]
\end{proof}

\subsubsection{Related notions of $G$-action}
Let $G$ be a group object in $\Spc$, and consider
the constant stack $\underline{G} \in \Stk_k$ with value $G$, which by abuse of
notation we sometimes denote by $G$. $\underline{G}$ is a group object in $\Stk_k$,
so the dg-category $\QCoh(G)$ obtains a comonoidal structure under pullback
along the group multiplication maps.
One often sees the following alternative definition of $G$-action on $\scrC$ in the literature:

\begin{defn}
\label{def: actions of group stacks on categories}
	Let $G \in \Grp(\Stk_k)$, and let $\scrC \in \dgCat_k$. Then
	a $G$-action on $\scrC$ is the data of a $\Perf(G)$-comodule structure on $\scrC$.
	Let $\scrD \in \DGCat_k$. Then a $G$-action on $\scrD$ is the data of
	a $\QCoh(G)$-comodule structure on $\scrD$.
\end{defn}

In fact, this alternative definition is equivalent to the one we gave
in \zcref{def: G-category} in the case when $G$ is
the Betti stack associated to group object in $\Spc$.

\begin{prop}
	\label{prop: equivalence of defintions of G-category}
    Let $G \in \Grp(\Spc)$, and let $\scrC \in \QCoh(G)\comod(\DGCat_k)$. Then the underlying dg-category of $\scrC$
	has a canonical $G$-action in the sense of \zcref{def: G-category}. Conversely, if the underlying
	dg-category of $\scrC$ has a $G$-action, then this induces a canonical $\QCoh(G)$-comodule structure
	on $\scrC$.
\end{prop}

\begin{proof}
	First note that the constant stack $\underline{G}$ is the colimit of the constant
	functor $G \to \ast \to \Stk_k$. On the other hand, since $\QCoh: \Stk_k \to \DGCat_k$ takes colimits to limits,
	$\QCoh(\underline{G}) \simeq \lim_G \QCoh(\ast) \simeq \Fun(G, \Vect_k)$.
	
	Consider the functor $L: \Fun(BG, \DGCat_k) \to \DGCat_k$ given by evaluation at the unique point
	in $BG$. Note that $L$ satisfies the conditions of the Barr--Beck--Lurie
	theorem: it is conservative and has both a right adjoint $R: \DGCat_k \to \Fun(BG, \DGCat_k)$ given by cotensoring
	with $G$ ($\scrC \mapsto \Fun(G, \scrC)$) and a left adjoint given by tensoring with $G$. 
	It remains to identify the comonad $L \circ R$ as the operation 
	of tensoring with $\QCoh(\underline{G}) \simeq \Fun(G, \Vect_k)$,
	which we do by proving that $\Fun(G, \scrC) \simeq \Fun(G, \Vect_k) \otimes_k \scrC$. 
	
	Let $K$ be any (small) anima. Recall that $\calP(K^{\op})^{\op} = \Fun(K^{\op})^{\op}$ is the category
	obtained by freely adjoining limits to $K$, meaning it satisfies the universal property,
	\[\Fun(K,\scrC) \simeq \Fun^{R}(\calP(K^{\op})^{\op}, \scrC),\]
	for any cocomplete category $\scrC$. Suppose that $\scrC$ is also presentable. Then,
	by the definition of the Lurie tensor product on presentable $\infty$-categories, we have
	\[\Fun(K,\scrC) \simeq \Fun^{R}(\calP(K^{\op})^{\op}, \scrC) \simeq \calP(K) \otimes \scrC,\]
	where we use that $K \simeq K^{\op}$ and that $\calP(K)$ is presentable.

	In particular, taking $K = G$ and $\scrC = \Vect_k$, we obtain the equivalence $\Fun(G, \Vect_k) \simeq \calP(G) \otimes \Vect_k$.
	To conclude, we observe that, for any $\scrC \in \DGCat_k$, we have
	\begin{align*}
		\calP(G) \otimes \scrC \simeq (\calP(G) \otimes \Vect_k) \otimes_k \scrC &\simeq \Fun(G, \Vect_k) \otimes_k \scrC \\
																															    &\simeq \QCoh(\underline{G}) \otimes_k \scrC.
	\end{align*}
	Thus, the comonad structure on $\DGCat_k$ is the coalgebra structure on $\QCoh(\underline{G})$.
\end{proof}

For any group prestack $\calG \in \Grp(\PreStk)$,
let $\calB\calG^{\bullet}$ denote the bar construction of $\calG$.
If $\calG \in \Grp(\PreStk)$ has the property
that the functor $\bLoc_{\calG}: \QCoh(\calG)\mod \to \ShvCat(\calG)$\footnotemark 
\footnotetext{Here, $\QCoh(\calG)$ is considered as a symmetric monoidal
category using the tensor product on $\QCoh(-)$, rather than the convolution
product induced by the group operation on $\calG$.}
(i.e.
the left adjoint to enhanced global sections $\Gamma^{\mathrm{enh}}_{\calG}$)
is fully faithful, we have
\begin{align}
    \QCoh(\calG)\comod(\DGCat_k)\footnotemark &\simeq \Tot(\QCoh(\calB\calG^{\bullet})\mod) \label{eqn:ShvCat1}\\
                             &\simeq \Tot(\ShvCat(\calB\calG^{\bullet})) \label{eqn:ShvCat2}\\
                             &\simeq \ShvCat(\calB\calG), \label{eqn:ShvCat3}
\end{align}
\footnotetext{Here, we consider $\QCoh(\calG)$
    equipped with the convolution comonoidal structure induced by the
    pullback aling the group operation maps.}
where equivalence \zcref[noname]{eqn:ShvCat1} is the content
of \cite[Corollary 10.1.5]{Gaitsgory1affineness}; equivalence \zcref[noname]{eqn:ShvCat2} is
the content of \cite[Proposition 10.1.3]{Gaitsgory1affineness}; and equivalence \zcref[noname]{eqn:ShvCat3}
follows from the definition of $\ShvCat(\calB\calG)$. Note
that $\calB\calG$ denotes the delooping of $\calG$ in $\PreStk_k$ (as opposed to in $\Stk_k$ if $\calG$
were in $\Stk_k \subset \PreStk_k$).
By \cite[Proposition 2.5]{PascaleffPaviaSibilla}, $\bLoc_{K_B}$ is in
fact fully faithful for any $K \in \Spc$.
In particular, for the constant prestack $\underline{G}$ associated to any $G \in \Grp(\Spc)$,
we obtain an equivalence,
\[\Fun(BG,\DGCat_k) \xrightarrow{\simeq} \ShvCat(\calB\underline{G}),\]
between $G$-actions on dg-categories and quasi-coherent sheaves of categories
on $\calB\underline{G}$.

\begin{rem}
	For our case of interest, $G=S^1$, the stack $BG$ is not $1$-affine itself, so
	we cannot further identify $\Gamma^{\enh}: \ShvCat(BG) \xrightarrow{\simeq} \QCoh(BG)\mod(\DGCat_k)$. 
\end{rem}

\begin{lem}\label{lem:geometric S^1-actions induce topological S^1-actions}
	Let $X\in \mathsf{PreStk}_{\mathrm{laft}}$ be equipped with an action of the group prestack $S^1_B$.
	Then $\mathsf{IndCoh}(X) \in \DGCat_k$ obtains a canonical $S^1$-action, which
    moreover is compatible with its $t$-structure.
\end{lem} 
\begin{proof}
    Recall that the data of an $S^1_B$-action on $X$\footnotemark
    \footnotetext{See \cite[Definition 4.3.3 and Remark 4.3.4]{KhanLectures} for the definition of group actions in an $\infty$-category.}
    is a certain simplicial object
    \[X_{\bullet} := \cdots \;\mathrel{\substack{
\longrightarrow\\[-0.6ex]
\longrightarrow\\[-0.6ex]
\longrightarrow\\[-0.6ex]
\longrightarrow}}\;
S_B^1\times S_B^1\times X
\;\mathrel{\substack{
\longrightarrow\\[-0.6ex]
\longrightarrow\\[-0.6ex]
\longrightarrow}}\;
S_B^1\times X
\rightrightarrows X.\]
    There is a natural map $X_{\bullet} \to (S^1_B)_{\bullet}$ given by the projections
    omitting the $X$ factor at each level. By definition, we have $\calB (S^1_B) \simeq |(S^1_B)_{\bullet}| := \colim_{[n] \in \Delta^{\op}} (S^1_B)_{\bullet}$,
    so applying colimits, we obtain a morphism $|X_{\bullet}| \to \calB(S^1_B)$.
    On the other hand, because $(-)_B$ preserves limits and colimits,
    we have $\calB(S^1_B) \simeq (BS^1)_B$, where $BS^1$ is the delooping in $\Spc$,
    so we obtain a morphism $|X_{\bullet}| \to (BS^1)_B$
    Using the straightening equivalence $\Spc_{/BS^1} \simeq \Fun(BS^1, \Spc)$,
    we obtain the following sequence of equivalences,
    \begin{align*}
        \PreStk_{/(BS^1)_B} = \Fun(\CAlg_k^{\cn}, \Spc)_{/(BS^1)_B} &\simeq \Fun(\CAlg_k^{\cn}, \Spc_{/BS^1}) \\
                                                                    &\simeq \Fun(\CAlg_k^{\cn}, \Fun(BS^1, \Spc)) \\
                                                                    &\simeq \Fun(\CAlg_k^{\cn} \times BS^1, \Spc) \\
                                                                    &\simeq \Fun(BS^1, \Fun(\CAlg_k^{\cn}, \Spc)) \\
                                                                    &= \Fun(BS^1, \PreStk).                                                          
    \end{align*}
Thus, a $S^1_B$-action on $X$ is equivalent to a functor $BS^1 \to \PreStk$ sending
the unique point of $BS^1$ to $X$. Moreover, since $X \in \PreStk_{\laft}$,
this determines a functor $\act_X: BS^1 \to \PreStk_{\laft}$.
Composing $\act_X^{\op}: (BS^1)^{\op} \to \PreStk_{\laft}^{\op}$ with the functor
$\mathsf{IndCoh}:\mathsf{PreStk}^{\text{op}}_{\text{laft}}\to \DGCat_k$,
to obtain a functor $A: BS^1 \simeq (BS^1)^{\op} \to \DGCat_k$ under which the unique
point of $BS^1$ is sent to $\IndCoh(X)$.
Thus, the $S^1_B$-action on $X$ induces an $S^1$-action on $\mathsf{IndCoh}(X)$.

It remains to show that this $S^1$-action is compatible with the $t$-structure in the sense that $A: BS^1 \to \DGCat_k$
factors through $\DGCat_k^t$.\footnotemark
\footnotetext{See \zcref{sec:conventions and notation} for the definition of $\DGCat_k^t$.} 
Under the adjunction, 
\[\begin{tikzcd}
	{B: \Alg_{\bbE_1}(\Cat_{\infty})} & {(\Cat_{\infty})_{\ast/}: \Omega_{\ast}},
	\arrow[shift left, from=1-1, to=1-2]
	\arrow[shift left, from=1-2, to=1-1]
    \end{tikzcd}\]
$A$ corresponds to a functor of monoids (i.e. $\bbE_1$-algebras),
$\wit{A}: S^1 \to \Omega_*(\DGCat_k) \simeq \Aut_{\DGCat_k}(\mathsf{IndCoh}(X))$,
where $\DGCat_k$ is regarded as an object of $(\Cat_{\infty})_{\ast/}$
via the choice of category $\IndCoh(X) \in \DGCat_k$.
Since $S^1 \simeq B\bbZ$, it suffices to show that the image under $\wit{A}$ of the unique point of
$S^1$ is a $t$-exact autoequivalence of $\IndCoh(X)$. But since
$\wit{A}$ is a monoidal map, $\ast \in S^1$ must be sent to the unit $\id_{\IndCoh(X)} \in \Aut_{\DGCat_k}(\mathsf{IndCoh}(X))$,
which is obviously $t$-exact.
\end{proof}

\subsubsection{Algebraic group actions}
While the previous discussion centered around actions of
group objects in $\Spc$, actions of algebraic groups---or, more generally,
actions of group stacks---play a significant role in this paper.
For our purposes, it is convenient to reformulate the definition
of $G$-action given in \zcref{def: actions of group stacks on categories} 
by imitating Preygel's definition for formal group actions
found in \cite{PreygelThesis}.

In order to formulate the definition, we recall the following.
By \cite[Theorem 4.5.3.1]{HA}, for any symmetric monoidal
category $\scrC^{\otimes}$, the functor,
\[p: \Mod(\scrC) \to \CAlg(\scrC)\]
is a coCartesian fibration such that the fiber
over $R \in \CAlg(\scrC)$ is the category $\Mod_R(\scrC)^{\otimes}$, and
the coCartesian lifts of $R \to S$
are given by the extension of scalars functor
$- \otimes_R S: \Mod_R(\scrC) \to \Mod_S(\scrC)$.

\begin{con}
\label{con: coCartesian section}
    For each object $C \in \scrC$, we construct a coCartesian
    section of $p$, denoted by $\underline{C}$, as follows.
    Since $\Mod_{\bbone}(\scrC) \simeq \scrC$ by \cite[Proposition 3.4.2.1]{HA}, 
    $C \in \Mod_{\bbone}(\scrC)$, meaning we obtain an object in the fiber of $p$
    over the initial object $\bbone \in \CAlg(\scrC)$. For each $R \in \CAlg(\scrC)$,
    the unique map $u_R: \bbone \to R$ lifts to the $p$-coCartesian edge $\wit{u_R}: C \to C \otimes_{\bbone} R$. 
    By dint of $p$ being a coCartesian fibration and the universal property
    of $p$-coCartesian edges, these lifts assemble coherently\footnotemark 
    \footnotetext{Indeed, if $\wit{f}: C \otimes_{\bbone} R \to (C \otimes_{\bbone} R) \otimes_R S$
    is the $p$-coCartesian lift of $f: R \to S \in \CAlg(\scrC)$, then the
    universal property of $p$-coCartesian maps tells us that $\wit{u_S}$
    factors as $\wit{f} \circ \wit{u_R}$.}
    into a coCartesian section,
    \[\underline{C}: \CAlg(\scrC) \to \Mod(\scrC).\]
\end{con}

We consider two special cases of this construction:
\begin{enumerate}
    \item Let $\scrC^{\otimes} = \Vect^{\otimes}$.
Then a given object $K \in \Vect_k$ determines a functor $\underline{K}: \CAlg_k \to \Mod(\Vect_k)$.
Restricting $\underline{K}$ along the inclusion $\CAlg_k^{\cn} \xhookrightarrow{} \CAlg_k$,
and composing with the forgetful functor $\Mod(\Vect) \to \Spc$, 
we obtain a prestack,
\begin{equation}
\label{eqn: prestack}
    \underline{K}: \CAlg_k^{\cn} \to \Spc,
\end{equation}
which, by abuse of notation, we also denote by $\underline{K}$.

    \item Let $\scrC^{\otimes} = \DGCat_k^{\otimes}$.
	    A given dg-category $\scrE \in \DGCat_k$ determines a functor \[\underline{\scrE}: \CAlg_k(\DGCat_k) \to \Mod(\DGCat_k).\]
There is a fully faithful embedding $\CAlg_k^{\cn} \to \CAlg(\DGCat_k)$ which assigns to the object $R \in\CAlg_k$
the symmetric monoidal category $\Mod_R^{\otimes} \in \CAlg(\DGCat_k)$. Restricting along this embedding
and composing with the functor $\oblv(-)^{\simeq}:\Mod(\DGCat_k) \to \Spc$ obtains a prestack
\begin{equation}
\label{eqn: prestack from category}
    \underline{\scrE}: \CAlg_k^{\cn} \to \Spc,
\end{equation}
which, by abuse of notation, we also denote by $\underline{\scrE}$.
\end{enumerate}

We also have the following related construction:

\begin{con}
\label{con: Aut prestack}
   Let $\scrC^{\otimes}$ be a symmetric monoidal 
   category, and let $C \in \scrC$. Denote by $\St(p): \CAlg(\scrC) \to \Cat_{\infty}$
   the straightening of the coCartesian fibration $p$ recalled above.
   The coCartesian section of $p$, $\underline{C}$, constructed in \zcref{con: coCartesian section}
   corresponds under the straightening operation to a natural transformation $\St(\underline{C}): \underline{\ast} \to \St(p)$,
   where $\underline{\ast}: \CAlg(\scrC) \to \Cat_{\infty}$ is the constant functor with value the point $\ast \in \Cat_{\infty}$.
   We then define 
   \[\underline{\Aut(C)} := \underline{\ast} \times_{\St(p)} \underline{\ast},\]
   where the product on the right-hand side is taken in $\Fun(\CAlg(\scrC), \Cat_{\infty})$.
   The resulting functor sends the object $A \in \CAlg(\scrC)$ to the category $\Aut_{\Mod_A(\scrC)}\left(\underline{C}(A)\right)$.

As a special case, we take $\scrC = \DGCat_k$, and $C = \scrE \in \DGCat_k$.
Composing with the functor $(-)^{\simeq}: \Cat_{\infty} \to \Spc$ and restricting along
the inclusion $\CAlg_k^{\cn} \xhookrightarrow{} \CAlg(\DGCat_k)$, we obtain a
functor $\res(-)^{\simeq}: \Fun(\CAlg(\DGCat_k), \Cat_{\infty}) \to \PreStk_k$,
which preserves limits. Thus, we have
\[\res(\underline{\Aut(\scrE)})^{\simeq} \xrightarrow{\simeq} \ast_B \times_{\res(\St(p))^{\simeq}} \ast_B,\]
where $\ast_B$ is the Betti prestack with value $\ast \in \Spc$. 
To simplify notation, we make the following conventions:
\begin{enumerate}
    \item $\wit{\DGCat_k}$ denotes the prestack $\res(\St(p))^{\simeq}$, and
    \item $\wit{\scrE}$ denote the natural transformation $\res(\St(\underline{\scrE}))^{\simeq}$.
\end{enumerate}
We may consider $\wit{\DGCat_k}$ as a pointed object of $\PreStk_k$ via
the map $\wit{\scrE}$. In which case, $\ast_B \times_{\res(\St(p))^{\simeq}} \ast_B$
may be rewritten as $\Omega_*(\wit{\DGCat_k})$, where $\Omega_*: (\PreStk_k)_{\ast_B/} \to \Grp(\PreStk)$
is the based loops functor on pointed prestacks.
The value on $R \in \CAlg_k^{\cn}$ of the prestack $\Omega_{\wit{\scrE}}(\wit{\DGCat_k})$
is the category $\Aut_{\DGCat_R}(\scrE \otimes_k R)^{\simeq}$.
This motivates the following definition of the \emph{automorphism prestack} of $\scrE$: 
\begin{equation}
\label{eqn: Aut prestack}
\underline{\calAut}(\scrE) \coloneq \Omega_*(\wit{\DGCat_k}) = \ast_B \times_{\res(\St(p))^{\simeq}} \ast_B.
\end{equation}
We emphasize that $\underline{\calAut}(\scrE)$ is an object of $\Grp(\PreStk_k)$.
\end{con}

The above construction allows us to give an alternative definition of
group prestack action on a category.

\begin{defn}
\label{def: alternative definition of action of group stacks}
    Let $\scrC \in \DGCat_k$, and let $\calG \in \Grp(\PreStk_k)$. 
    A $\calG$-action on $\scrC$ is a map of group prestacks $\calG \xlongrightarrow{} \underline{\calAut}(\scrC)$,
    where the latter object is the group prestack \zcref[noname]{eqn: Aut prestack} considered above.
\end{defn}

\zcref{def: actions of group stacks on categories} and \zcref{def: alternative definition of action of group stacks}
do not agree in general, but it turns out that for $\calG \in \Grp(\PreStk_k)$ such that
$\bLoc_{\calG}$ is a fully faithful---which includes the cases $\calG = S^1_B$
and $\calG = B\mathbb{G}_a$---they do.

\begin{lem}
\label{lem: equivalent defns of action}
    Let $\calG \in \Grp(\PreStk_k)$ 
    such that the functor $\bLoc_{\calG}: \QCoh(\calG)\mod \to \ShvCat(\calG)$
    is fully faithful, and let $\scrC \in \DGCat_k$. 
    Then the data 
    of a $\QCoh(\calG)$-comodule structure on $\scrC$ is equivalent to the
    data of a map $\calG \to \underline{\calAut}(\scrC) \in \Grp(\PreStk_k)$. 
    That is, \zcref{def: actions of group stacks on categories}
    and \zcref{def: alternative definition of action of group stacks} are equivalent definitions
    of $\calG$-action on the category $\scrC$.
\end{lem}

\begin{proof}
    We begin by producing an equivalence,
    \[\Map_{\Grp(\PreStk)}(\calG, \underline{\calAut}(\scrC)) \to \ShvCat(\calB\calG)^{\simeq} \times_{\DGCat_k^{\simeq}} \{\scrC\},\]
    where the product on the right hand side is taken along the functor 
    $\ShvCat(\calB\calG) \to \ShvCat(\Spec k) \simeq \DGCat_k$ given by pullback
    along the unique point $\Spec k\to \calB\calG$.
    
    Since the category $\PreStk$ is an $\infty$-topos,
    we have (\cite[Notation 5.3.6.11]{HA}) an adjunction,
    \[\begin{tikzcd}
	{\calB: \Grp(\PreStk_k)} & {(\PreStk_k)_{\Spec k/}: \Omega_{\ast}},
	\arrow[shift left, from=1-1, to=1-2]
	\arrow[shift left, from=1-2, to=1-1]
    \end{tikzcd}\]
    where we note that $\Spec k$ is the terminal object of $\PreStk_k$.
    As $\underline{\calAut}(\scrC)$ is, by definition, the based loop space $\Omega_*(\wit{\DGCat_k})$,
    we have the following sequence of equivalences.
    \begin{align*}
        \Map_{\Grp(\PreStk_k)}(\calG, \underline{\calAut}(\scrC)) &\simeq \Map_{(\PreStk_k)_{\Spec k/}}\left(\calB\calG, \wit{\DGCat_k}\right) \\
                                                                &\simeq \Map_{\PreStk_k}\left(\calB\calG, \wit{\DGCat_k}\right) \times_{\Map_{\PreStk_k}(\Spec k, \wit{\DGCat_k})} \{\wit{\scrC}\} \\
                                                                &\simeq \left(\lim_{\Spec R \to \calB\calG,\; R \in \CAlg_k^{\cn}} \Map_{(\PreStk_k)}\left(\Spec R, \wit{\DGCat_k}\right)\right) \times_{(\DGCat_k)^{\simeq}} \{\scrC\} \\
                                                                &\simeq \left(\lim_{\Spec R \to \calB\calG,\; R \in \CAlg_k^{\cn}} \wit{\DGCat_k}(R)\right) \times_{(\DGCat_k)^{\simeq}} \scrC \\
                                                                &\simeq \left(\lim_{\Spec R \to \calB\calG,\; R \in \CAlg_k^{\cn}} (\DGCat_R)^{\simeq}\right) \times_{(\DGCat_k)^{\simeq}} \scrC \\
                                                                &\simeq \left(\lim_{\Spec R \to \calB\calG,\; R \in \CAlg_k^{\cn}} (\DGCat_R)\right)^{\simeq} \times_{(\DGCat_k)^{\simeq}} \scrC \\
                                                                &\simeq \ShvCat(\calB\calG)^{\simeq} \times_{(\DGCat_k)^{\simeq}} \scrC,
    \end{align*}
    where the final equivalence follows from the definition of $\ShvCat(-)$ found in \cite{Gaitsgory1affineness}
    and the product in the final line is the limit of the span $\ShvCat(\calB\calG) \xlongrightarrow{\pt^*} \ShvCat(\Spec k) \simeq (\DGCat_k)^{\simeq} \xlongleftarrow{} \{\scrC\}$.
    Now recall from the discussion above that, since the functor $\bLoc_{\calG}$ is
    fully faithful, we have an equivalence $\QCoh(\calG)\comod(\DGCat_k) \simeq \ShvCat(\calB\calG)$.
    Combining this equivalence with the ones in the chain above,
    we obtain an equivalence,
    \[\Map_{\Grp(\PreStk_k)}(\calG, \underline{\calAut}(\scrC)) \simeq \QCoh(\calG)\comod(\DGCat_k)^{\simeq} \times_{(\DGCat_k)^{\simeq}} \{\scrC\},\]
    as desired.
\end{proof}

\begin{rem}
    As noted above, $\bLoc_{\calG}$ is fully faithful when $\calG = G_B$
    for any $G \in \Grp(\Spc)$. In particular, \zcref{lem: equivalent defns of action}
    applies to the group prestack $S^1_B$, which utilize in the proof of the proposition below.
\end{rem}

\begin{prop}
	\label{prop: BG_a-actions are unipotent}
	Let $\scrC \in \DGCat_k$ be a dg-category
	equipped with a $B\bbG_a$-action. Consider $\scrC$ as
	an $S^1$-category via the affinization map $S^1_B \to B\bbG_a$.
	Let $c \in \scrC$, and let $\rho: c \xrightarrow{\simeq} c$
	denote the automorphism induced by this $S^1$-action.
	Then $\rho$ is a unipotent operator\footnotemark if $c$ is compact.
	\footnotetext{i.e. $\rho \simeq 1_c + N_c$ for some
		operator $N_c$ such that $N_c^\ell \simeq 0$ for
		some $\ell \geq 0$.}
\end{prop}

\begin{proof}
	The choice of an object $c \in \scrC$ defines a
	map $c: \Vect_k \to \scrC$. For each $R \in \CAlg_k^{\cn}$,
	there is a natural $k$-linear functor $- \otimes_k R: \scrC \to \scrC \otimes_k R$
	given by extension of scalars. Precomposing this functor with $c$,
	we get a $k$-linear functor which under adjunction corresponds
	to an $R$-linear functor $\Mod_R \simeq \Vect_k \otimes_k R \to \scrC \otimes_k R$, under
	which the generator $R \in \Mod_R$ is sent to the object $c \otimes_k R$.
	Because of the naturality of this construction, these functors
	assemble into a natural transformation $\underline{c}: \underline{\Vect_k} \to \underline{\scrC}$,
	which we view as a morphism in the $\infty$-category $\PreStk_k = \Fun(\CAlg_k^{\cn}, \Spc)$.
	
	On the other hand, a $B\bbG_a$-action on $\scrC$ is the data of a map
	$B\bbG_a \xlongrightarrow{} \underline{\calAut}(\scrC)$ in 
    $\Grp(\PreStk_k)$ (i.e. a natural transformation). Note that, by \cite[Remark 2.1.3.4]{HA},
    we have an equivalence 
    \[\Fun(\CAlg_k^{\cn}, \Grp(\Spc)) \simeq \Grp(\Fun(\CAlg_k^{\cn}, \Spc)).\]
    As noted above, the category $\PreStk$ is an $\infty$-topos, so by \cite[Notation 5.3.6.11]{HA},
    we have an adjunction,
    \[\begin{tikzcd}
	{\calB: \Grp(\PreStk_k)} & {(\PreStk_k)_{\Spec k/}: \Omega_{\ast}}.
	\arrow[shift left, from=1-1, to=1-2]
	\arrow[shift left, from=1-2, to=1-1]
    \end{tikzcd}\]
    Note that, by \cite[Lemme 2.2.5]{ToenChampsAffine},
    we have that $\calB \bbG_a \simeq B\bbG_a$, where the latter
    stack corresponds to the underlying prestack
    of the delooping of $\bbG_a$ as an object
    of $\Grp(\Stk_k)$, rather than of $\Grp(\PreStk_k)$.
    Thus, under the above adjunction, the action $\bbE_1$-algebra map
    $B\bbG_a \to \underline{\calAut}(\scrC)$ corresponds to the
    data of an $\bbE_2$-algebra map
    $\bbG_a \to \Omega_*\underline{\calAut}(\scrC) \simeq \Aut_{\underline{\calAut}(\scrC)}(\underline{\id_{\scrC}})$,
	which, in particular, gives a map $\underline{\rho}: \bbG_a \to \Aut(\underline{c})$, for each $c \in \scrC$.
	This map induces a map of $R$-points,
	\[R \simeq \bbG_a(R) \xrightarrow{\rho_R} \Aut_{\scrC \otimes_k R}(c \otimes_k R).\]
	If $c$ is compact, the canonical map 
	\begin{equation}
		\label{eqn: canonical identification}
		\End_{\scrC}(c) \otimes_k R \xrightarrow{} \End_{\scrC \otimes_k R}(c \otimes_k R)
	\end{equation}
	is an equivalence, so we may regard $\Aut_{\scrC \otimes_k R}(c \otimes_k R)$ as an $\bbE_1$-subspace
	of $\End_{\scrC}(c) \otimes_k R$.
	Now take $R = k[x]$, and consider the element $x \in k[x]$, 
	which we regard as a map $x: \ast \to k[x]$. The composition $\rho_{k[x]}(x)$ gives
	a point of $\End_{\scrC}(c) \otimes_k k[x]$. On the other hand, we have
	\begin{align*}
		\Map_k(k, \End_{\scrC}(c) \otimes_k k[x_1, \ldots x_n]) &\simeq \Map_k\left(k, \End_{\scrC}(c) \otimes_k (\oplus_{i_1, \ldots i_n \geq 0} k \cdot x_1^{i_1}\cdots x_n^{i_n})\right) \\
		&\simeq \Map_k\left(k, \oplus_{i_1, \ldots i_n \geq 0}\left( \End_{\scrC}(c) \cdot x_1^{i_1}\cdots x_n^{i_n}\right)\right) \\
		&\xleftarrow{\simeq} \bigoplus_{i_1, \ldots i_n \geq 0} \Map_k\left(k, \End_{\scrC}(c) \cdot x_1^{i_1}\cdots x_n^{i_n} \right) \\
		&\simeq \bigoplus_{i_1, \ldots i_n \geq 0} \Omega^{\infty}\End_{\scrC}(c) \cdot x_1^{i_1}\cdots x_n^{i_n}
	\end{align*}
	where the notation ``$x_1^{i_1}\cdots x_n^{i_n}$" is a suggestive way of writing
	the index, and the arrow in the third line is an equivalence
	because $k \in \Vect_k$ is compact. Thus, a point of 
	$\End_{\scrC}(c)[x_1, \ldots, x_n] := \End_{\scrC}(c) \otimes_k k[x_1, \ldots, x_n]$ is precisely a sum, or polynomial,
	\[\sum_{(i_1, \ldots, i_n) \in \bbN^n} a_{i_1 \ldots i_n} x_1^{i_1}\cdots x_n^{i_n},\]
	where $a_{i_1 \ldots i_n} \in \Omega^{\infty}(\End_{\scrC}(c))$. In the following, let $u_f$
	denote the polynomial given by $\rho_{k[x_1, \ldots, x_n]}(f)$, for $f \in k[x_1, \ldots, x_n]$.
	
	Because $\varphi: \bbG_a \to \Aut(\underline{c})$
	is a natural transformation, we have a commutative diagram of $\bbE_1$-spaces
	corresponding to any given $k$-algebra map $\varphi: k[x] \to k[t,s]$:
	\[\begin{tikzcd}
		{k[x] \simeq \bbG_a(k[x])} & {\End_{\scrC}(c)[x]} \\
		{k[t,s] \simeq \bbG_a(k[t,s])} & {\End_{\scrC}(c)[t,s]}
		\arrow["{\rho_{k[x]}}", from=1-1, to=1-2]
		\arrow["\varphi"', from=1-1, to=2-1]
		\arrow["{\varphi^*}", from=1-2, to=2-2]
		\arrow["{\rho_{k[t,s]}}", from=2-1, to=2-2]
	\end{tikzcd}\]
	where we have used the fact that $\End(\underline{c})(\varphi)$
	corresponds to the pullback of polynomials $\varphi^*$
	under the identification \zcref[noname]{eqn: canonical identification}.
	Now let $\varphi_s: k[x] \to k[t,s]$, $\varphi_t: k[x] \to k[t,s]$, 
	and $\varphi_{s+t}: k[x] \to k[t,s]$, respectively, denote the maps defined by sending the generator
	$x$ to $s$, $t$, and $s+t$, respectively.
	Then the commutative diagram implies that $\varphi_s^*u_x \simeq \rho_{k[t,s]}(s) = u_s$, 
	$\varphi_t^*u_x \simeq \rho_{k[t,s]}(t) = u_t$,
	and $\varphi_{s+t}^*u_x \simeq \rho_{k[t,s]}(s+t) = u_{s+t}$.
	On the other hand, the fact that $k[t,s] \simeq \bbG_a(k[t,s]) \to \End(c)[t,s]$ is map of $\bbE_1$-spaces
	tells us that $u_{s+t} \simeq u_s \circ u_t$. Thus, we deduce that
	\begin{equation*}
		\varphi_{s+t}^*u_x \simeq \varphi_s^*u_x \circ \varphi_t^*u_x.
	\end{equation*}
	If $u_x = a_0 + a_1x + \cdots + a_n x^n$, where $a_i \in \Omega^{\infty}(\End(c))$,
	then this identity simplifies to the equivalence,
	\begin{equation}
		\label{eqn: key identity}
		\sum_{i=0}^n \sum_{k=0}^{i} \binom{i}{k} a_i s^{i-k}t^k \simeq \sum_{i=0}^n\sum_{j=0}^n a_i a_j s^i t^j,
	\end{equation}
	which imposes a relation on the coefficients $a_i$. To find this relation, we compare the coefficients
	of the components on each side of \zcref[noname]{eqn: key identity} above.
	On the left-hand side, the coefficient of $s^\a t^\b$ is $\binom{\a+\b}{\b} a_{\a+\b}$, while on the right-hand
	side, $s^{\a}t^{\b}$ is $a_{\a} \circ a_{\b}$, so we obtain the relation,
	\[\binom{\a+\b}{\b} a_{\a+\b} \simeq a_{\a}\circ a_{\b},\]
	which can be seen by induction to impose the relation $k^!a_k \simeq a_1^k$
	for any $1 \leq k \leq n$. Thus, $a_1$ is seen to be nilpotent (note that the relation above holds even
	when an endomorphism $a_i$ is zero, as it will be for sufficiently large $i$),
	so $u_x$ has the form $1 + Nx + \cdots N^n x^n$ for some nilpotent operator
	$N \in \Omega^{\infty}(\End(c))$.
	
	The map $\bbG_a \to \underline{\calAut}(c)$
	pulls back along the map $\underline{\bbZ} \to \bbG_a$ to obtain
	the map of prestacks $\underline{\bbZ} \to \underline{\calAut}(c)$.	
	The canonical automorphism $\rho: c \xrightarrow{\simeq} c$
	given by the image of $1 \in \bbZ$ under the 
	induced morphism $\bbZ \to \Aut_{\scrC}(c)$ in the usual
	definition of $S^1$-action on $\scrC$ is obtained
	by looking at the map on $k$-points induced by the map $\underline{\bbZ} \to \underline{\calAut}(c)$.
	By naturality, we obtain a commutative diagram,
	\[\begin{tikzcd}
		{\bbZ \simeq \underline{\bbZ}(k[t])} & {k[x] \simeq \bbG_a(k[t])} & {\End_{\scrC}(c)[x]} \\
		{\bbZ \simeq \underline{\bbZ}(k)} & {k \simeq \bbG_a(k)} & {\End_{\scrC}(c)}
		\arrow["1", from=1-1, to=1-2]
		\arrow["{\id_{\bbZ}}"', from=1-1, to=2-1]
		\arrow["{\rho_{k[t]}}", from=1-2, to=1-3]
		\arrow["{\ev_1}", from=1-2, to=2-2]
		\arrow["{\ev_1^*}", from=1-3, to=2-3]
		\arrow["1", from=2-1, to=2-2]
		\arrow["{\rho_k}", from=2-2, to=2-3],
	\end{tikzcd}\]
	where $1: \bbZ \to k[x]$ and $1: \bbZ \to k$ denote
	the maps given by $1 \mapsto 1$. This diagram tells us
	that the canonical automorphism $\rho: c \to c$ is
	given by $\ev_1^*u_x \simeq 1 + N + \cdots N^n \in \Omega^{\infty}(\End_{\scrC}(c))$.
	The latter is clearly a unipotent endomorphism, as claimed.
\end{proof}

\subsection{Stacks}

\subsubsection{Coherent (ind-)geometric stacks}
We seek to apply \cite[Proposition 5.30]{CautisWilliams} to the stacks which we study in this paper. To do this, we record results below which show that all stacks considered in the paper are coherent ind-geometric in the sense of Definitions 4.1 and 4.3 of \emph{op. cit.}
\begin{lem}\label{lem: formal completions are ind-geometric}
	Let $Z$ be a closed substack of a geometric stack $X$ almost of finite type in the sense of \cite{CautisWilliams}. Then $\wih{X_Z}$ is a laft reasonable ind-geometric stack.
\end{lem} 
\begin{proof}
	Note that $\wih{X_Z}$ is a laft prestack since $X$ (and hence $Z_{\text{red}}$) is laft. We first check that $Z$ is a geometric stack. Since the map $Z\to X$ is a closed immersion, pulling back the geometric atlas $U\to X$ gives a faithfully flat map $U\times _X Z\to Z$. Moreover, the map $U\times _XZ\to U$ is a closed immersion; hence, $U\times _XZ$ is affine by \cite[Proposition 3.22]{CautisWilliams}. This implies that the map $Z\to X$ is affine by \textit{loc. cit.} Now, since formal completions only depend on the underlying reduced stack, we can assume $Z=Z_{\text{red}}$ is reduced and classical. \cite[Chapter 9 Corollary 5.2.3]{GR17II} identifies 
	\begin{align*}
		\wih{X_Z}\simeq\colim_n Z_n,
	\end{align*} 
	where the $Z_n$ are the infinitesimal neighborhoods of $Z$ constructed in \emph{loc. cit.} The construction therein inductively defines the $Z_n$ by maps
	\begin{align*}
		i_{n-1,*}(\Sym_!^n (\mathbb{T}_{Z /X}[1])[-1])\to \mathbb{T}_{Z_{n-1}/X},
	\end{align*} 
	where $i_{n-1}:Z\to Z_{n-1}$ and $\Sym_!$ denotes the symmetric algebra with respect to the tensor product $\overset{!}{\otimes }$ on $\mathsf{IndCoh}(Z)$.\footnote{Here, $\mathbb{T}_{Z /X}$ denotes the ind-coherent tangent complex of \cite[Chapter 1 Section 4.4]{GR17II}.} By Chapter 8 Proposition 5.5.5 of \emph{op. cit.}, such a map corresponds to a square-zero extension of $Z_{n-1}$ by an ideal $\mathcal{I}_n$. At this point, we can already argue that the maps $Z_n\to X$ are affine and the $Z_n$ are geometric. Arguing inductively, each $U_n=Z_n\times _XU$ is a square-zero extension of an affine scheme. Thus, the maps $U_n\to U$ are affine, and they descend along the faithfully flat atlas to $Z_n\to X$. Similarly, the same argument shows that the maps $i_n$ are affine. Now, we claim that the ideal $\mathcal{I}_n$ is precisely 
	\begin{align*}
		\mathcal{I}_n\simeq i_{n-1,*}(\Sym^n(\mathbb{L}_{Z /X}[-1]))\in \mathsf{APerf}(Z_{n-1})^{\le 0}.\footnotemark
	\end{align*}  
	\footnotetext{Here, the symmetric algebra is with respect to the usual tensor product on $\mathsf{QCoh}(Z)$.}
	We will first show how this proves the lemma and then identify $\mathcal{I}_n$ with this almost perfect complex. Note that $\mathcal{I}_n$ is indeed connective, since the affine pushforwards $i_{n-1,*}$ are $t$-exact on $\mathsf{QCoh}$ and the closed immersion $Z\to X$ implies that $\mathbb{L}_{Z /X}[-1]$ is connective. Furthermore, $Z$ is almost of finite type: It is classical, so it admits a classical closed immersion into $X^{\text{cl}}$ which is finite type. The $Z_n$ are also almost of finite type. Indeed, arguing inductively again, a square-zero extension of an almost finite type stack by a connective almost perfect complex is itself almost of finite type. This can be checked flat locally using the geometric atlas, where it reduces to the statement that if
	\begin{align*}
		I\to B'\to B
	\end{align*} 
	is a square-zero extension with $B$ almost of finite type and $ I\in \mathsf{APerf}(B)^{\le 0}$, then $\mathrm H^0(B')$ is a square-zero extension of $\mathrm H^0(B)$ by a finite module. Thus, it is of finite type and Noetherian. We then use the long exact sequence in cohomology to conclude that each $\mathrm H^i(B')$ is finite over $\mathrm H^0(B')$ as desired. Thus, we have a presentation 	
	\begin{align*}
		\wih{X_Z}\simeq\colim_n Z_n,
	\end{align*} 
	with all structure maps almost finitely presented closed immersions.\footnote{Each $i_{n}$ is a square-zero extension by a connective almost perfect complex. Arguing inductively and working flat locally, the square-zero ideal $ I$ fits into a fiber sequence
		\begin{align*}
			I\to B'\to B.
		\end{align*} 
	Since $I$ is connective, we see that $\mathrm H^0(B')\to \mathrm H^0(B)$ is surjective as required.} We can furthermore take all $Z_n$ truncated. Indeed, write $Z_{n,m}\simeq\tau_{\le m}Z_n$. Then we can write
	\begin{align*}
		\wih{X_Z}\simeq\colim_{n,m}Z_{n,m},
	\end{align*} 
	which is a reasonable ind-geometric presentation.

	We will now show the isomorphism 
	\begin{align*}
		\mathcal{I}_n\simeq i_{n-1,*}(\Sym^n(\mathbb{L}_{Z /X}[-1])).
	\end{align*} 
	By \cite[Chapter 8 Proposition 5.5.5]{GR17II}, it is sufficient to show that 
	\begin{align*}
		\mathbb{D}^{\text{Serre}}_{Z_{n-1}}\left(i_{n-1,*}(\Sym^n(\mathbb{L}_{Z /X}[-1]))\right)\simeq i_{n-1,*}(\Sym^n(\mathbb{T}_{Z /X}[1])).
	\end{align*} 
	We then compute
	\begin{align*}
		\mathbb{D}^{\text{Serre}}_Z(\mathbb{L}_{Z/X}[-1])\simeq \mathbb{T}_{Z /X}[1].
	\end{align*} 
	Hence,
	\begin{align*}
		\mathbb{D}^{\text{Serre}}_{Z_{n-1}}(i_{n-1,*}(\Sym^n(\mathbb{L}_{Z /X}[-1])))\simeq &i_{n-1,*}\left( \mathbb{D}^{\text{Serre}}_Z(\Sym^n(\mathbb{L}_{Z /X}[-1])) \right) \\
		\simeq&i_{n-1,*}(\Sym_!^n (\mathbb{T}_{Z /X}[1])),
	\end{align*} 
	where we note that Serre duality exchanges the usual tensor product on $\mathsf{QCoh}(Z)$ with the tensor product $\overset{!}{\otimes }$ on $\mathsf{IndCoh}(Z)$.
\end{proof}

\begin{lem}\label{lem: formal completions of aft algebraic stacks are coherent ind-geometric}
	Suppose that $Z$ is a closed substack of an algebraic stack $X$ almost of finite type and with affine diagonal. Suppose also that $Z\to X$ is almost of finite presentation. Denote by $U$ the open complement to $Z\to X$. Then $\wih{X_Z}$ is a coherent ind-geometric stack, and $Z$, $U$, and $X$ are coherent geometric stacks.
\end{lem}
\begin{proof}
	We seek to apply \zcref{lem: formal completions are ind-geometric}. By assumption $X$ admits a smooth affine atlas $\mathcal{U}\to X$, which we can use both as the geometric atlas and the coherent atlas. Indeed, by \cite[Proposition 7.2.4.31]{HA}, $\mathcal{U} =\Spec A$ is Noetherian, since $\mathcal{U}$ is almost of finite type. This implies that $X$ is coherent by the remark after \cite[Definition 3.3]{CautisWilliams}. This allows us to apply \zcref{lem: formal completions are ind-geometric}. Moreover, pulling back the atlas $\mathcal{U}\to X$ to $Z$ gives a geometric atlas for $Z$. Since $Z\to X$ is a closed immersion almost of finite presentation, this implies that $\mathcal{U}\times _XZ$ is also a Noetherian affine scheme, which gives a faithfully flat (in fact, smooth) atlas of $Z$. Thus, $Z$ is coherent.

	We next claim that the $Z_n$ defined in the proof of \zcref{lem: formal completions are ind-geometric} are also coherent geometric stacks. Pulling back the smooth atlas $\mathcal{U}\to X$ to $Z_ n$ to obtain $\mathcal{U}_n\coloneq Z_n\times _X\mathcal{U}$, these are affine schemes $\mathcal{U}_n=\Spec A_n$ as in the proof of \zcref{lem: formal completions are ind-geometric}. Moreover, $\mathcal{U}_n$ is laft because $\mathcal{U},X,$ and $Z_n$ are. Thus, $\mathcal{U}_n$ is almost of finite type, so $A_n$ is Noetherian. The presentation
	\begin{align*}
		\wih{X_Z}\simeq\colim_{n,m}Z_{n,m}
	\end{align*} 
	then supplies the desired coherent ind-geometric presentation of $\wih{X_Z}$.

	Lastly, we show that $U$ is a coherent geometric stack. Pulling back the smooth atlas $\mathcal{U}\to X$ we obtain a smooth atlas of $U$. Since $U$ is quasi-compact,\footnote{Note that $X$ is quasi-compact, as it is almost of finite type.} this induces a smooth affine atlas $\mathcal V\to U$. Hence, $U$ is geometric. Moreover, $\mathcal{V}\to \mathcal{U}$ is locally of finite presentation, since it is in particular an open immersion. This implies that $\mathcal{V}$ is laft and affine, hence almost of finite type. Thus, it is the spectrum of a Noetherian $k$-algebra, as required.
\end{proof}

\begin{cor}\label{cor: loop spaces of formal completions are ind-geometric}
	In the setup of \zcref{lem: formal completions of aft algebraic stacks are coherent ind-geometric}, $\mathcal{L}(\wih{X}_Z)$ is a laft coherent ind-geometric stack, and $\mathcal{L}X$, $\mathcal{L}Z$, and $\mathcal{L}U$ are coherent geometric stacks.
\end{cor} 
\begin{proof}
	If $X$ is a quasi-compact algebraic stack with affine diagonal, then so is $\mathcal{L}X$. Similarly, the map $\mathcal{L}Z\to \mathcal{L}X$ is a closed immersion by \zcref{lem: loops of closed immersion is closed immersion}. Since $Z$ and $X$ are almost of finite type, so are $\mathcal{L}Z$ and $\mathcal{L}X$. Thus, $\mathcal{L}Z\to \mathcal{L}X$ is almost finitely presented, which allows us to apply \zcref{lem: formal completions of aft algebraic stacks are coherent ind-geometric}. The argument for $\mathcal{L}U$ is as in the proof of \zcref{lem: formal completions of aft algebraic stacks are coherent ind-geometric}.
\end{proof}

\begin{lem}\label{lem: ind-noetherian stacks have ind-noetherian base changes}
	Let $X$ be a coherent ind-geometric stack in the sense of \cite[Definitions 4.1 and 4.2]{CautisWilliams}. Let $ A$ be a Noetherian $k$-algebra, 
    and let $X_A \coloneq X\times \Spec A$ denote the base change. If the coherent presentation $X\simeq\colim_n X_n$ has all 
    $X_n$ equipped with a faithfully flat atlas $\Spec A_n\to X_n$ such that $A_n$ are almost of finite type, then $X_A$ is coherent ind-geometric.
\end{lem} 
\begin{proof}
	Since the colimit $\colim_n X_n$ is a colimit in the category $\wih{\mathsf{Stk}}_k$ of convergent stacks, finite limits commute with this colimit by \cite[Example 7.3.4.7]{HTT} (cf. also the proof of \cite[Proposition 4.21]{CautisWilliams}). In particular, this implies that
	\begin{align*}
		X_A\simeq\colim_n (X_n\times \Spec A).
	\end{align*} 
	Moreover, if $U_n=\Spec A_n\to X_n$ is a faithfully flat atlas of $X_n$ with $U_n$ almost of finite type, then $ U_n\times \Spec A\to X_n\times \Spec A$ is also an atlas exhibiting $X_n\times \Spec A$ as locally Noetherian in the sense of \cite[Definition 3.3]{CautisWilliams}. Indeed, the product $U_n\times \Spec A$ is Noetherian, because the $U_n$ are almost of finite type and $A$ is Noetherian. Thus, we conclude that $X_A$ has a reasonable presentation by locally Noetherian geometric stacks; in particular, this is a coherent presentation by the remark after \emph{loc. cit.}
\end{proof}

\begin{cor}\label{cor: loop spaces have coherent ind-geometric base changes}
	Let $A$ be a Noetherian $k$-algebra. In the setup of \zcref{cor: loop spaces of formal completions are ind-geometric}, the base change along
    $\Spec A \to \Spec k$ of each of the stacks $\mathcal{L}U$, $\mathcal{L}Z$, $\mathcal{L}X$, and $\mathcal{L}(\wih{X_Z})$ is coherent ind-geometric.
\end{cor}

\begin{proof}
	All of $\mathcal{L}U,\mathcal{L}Z$ and $\mathcal{L}X$ are locally Noetherian (in fact, Noetherian) geometric stacks. Hence, \zcref{lem: ind-noetherian stacks have ind-noetherian base changes} applies to them. For $\mathcal{L}(\wih{X_Z})$, we can again apply \emph{loc. cit.} Indeed, observe that all stacks in the presentation constructed in \zcref{cor: loop spaces of formal completions are ind-geometric} and \zcref{lem: formal completions of aft algebraic stacks are coherent ind-geometric} are locally Noetherian.
\end{proof}

\subsubsection{Properties of $!$-pullback and $*$-pushforward functors}
\begin{lem}
    \label{lem: etale morphisms have t-exact !-pullbacks}
	Let $f:X\to Y$ be an \'etale morphism of laft Artin stacks. Then
    $f^!: \mathsf{IndCoh}(Y)\to \mathsf{IndCoh}(X)$ is $t$-exact.
\end{lem} 
\begin{proof}
    Since $f$ is \'etale, $f^! \simeq f^{*,\IndCoh}$, so it suffices to show that $f^{*,\IndCoh}$
    is $t$-exact. By \cite[Proposition 11.7.5(b)]{GaitsgoryIndCoh}, the functors
    $\Psi_Y: \IndCoh(Y) \to \QCoh(Y)$ and $\Psi_X: \IndCoh(X) \to \QCoh(X)$
    are $t$-exact. Moreover, they intertwine $f^*$ and $f^{*, \IndCoh}$ in the sense
    that $\Psi_X \circ f^{*, \IndCoh} \simeq f^* \circ \Psi_Y$. We claim that, in order
    to show that $f^{*, \IndCoh}$ is $t$-exact, it suffices to show that
    $f^*: \QCoh(Y) \to \QCoh(X)$ is t-exact. 
    
    Indeed, suppose that $f^*$ is 
    $t$-exact, let $\F \in \IndCoh(Y)^{\leq 0}$. 
    Then $f^* \Psi_Y(\F) \simeq \Psi_X f^{*,\IndCoh}(\F) \in \QCoh(X)^{\leq 0}$.
    Because $\Psi_X$ is $t$-exact, we have 
    \[\Psi_X(\tau^{\geq 1}f^{*,\IndCoh}\calF) \simeq
    \tau^{\geq 1}\Psi_X(f^{*,\IndCoh}(\calF).\] 
    Since $\Psi_X(f^{*,\IndCoh}(\calF)$ is
    coconnective, $\tau^{\geq 1}\Psi_X(f^{*,\IndCoh}(\calF) \simeq 0$, meaning $\Psi_X(\tau^{\geq 1}f^{*,\IndCoh}\F)
    \simeq 0$. On the other hand, by \cite[Proposition 11.7.5(b)]{GaitsgoryIndCoh}
    again, $\Psi_X: \IndCoh(X)^+ \xrightarrow{\simeq} \QCoh(X)^+$
    is an equivalence. Thus, $\tau^{\geq 1}f^{*, \IndCoh}\F \simeq 0$, implying that
    $f^{*,\IndCoh}\F \in \IndCoh(X)^{\leq 0}$.

    On the other hand, because $f$ is eventually coconnective,
    $f^{*,\IndCoh}$ preserves the subcategory of eventually
    coconnective, or bounded below, objects. Thus, if $\F \in \IndCoh(Y)^{\geq 0}
    \subset \IndCoh(Y)^+$, we see $f^{*,\IndCoh}\F \in \IndCoh(X)^+$. 
    Now note that 
    \[\Psi_Xf^{*,\IndCoh}\F \simeq f^*\Psi_Y(\F) \in \QCoh(X)^{\geq 0}\]
    by the $t$-exactness of $\Psi_Y$ and $f^*$. Because $\Psi_X$ restricts 
    to a $t$-exact equivalence $\IndCoh(X)^+ \xrightarrow{\simeq} \QCoh(X)^+$,
    we therefore have that $f^{*,\IndCoh}\F \in \IndCoh(X)^{\geq 0}$.

    We conclude simply by noting that $f^*: \QCoh(Y) \to \QCoh(X)$
    is $t$-exact because $f$ is \'etale (\cite[Chapter 3 Proposition 1.5.4(b)]{GR17}).
\end{proof} 
\begin{lem}\label{lem: laft qc algebraic stacks with affine diagonal are QCA}
	Let $Y$ be a laft, quasi-compact algebraic stack with affine diagonal. Then $Y$ is a QCA stack in the sense of \cite[Definition 1.1.8]{DrinfeldGaitsgory}.
\end{lem} 
\begin{proof}
	Geometric stabilizers are affine as they are base changes of the affine diagonal. The classical inertia stack is finitely presented because $Y^{\text{cl}}$ is locally of finite type with affine diagonal. Thus, $(\mathcal{L}Y)^{\text{cl}}\to Y^{\text{cl}}$ is a base change of the diagonal of $Y^{\text{cl}}$. This diagonal is qcqs; hence, it is of finite presentation.
\end{proof} 
\begin{lem}\label{lem: stacks in Rep are QCA}
	Let $f:X\to Y$ be a finitely presented morphism in $\mathsf{Stk}_{\mathrm{laft}}$ which is representable by algebraic spaces, and suppose that $Y$ is a laft, quasi-compact algebraic stack with affine diagonal. If $X$ has affine diagonal, then it is QCA in the sense of \cite[Definition 1.1.8]{DrinfeldGaitsgory}. 
\end{lem} 
\begin{proof}
	We will show that $X$ is a laft quasi-compact algebraic stack and then use \zcref{lem: laft qc algebraic stacks with affine diagonal are QCA} to conclude. Using a smooth affine atlas $U\to Y$, we obtain an algebraic space $V=X\times _YU$ with $V\to U$ finitely presented. In particular, it is quasi-compact; hence, $V$ is quasi-compact. Choosing an \'etale surjection $W\to V$ with $W$ affine, we obtain a smooth affine atlas of $X$.
\end{proof}
\begin{cor}\label{cor: smooth pullbacks preserve Coh}
    In the setup of \zcref{lem: stacks in Rep are QCA}, suppose that $f$ is furthermore smooth. Then 
    \begin{align*}
		f^!: \mathsf{IndCoh}(Y)\to \mathsf{IndCoh}(X)
	\end{align*} 
	restricts to a functor
	\begin{align*}
		f^!|_{\mathsf{Coh}(Y)}: \mathsf{Coh}(Y)\to \mathsf{Coh}(X).
	\end{align*} 
    In particular, $f^!$ preserves compact objects.
\end{cor} 
\begin{proof}
	By \zcref{lem: stacks in Rep are QCA}, we can use the characterization of coherent sheaves on QCA stacks given in \cite[Section 3.3.1]{DrinfeldGaitsgory}. In particular, \emph{loc. cit.} characterizes $\mathsf{Coh}(Y)\subset \mathsf{IndCoh}(Y)$ as those sheaves $\mathcal{F}$ such that for any smooth $y:S\to Y$ with $S$ affine, $y^!\mathcal{F}\in \mathsf{Coh}(S)$. Thus, suppose given a smooth map $x:T\to X$ with $T$ an affine scheme, and let $\mathcal{F}\in \Coh(Y)$. Since $f\circ x $ is smooth, we find $(f\circ x)^!\mathcal{F}\simeq x^! f^!\mathcal{F}\in \Coh(T)$ as desired.
\end{proof}

\begin{lem}\label{lem: open immersion pushforward is fully faithful}
	Let $X$ be a laft algebraic stack, and let $j:U\hookrightarrow X$ be an almost finitely presented open immersion. Then the pushforward
	\begin{align*}
		j_*: \mathsf{IndCoh}(U)\to \mathsf{IndCoh}(X)
	\end{align*} 
	is fully faithful.
\end{lem} 
\begin{proof}
	For derived schemes, this is \cite[Lemma 4.1.1]{GaitsgoryIndCoh}. Let $S\to X$ be a smooth atlas. Consider the Cartesian diagram
	\[
	\begin{tikzcd}
		U\times _XS\arrow[r,"j'"]\arrow[d,"p'"]&S\arrow[d,"p"]\\
		U\arrow[r,"j"]&X
	\end{tikzcd}.
	\] 
Note that $j'$ is also smooth and surjective; hence, $U$ is an algebraic stack as well. We would like to show that the counit
\begin{align*}
	j^*j_*\to \id
\end{align*} 
is an isomorphism. We can pull it back along $p'^!$ to obtain
\begin{align*}
	p'^!j^!j_*\simeq j'^!p^!j_*\xleftarrow{\simeq}j'^!j'_*p'^!
\end{align*} 
Thus, the pulled-back counit becomes
\begin{align*}
	j'^!j'_*p'^!\xrightarrow{\simeq}p'^!.
\end{align*} 
Moreover, tracing through the definitions, one finds that this map is precisely the counit of the $(j'^*,j'_*)$-adjunction precomposed with $p'^!$. Since $p'^!$ is conservative, we conclude. Indeed, conservativity follows from descent along $p'$.
\end{proof}

\begin{lem}\label{lem: open immersion pushfoward is t-bounded}
	Let $X$ be an algebraic stack almost of finite type and with affine diagonal, and let $j:U\hookrightarrow X$ an open substack. Then the pushforward
	\begin{align*}
		j_*: \mathsf{IndCoh}(U)\to \mathsf{IndCoh}(X)
	\end{align*} 
	is $t$-bounded.
\end{lem}

\begin{proof}
	Consider the Cartesian diagram
	\[
	\begin{tikzcd}
		U\times _XS\arrow[r,"j'"]\arrow[d,"p'"]&S\arrow[d,"p"]\\
		U\arrow[r,"j"]&X
	\end{tikzcd},
	\] 
	where $S\to X$ is a smooth affine atlas of $X$. By \cite[Proposition 11.7.5]{GaitsgoryIndCoh} and Corollary 7.5.4 of \emph{op. cit.}, the functors $p^ !$ and $p'^!$ are $t$-exact up to a shift; hence, they are $t$-bounded. Thus, base change along the diagram implies that it is enough to check that $j'_*$ is $t$-bounded. Since open immersions are schematic, $U\times _X S\to S$ is an open immersion of schemes; thus, we have reduced to the case when $j:U\hookrightarrow X$ is an open immersion of schemes with $X=\Spec A$ affine. Note that $j$ is left $t$-exact by \cite[Chapter 4 Lemma 4.1.4]{GR17}, as it is right adjoint to a $t$-exact functor. To obtain the upper bound, we will first show that quasi-coherent pushforward $j_*^{\QCoh}$ is $t$-bounded. Write the scheme
	\begin{align*}
		U=\bigcup_{i=1}^r D(f_i)\subset \Spec A,
	\end{align*} 
	where $f_i\in H^0(A)$.\footnote{Clearly, we can write $U^{\text{cl}}$ in this way, with the $D(f_i)^{\text{cl}}$. Then, pulling back along $D(f_i)\to X$, we obtain the desired presentation of $U$, noting that the underlying topological spaces of $U^{\mathrm{cl}}$ and of $U$ are the same.} We will induct on $r$. In the case $r=1$, $U$ is affine and quasi-coherent pushforward $j_*^{\QCoh}$ is just restriction of modules along a localization, which is $t$-bounded. For $r>1$, write 
	\begin{align*}
		U=T_1\cup T_2,\qquad\qquad
		T_1= \bigcup_{i=1}^{r-1}D(f_i), \qquad\qquad
		T_2= D(f_r) .
	\end{align*} 
	Write the inclusions $j_i:T_i\to X$ and $j_{12}:T_1\cap T_2\coloneq T_1\times_X T_2 \to X$.\footnote{Since open immersions are monomorphisms, $T_1\times_X T_2\simeq T_1\times_U T_2$.} Quasi-coherent pushforward can be evaluated as the fiber in the Mayer-Vietoris triangle
	\begin{align*}
		j^{\QCoh}_*(\mathcal{F})\to j^{\QCoh}_{1*}(\mathcal{F}|_{T_1})\oplus j^{\QCoh}_{2*}(\mathcal{F}|_{T_2})\to j^{\QCoh}_{12*}(\mathcal{F}|_{T_1\cap T_2}).
	\end{align*} 
	Induction gives that the middle term is bounded. For the final term, we have
	\begin{align*}
		T_1\cap T_2\simeq \bigcup_{i=1}^{r-1}D(f_if_r),
	\end{align*} 
	so the final term is bounded as well. We conclude by rotating the Mayer-Vietoris triangle to obtain
	\begin{align*}
		j^{\QCoh}_{12*}(\mathcal{F}|_{T_1\cap T_2})[-1]\to j^{\QCoh}_*(\mathcal{F})\to j^{\QCoh}_{1*}(\mathcal{F}|_{T_1})\oplus j^{\QCoh}_{2*}(\mathcal{F}|_{T_2}).
	\end{align*} 
Thus, we see that the quasi-coherent pushforward $j_*^{\QCoh}$ is $t$-bounded on $\mathsf{QCoh}$. To extend to $\mathsf{IndCoh}$, we use that both $\Psi_X$ and $\Psi_U$ are $t$-exact equivalences on the bounded below categories by \cite[Proposition 11.7.5]{GaitsgoryIndCoh}. Hence,
	\begin{align*}
	\Psi_Xj_*\simeq j_*^{\QCoh}\Psi_U.
\end{align*} 
Since both $j_*$ and $j_*^{\QCoh}$ are left $t$-bounded, we conclude.
\end{proof}
\begin{lem}\label{lem: pullback along formal completion inclusion is t-bounded}
Let $X$ be an algebraic stack almost of finite type and with affine diagonal. Let $Z\hookrightarrow X$ be a closed substack. Then the $!$-pullback
\begin{align*}
	\wih{i}^!: \mathsf{IndCoh}(X)\to \mathsf{IndCoh}(\wih{X_Z})
\end{align*} 
along the inclusion $\wih{i}:\wih{X_Z}\to X$ is $t$-bounded.
\end{lem} 
\begin{proof}
	As in \zcref{lem: open immersion pushfoward is t-bounded}, we reduce to the case when $X$ is a scheme by pulling back along the atlas $S\to X$. Then, we apply Proposition 7.4.5 and Lemma 7.4.8 of \cite{GaitsgoryIndCoh} to conclude, noting that the inclusion
	\begin{align*}
		\mathsf{IndCoh}_Z(X)\subset \mathsf{IndCoh}(X)
	\end{align*} 
	is $t$-exact.
\end{proof}

\subsubsection{Quotient stacks}

Here we collect various results on the structure
of global quotient stacks (i.e. quotients of
derived schemes by actions of group schemes), starting with the
facts that open and closed immersions 
of quotient stacks are themselves quotient stacks.

\begin{lem}\label{lem:closed-substack-is-quotient}
	Let $ \mathcal{Z}$ be a closed substack of a stack $X /G$. Then there is an isomorphism $Z /G\xrightarrow{\cong}\mathcal{Z}$ for a $G$-invariant closed subscheme $Z\hookrightarrow X$.
\end{lem} 
\begin{proof}
Since closed immersions are stable under base change, the map $Z\coloneq \mathcal{Z}\times _{X /G}X\to X$ is a closed immersion. Thus, $Z$ is a scheme. By construction, $Z\to \mathcal{Z}$ is a $G$-torsor. Thus, the natural $G$-equivariant map $Z\to \mathcal{Z}$ factors through the quotient $Z /G$, giving $Z\to Z /G\xrightarrow{\cong }\mathcal{Z}$.
\end{proof} 

\begin{lem}\label{lem:open-substack-is-quotient}
	Let $\mathcal{U}$ be an open substack of a stack $X /G$. Then there is an isomorphism $U /G\xrightarrow{\cong}\mathcal{U}$ for a $G $-invariant open subscheme $U\hookrightarrow X$.
\end{lem} 
\begin{proof}
	The proof is the same as the one above for \zcref{lem:closed-substack-is-quotient}.
\end{proof}

We also have the following lemma describing the
formal completion along a $G$-equivariant map of derived schemes.

\begin{lem}
	\label{lem: formal completion is quotient}
	Let $Y$ and $U$ be prestacks equipped
	with an action of a smooth group prestack $G$,
	and let $f: Y \to U$ be a $G$-equivariant map. Then
	the formal completion $\wih{U_Y}$ has a canonical $G$-action, and
	there is a canonical equivalence $\wih{U_Y}/G \xrightarrow{\simeq} \wih{(U/G)_{Y/G}}$.
\end{lem}

\begin{proof}
	The canonical $G$-action on $\wih{U_Y}$ is given as follows. Since $(-)_{\dR}$
	preserves limits, $G_{\dR}$ has a group structure coming from
	that of $G$, and $U_{\dR}$ and $Y_{\dR}$ have natural $G_{\dR}$-actions
	coming from the $G$-actions on $U$ and $Y$. These $G_{\dR}$-actions
	restrict along the group map $G \to G_{\dR}$ to give $G$-actions on $U_{\dR}$
	and $Y_{\dR}$. Taking the fiber product $U \times_{U_{\dR}} Y_{\dR}$ in the
	category of $G$-prestacks gives the desired $G$-action on $\wih{U_Y}$.

	In order to show the desired equivalence,
	we note that there is a natural morphism 
	$\wih{U_Y}/G \xrightarrow{} \wih{(U/G)_{Y/G}}$
	induced by the canonical maps $U \to U/G$ and $Y \to Y/G$.
	We proceed by showing that this natural map induces an
	isomorphism between the spaces of $S$-points of each stack,
	for an arbitrary affine derived scheme $S$. 
	By definition, $\wih{(U/G)_{Y/G}}(S) = \sfMap(S, \wih{(U/G)_{Y/G}})$ is the
	$\infty$-groupoid of compatible diagrams,
	\[\left\{\begin{tikzcd}
		{S_{\red}} & {Y/G} \\
		S & {U/G}
		\arrow[from=1-1, to=1-2]
		\arrow[from=1-1, to=2-1]
		\arrow["f", from=1-2, to=2-2]
		\arrow[from=2-1, to=2-2]
	\end{tikzcd}\right\},\]
	where by abuse of notation we use $f$ also
	to denote the induced map $Y/G \to U/G$.
	That is, an $S$-point of $\wih{(U/G)_{Y/G}}$ is an $S$-point
	of $U/G$ together with a chosen factorization through $f$ of its restriction to $S_{\red}$.
	At the same time, an $S$-point of $U/G$ is the data of a $G$-torsor $P \to S$ together with 
	a $G$-equivariant map $\theta: P \to U$. Thus, the $S$-points of $\wih{(U/G)_{Y/G}}$ 
	are equivalently compatible diagrams,
	\[\left\{\begin{tikzcd}
		{P \times_S S_{\red}} & Y \\
		P & U
		\arrow[from=1-1, to=1-2]
		\arrow[from=1-1, to=2-1]
		\arrow["f", from=1-2, to=2-2]
		\arrow["\theta", from=2-1, to=2-2]
	\end{tikzcd}\right\},\]
	i.e. the data of a $G$-torsor $P \to S$ together
	with a chosen factorization of $P \times_S S_{\red}$
	through $f$. 
	
	On the other hand, the $S$-points of $(\wih{U_Y}/G)(S)$ is the $\infty$-groupoid
	whose points consist of a chosen $G$-torsor $P \to S$ together
	with a $G$-equivariant map $p: P \to \wih{U_Y}$, and a
	map $p: P \to \wih{U_Y}$ is a point in the $\infty$-groupoid
	of compatible diagrams,
	\[\left\{\begin{tikzcd}
		{P_{\red}} & Y \\
		P & U
		\arrow[from=1-1, to=1-2]
		\arrow[from=1-1, to=2-1]
		\arrow["f", from=1-2, to=2-2]
		\arrow[from=2-1, to=2-2]
	\end{tikzcd}\right\}.\]

	The map of points
	\[\begin{tikzcd}
		& {P_{\red}} && Y \\
		{P \times_S S_{\red}} && Y \\
		& P && U \\
		P && U
		\arrow[from=1-2, to=1-4]
		\arrow[from=1-2, to=2-1]
		\arrow[from=1-2, to=3-2]
		\arrow[from=1-4, to=2-3]
		\arrow[from=1-4, to=3-4]
		\arrow[from=2-1, to=2-3]
		\arrow[from=2-1, to=4-1]
		\arrow[from=2-3, to=4-3]
		\arrow[from=3-2, to=3-4]
		\arrow[from=3-2, to=4-1]
		\arrow[from=3-4, to=4-3]
		\arrow[from=4-1, to=4-3]
	\end{tikzcd}\]
	corresponding to the induced map of $\infty$-groupoids
	$\left(\wih{U_Y}/G\right)(S) \xrightarrow{} \left(\wih{(U/G)_{Y/G}}\right)(S)$
	is the map of diagrams in which $Y \to Y$, $U \to U$, $P \to P$ are the identity morphisms,
	and $P_{\red} \to P \times_S S_{\red}$ is the universal map to the fiber product.
	Since $P_{\red} \to P \times_S S_{\red}$ is natural in $P$, in order to show that
	$(\wih{U_Y}/G)(S) \xrightarrow{} (\wih{(U/G)_{Y/G}})(S)$ is an equivalence,
	it suffices to show that $P_{\red} \to P \times_S S_{\red}$ is an equivalence.
	This is clearly true \'etale locally on $S_{\red}$ since $G$ is smooth (and $P$ is a $G$-torsor), and
	since the property of being reduced is \'etale local, we are done.
\end{proof}


\subsubsection{Miscellanea}

\begin{lem}
	\label{lem: formal completion commutes with fiber products}
	Suppose that $X \to Y$ and $Z \to Y$ are maps
	of prestacks. Let $W := Z \times_Y X$ denote the fiber product.
	Then $\wih{Z_W} \simeq \wih{Y_X} \times_Y Z$.
\end{lem}

\begin{proof}
	By definition, $\wih{Z_W} = Z \times_{Z_{\dR}} W_{\dR}$.
	On the other hand, we have a natural equivalence $W_{\dR} = (Z \times_Y X)_{\dR}
	\xrightarrow{\simeq} Z_{\dR} \times_{Y_{\dR}} X_{\dR}$, giving
	an equivalence,
	\begin{align*}
		\wih{Z_W} &\xrightarrow{\simeq} Z \times_{Z_{\dR}} \left(Z_{\dR} \times_{Y_{\dR}} X_{\dR}\right) \\
		&\xrightarrow{\simeq} Z \times_{Y_{\dR}} X_{\dR} \\
		&\xleftarrow{\simeq} Z \times_Y Y \times_{Y_{\dR}} X_{\dR} = Z \times_Y \wih{Y_X}.
	\end{align*}
\end{proof}

\begin{lem}
	\label{lem: formal completion of formal completion is formal completion}
	Let $X \to Y$ and $Z \to W$ be maps of prestacks, and suppose
	that we have a map $\calY := \wih{Y_X} \to \wih{W_Z} =: \calW$. Then there
	is a natural equivalence,
	\[\wih{\calW_{\calY}} \xrightarrow{\simeq} \wih{W_X},\]
	where the formal completion on the right-hand side it taken
	along the composition of the sequence of maps $X \to \wih{Y_X} \to \wih{W_X} \to W$.
\end{lem}

\begin{proof}
	Using the universl property of fiber products, we obtain the following sequence of maps:
	\begin{align*}
		\calW_{\calY} &:= \calW \times_{\calW_{\dR}} \calY_{\dR} \\
		&= \wih{W_Z} \times_{\wih{W_Z}_{\dR}} \wih{Y_X}_{\dR} \\
		&= (W \times_{W_{\dR}} Z_{\dR}) \times_{(W \times_{W_{\dR}} Z_{\dR})_{\dR}} (Y \times_{Y_{\dR}} X_{\dR})_{\dR} \\
		&\xrightarrow{\simeq} (W \times_{W_{\dR}} Z_{\dR}) \times_{W_{\dR} \times_{W_{\dR}} Z_{\dR}} (Y_{\dR} \times_{Y_{\dR}} X_{\dR}) \\
		&\xrightarrow{\simeq} (W \times_{W_{\dR}} Z_{\dR}) \times_{Z_{\dR}} X_{\dR} \\
		&\xrightarrow{\simeq} W \times_{W_{\dR}} X_{\dR} \\
		&=: \wih{W_X},
	\end{align*}
	where the only non-trivial equivalence is the one on the fourth line coming from
	the fact that $(-)_{\dR}$ preserves limits as a right adjoint.
\end{proof}


\subsection{Loop spaces}

\begin{lem}
	\label{lem: loops of closed immersion is closed immersion}
	Let $i:Z \hookrightarrow X$ be a closed substack of a stack $X$, and let $K \in \Ani$ be nonempty. 
	Then the induced map $i_{\Map}: \Map(K_B,Z)\to \Map(K_B,X)$ is a closed immersion,
	where $K_B$ denotes the Betti stack of $K$.
\end{lem}

\begin{proof}
	For every $K \in \Ani$, we tautologically have $\colim_K \ast_K \xrightarrow{\simeq} K$, where
	the latter indicates the constant diagram $\ast_K: K \to \Ani$ with value $\ast$. The functor
	$\underline{(-)}: \Ani \to \PreStk$ taking an anima to its constant prestack preserves colimits,
	as does the sheafification functor $L: \PreStk \to \Stk$. Thus, the Betti stack $K_B := L(\underline{K})$ is equivalent
	to $\colim_K \ast_B$. For any stack $Y$,
	$\Map(-, Y)$ takes colimits in the first argument to limits, so we obtain an
	equivalence 
	\[\Map(K_B, Y) \xrightarrow{\simeq} \Map(\colim_K \ast_B, Y) \xrightarrow{\simeq} \lim_K \Map(\ast_B, Y) \simeq \lim_K Y,\]
	where the latter object, explicitly, is the limit of the diagram
	\[D_Y: K \simeq K^{\op} \xrightarrow{\ast_K^{\op}} \Ani^{\op} \xrightarrow{(-)_B^{\op}} \Stk^{\op} \xrightarrow{\Map(-,Y)} \Stk.\]
	Moreover, this equivalence is natural in $Y$, so we obtain a commutative diagram,
	\[\begin{tikzcd}
		{\Map(K_B,Z)} & {\lim_K D_Z} \\
		{\Map(K_B,X)} & {\lim_K D_X}
		\arrow["\simeq", from=1-1, to=1-2]
		\arrow["{i_{\Map}}"', from=1-1, to=2-1]
		\arrow["{\i}", from=1-2, to=2-2]
		\arrow["\simeq", from=2-1, to=2-2],
	\end{tikzcd}\]
	where $\i$ denotes the map of limits induced by the
	natural transformation $D_Z \to D_X$ induced by $i: Z \hook X$.
	It suffices, therefore, to show that $\i$ is a closed immersion.

	Given a diagram in $\Fun(K, \PreStk)_{/D_Y} \simeq \Fun(K, \PreStk_{/Y})$,
	there are two limit operations one may perform. 
	The first is to take the limit within the overcategory $\PreStk_{/Y}$
	to obtain a functor,
	\[{\lim_K}^{/Y}: \Fun(K, \PreStk)_{/D_Y} \to \PreStk_{/Y},\]
	which corresponds to the right adjoint of the constant diagram functor
	$D_{(-)}: \PreStk_{/Y} \to \Fun(K, \PreStk_Y) \simeq \Fun(K, \PreStk)_{/D_Y}$. 
	In particular, if $Y = \ast$, we get the functor $\lim_K: \Fun(K, \PreStk)
	\to \PreStk$. The functor $\lim_K$ in turn induces a functor on overcategories,
	\[\lim_K: \Fun(K, \PreStk)_{/D_Y} \to \PreStk_{/\lim_K D_Y},\]
	which furnishes the second limit operation.
	
	Now let $S$ be an affine derived scheme, and let $f: S \to \lim_K D_X$ be an
	arbitrary map. In order to show that $\i$ is a closed immersion,
	we must, by definition (\cite[\S 3.7]{CautisWilliams}),
	show that the map of $0$-truncated stacks,
	\[\tau_{\leq 0}(\lim_K Z \times_{\lim_K X} S) \to \tau_{\leq 0} S\]
	is a closed immersion.
	For any $\calT \in \PreStk_{/\lim_K D_X}$,
	and any $\calS \in \PreStk_{/S}$, we have the natural
	equivalence,
	\begin{equation}
		\label{eqn: morphisms in overcategory}
	\sfMap_{\PreStk_{/S}}(\calS, \calT \times_{\lim_K D_X} S) \simeq \sfMap_{\PreStk}(\calS, \calT) \times_{\sfMap_{\PreStk}(\calS, \lim_K D_X)} \sfMap_{\PreStk_{/S}}(\calS, S),
	\end{equation}
	coming from the fact that the overcategory
	$\PreStk_{/S}$ is equivalent to the oriented
	product $\PreStk \vec{\times}_{\PreStk} \{S\}$
	(\cite[\href{https://kerodon.net/tag/02VT}{Tag 02VT}]{kerodon}),
	and morphisms in the latter are precisely what appears above.
	Now let $\calT = \lim_K D_Z$, where the structure map
	$\lim_K D_Z \to \lim_K D_X$ comes from the map $D_Z \to D_X$ induced
	by $i: Z \hook X$. Then the right-hand side of \zcref[noname]{eqn: morphisms in overcategory}
	becomes
	\begin{align*}
		\sfMap(\calS, \lim_K D_Z) \times_{\sfMap(\calS, \lim_K D_X)} \sfMap_{/S}(\calS, S) &\simeq \sfMap(\calS, \lim_K D_Z) \times_{\sfMap(\calS, \lim_K D_X)} \ast \\
																																					 &\simeq \lim_K \sfMap(\calS, D_Z) \times_{\lim_K \sfMap(\calS, D_X)} \lim_K \Delta_{\ast} \\
																																					 &\simeq \lim_K \left(\sfMap(\calS, D_Z) \times_{\sfMap(\calS, D_X)} \Delta_{\ast}\right) \\
																																					 &\simeq \lim_K \left(\sfMap_{/D_S}(\calS, D_Z \times_{D_X} D_S)\right) \\
																																					 &\simeq \sfMap\left(\calS, {\lim_K}^{/X} (D_Z \times_{D_X} D_S)\right)																																 
	\end{align*}
	where $\Delta_{\ast}$ denotes the constant $K$-shaped diagram $K \to \Ani$
	with value $\ast$. Here, the first equivalence uses that 
	$\sfMap_{\PreStk_{/S}}(\calS, S) \simeq \{\calS \to S\} \simeq \ast$;
	the second equivalence uses that $\sfMap(\calS, -)$ commutes with limits
	and $\ast \simeq \lim_K \Delta_{\ast}$; the third equivalence use
	that limits commute with one another; the fourth equivalence uses the
	definition of mapping spaces in the overcategory; and the final equivalence uses
	that limits commute with $\sfMap_{/D_X}(\calS, -)$ once again. 
	Since $\calS \in \PreStk_{/S}$ was arbitrary, this shows that,
	\[{\lim_K}^{/S} (D_Z \times_{D_X} D_S) \xrightarrow{\simeq} \lim_K D_Z \times_{\lim_K D_X} S.\]
	
	Thus, we are reduced to showing that the canonical map $\tau_{\leq 0}(\lim_K^{/S}(D_Z \times_{D_X} D_S)) \to \tau_{\leq 0}S$
	is the image of a closed immersion of ordinary affine schemes 
	under the inclusion $i_{\leq 0}: \PreStk_{\leq 0} \to \PreStk$.
	Since $r_{\leq 0} \circ i_{\leq 0} \simeq \id_{\PreStk_{\leq 0}}$,\footnotemark
		\footnotetext{$r_{\leq 0}: \PreStk \to \PreStk_{\leq 0}$ denotes the
			restriction of prestacks along the inclusion of classical rings.}
	it therefore suffices to show that $r_{\leq 0}(\lim_K^{/S}(D_Z \times_{D_X} D_S)) \to r_{\leq 0}S$
	is a closed immersion of ordinary affine schemes.
	For each point $k \in K$, we note that
	$r_{\leq 0}(D_Z(k) \times_{D_X(k)} D_S(k)) = r_{\leq 0}(Z \times_X S) \to r_{\leq 0}S$ is a closed immersion
	because $Z \hook X$ is a closed immersion.
	In particular, since $r_{\leq 0}S \simeq \Spec \pi_0 \calO(S)$ is a classical
	affine scheme, this implies that $r_{\leq 0}(D_Z(k) \times_{D_X(k)} D_S(k))$ is
	an affine classical scheme, and the induced map
	$\pi_0\calO(S) \to \pi_0\calO(D_Z(k) \times_{D_X(k)} D_S(k))$ is a surjective map of commutative rings.
	On the other hand, since $r_{\leq}$ is a right adjoint (to the inclusion $i_{\leq 0}$),
	it preserves limits. In particular, we have $r_{\leq 0}(\lim_K^{/S}(D_Z \times_{D_X} D_S)) \xrightarrow{\simeq}
	\lim_K^{/(r_{\leq 0}S)} r_{\leq 0}(D_Z \times_{D_X} D_S)$. The latter is a limit
	of classical affine schemes.
	Since $r_{\leq 0}S$ is affine, the canonical map $\lim_K^{/S} r_{\leq0}(D_Z \times_{D_X} D_S) \to r_{\leq 0}S$
	is determined by its induced map on global sections, 
	\begin{equation}
		\label{eqn: surjective morphism}
		\pi_0\calO(S) \to \calO\left({\lim_K}^{/S} r_{\leq 0}(D_Z \times_{D_X} D_S)\right) \simeq {\colim_K}^{\pi_0\calO(S)/} \calO\left(r_{\leq 0}\left(D_Z(k)  \times_{D_X(k)} D_S(k)\right)\right)
	\end{equation}
	where the equivalence comes from the fact that
	${\lim_K}^{/S} r_{\leq 0}(D_Z \times_{D_X} D_S)$ is a limit
	of affine schemes.
	If we can show that this induced map of global sections
	is surjective, then we are done. 
	Since every morphism in $K$ is invertible, however, all the structure
	morphisms of the diagram $\calO\left(r_{\leq 0}\left(D_Z(k)  \times_{D_X(k)} D_S(k)\right)\right)$ are isomorphisms.
	Thus, it is obvious that the map \zcref[noname]{eqn: surjective morphism} is surjective, as desired.
\end{proof}

In light of the above, it makes sense to formulate the following lemma.

\begin{lem}
	\label{lem:loops-of-complement}
	Let $i:Z\hookrightarrow X$ be a closed substack of a stack $X$ with open complement $U$,
	and let $K \in \Ani$ be a nonempty connected anima, whose Betti stack is denoted $K_B$.
	Then $\Map(K_B,U)$ is the open complement to $\Map(K_B,Z)$ in $\Map(K_B,X)$.
\end{lem}

\begin{proof}
	Let $S$ be an affine derived scheme,
	and consider a point $f:S\to \Map(K_B,X)$. By definition, 
	$f$ factors through the complement $\Map(K_B,X) \setminus \Map(K_B,Z)$ if and only if
	$S\times_{\Map(K_B,X)}\Map(K_B,Z)\simeq \emptyset$.
	On the other hand, $f$ factors through $\Map(K_B,U)$ if and only if,
	for the map $\tilde{f}: S \times K_B \to X$ corresponding to $f$ under adjunction,
	$(S\times K_B)\times _XZ \simeq \emptyset$.
	It thus suffices to show that
	\[
	S\times _{\Map(K_B,X)}\Map(K_B,Z)\simeq \emptyset  \iff	(S\times K_B)\times _XZ\simeq \emptyset.
	\]
	Suppose $(S\times K_B)\times _XZ\simeq \emptyset$. For any $T$, we have the following
	chain of equivalences: 
	\begin{align*}
		\sfMap(T,S &\times_{\Map(K_B,X)}\Map(K_B,Z)) \\
		&\simeq \sfMap(T,S)\times _{\sfMap(T,\Map(K_B,X))}\sfMap(T,\Map(K_B,Z))\\
		&\simeq \sfMap(T,S)\times_{\sfMap(T,\Map(K_B, S \times K_B))}\sfMap(T,\Map(K_B, S\times K_B))\times_{\sfMap(T,\Map(K_B,X))}\sfMap(T, \Map(K_B,Z))\\
		&\simeq \sfMap(T,S)\times_{\sfMap(T\times K_B,S\times K_B)}\sfMap(T\times K_B,(S\times K_B)\times_X Z)\\
		&\simeq \emptyset,
	\end{align*} 
	where in the second line the map $\sfMap(T,S) \to \sfMap(T, \Map(K_B, S \times K_B))$ is given by
	postcomposition with the map $S \to \Map(K_B, S \times K_B)$ corresponding to the identity $\id_{S \times K_B}$ under adjunction, 
	and the map $\sfMap(T, \Map(K_B, S \times K_B)) \to \sfMap(T, \Map(K_B,X))$ is given
	by postcomposition with the map $\Map(K_B, S \times K_B) \to \Map(K_B, X)$ given itself by postcomposition
	with the map $S \times K_B \to X$ corresponding under adjunction to the map $S\to\Map(K_B,X)$ coming from the first coordinate of
	the chosen $T$-point in $S\times _{\Map(K_B,X)}\Map(K_B,Z)$.
	
	It remains to show the converse. Suppose that $(S\times K_B)\times_X Z\not\simeq \emptyset$. Since emptiness is witnessed by
	field-valued points, there exists $T = \Spec \kappa$ for some field $\kappa$ such that $\sfMap(T, (S \times K_B) \times_X Z))$ is nonempty. Choose a point in
	$\sfMap(T, (S \times K_B) \times_X Z))$. The data of such a point is, by the universal property, 
	the data of maps $\wit{s}: T \to S \times K_B$, and $z: T \to Z$ together with a homotopy of the compositions,
	\[(T \xrightarrow{\wit{s}} S \times K_B \xrightarrow{\wit{f}} X) \xLongrightarrow{\eta} (T \xrightarrow{z} Z \xrightarrow{i} X).\]
	The task at hand is to produce a $T$-point of $S\times_{\Map(K_B,X)}\Map(K_B,Z)$ from this data.
	
	Let $s: T \to S$ denote the map obtained by the composition $T \xrightarrow{\wit{s}} S \times K_B \xrightarrow{\pr_1} S$, and 
	$k: T \to K_B$ the map obtained by $T \xrightarrow{\wit{s}} S \times K_B \xrightarrow{\pr_2} K_B$. 
	We seek a point $\zeta: T \to \Map(K_B,Z)$ together with a homotopy $H: f \circ s \xLongrightarrow{} i_* \circ \zeta$, where
	$i_*: \Map(K_B,Z) \to \Map(K_B,X)$ denotes the map given by composition with $i$. There is a canonical map,
	\[\sfMap(T, \Map(K_B,X)) \xrightarrow{} \sfMap_{\Spc}(K_B(T), X(T)),\] 
	under which the map $f \circ s: T \to \Map(K_B,X)$ is sent to a map $\tilde{f}_s: K_B(T) \to X(T)$. The homotopy $\eta$ tells us precisely
	that the image of $k \in K_B(T)$ lies in the image of the map $Z(T) \to X(T)$.
	Consider the Cartesian diagram 
	\[\begin{tikzcd}
		{K_B(T) \times_{X(T)} Z(T)} & {K_B(T)} \\
		{Z(T)} & {X(T)}
		\arrow["{p_K}", from=1-1, to=1-2]
		\arrow["{p_Z}"', from=1-1, to=2-1]
		\arrow["\sigma"', Rightarrow, nfold, from=1-2, to=2-1]
		\arrow["{\tilde{f}_s}", from=1-2, to=2-2]
		\arrow["{i_T}", from=2-1, to=2-2]
	\end{tikzcd}\]
	We claim that $p_K$ is an equivalence. Indeed, it has one nonempty fiber by assumption, 
	over $k\in K_B(T)$. Given another fiber, over $k'\in K_B(T)$, there is a path $k\to k'$. 
	This path defines a map of fibers by concatenation: A point in the fiber over $k$ 
	is a point $a\in K_B(T)\times_{X(T)}Z(T)$ together with a path $p_K(a)\to k$. 
	Thus, the fiber over $k'$ cannot be empty. Since $T$ is classical, $i_T:Z(T)\hookrightarrow X(T)$ is a monomorphism of 
	spaces,\footnotemark
	\footnotetext{If $i:Z\to X$ be a closed immersion of stacks, and $T = \Spec A$ a classical affine scheme,
		then $i_T:Z(T)\hookrightarrow X(T)$ is a monomorphism of spaces. Indeed, fix 
		$x: T \to X$, and write $T \times_X Z \simeq \Spec B$, with 
		$H^0(B) \simeq A/\mathfrak{a}$ for some ideal $\mathfrak{a} \subset A$. 
		The homotopy fiber of $i_T: Z(T) \hookrightarrow X(T)$ over $x$ is given by
		$\mathsf{Map}_{/T}(T,T \times_X Z) \simeq \sfMap_{\CAlg_A}(B,A) \simeq \Hom_A(A/\mathfrak{a}, A)$, 
		since maps from a connective algebra into a discrete algebra factor 
		through $\pi_0$. Now, an $A$-algebra
		map $A/\mathfrak{a} \to A$ exists precisely when $\mathfrak{a} = 0$, 
		in which case it is unique; hence, this fiber is either empty or contractible.}
	 hence its base change $p_K$ is a monomorphism, 
	so its fibers are all nonempty and contractible. Thus, $p_K$ is an equivalence. 
	Thus, we obtain a map $\tilde\zeta\coloneq p_Z\circ (p_K)^{-1}: K_B(T) \to Z(T)$.
	Since $T = \Spec \kappa$, it is a connected affine scheme, so $K_B(T) \simeq K$.
	Because $\Stk_k$ is an $\infty$-topos, it is cotensored over $\Spc$ (\cite[Remark 5.5.2.6]{HTT}),
	meaning there is a natural equivalence,
	\begin{equation}
		\label{eqn: cotensor}
		\sfMap(T, \Map(K_B, Z)) \simeq \sfMap(K, Z(T)).
	\end{equation}
	Let $\zeta:T\to\Map(K_B,Z)$ denote the point that
	corresponds to $\tilde \zeta$ under this equivalence. 
	
	It remains to specify a homotopy $H: f \circ s \xLongrightarrow{} i_* \circ \zeta$. We have an explicit homotopy
	\begin{align*}
		\wit{H} :\tilde f_s \simeq \tilde f_s\circ p_K \circ (p_K)^{-1} \xLongrightarrow{\sigma \circ (p_K)^{-1}} i_T\circ p_Z\circ (p_K)^{-1} \simeq i_T\circ\tilde\zeta,
	\end{align*}
	of points in $\sfMap(K, X(T))$. Under the natural equivalence \zcref[noname]{eqn: cotensor}, 
	$\wit{H}$ gives a path $H: f\circ s \xLongrightarrow{} i_* \circ \zeta$
	in $\Map(T,\Map(K_B,X))$. Thus, the tuple $(s,\zeta,H)$ specifies $T$-point of 
	$S \times_{\Map(K_B,X)}\Map(K_B,Z)$,
	so we are done.
\end{proof}

In particular, taking $K = S^1$ in
\zcref{lem: loops of closed immersion is closed immersion}
and \zcref{lem:loops-of-complement},
we obtain the following
results about the loop stack of a complementary substack.

\begin{cor}
	\label{cor: loops of complement}
	Let $X$ be a derived stack, and 
	let $i:Z\hookrightarrow X$ be a closed immersion. Let $U \hook X$ denote
	its complement.
	Then $\calL Z \hook \calL X$ is a closed immersion, and
	$\calL U \hook \calL X$ is its complementary open substack.
\end{cor}

It will be important later on to identify the loop stack of a formal completion
with the formal completion along the corresponding induced map
of loop stacks.

\begin{lem}
	\label{lem:loops-commute-with-formal-completion}
	Let $Y \to X$ be any map of derived stacks. Then there is an $S^1$-equivariant
	isomorphism,
	\[
		\calL\left(\wih{X_Y}\right) \simeq \wih{\calL X_{\calL Y}}.
	\] 
\end{lem} 
\begin{proof}
	By definition, $\wih{X_Y} \coloneq X \times _{X_{\dR}}Y_{\dR}$.
	As a representable functor, $\calL(-)$ commutes with limits, 
	so the natural morphism $\calL(\wih{X_Y}) \xrightarrow{\simeq} \calL X
	\times _{\calL(X_{\dR})}\calL(Y_{\dR})$ is an isomorphism.
	Thus, the lemma follows if we show that the de Rham stack and loop space
	constructions commute, i.e. if  we show the natural map $(\calL Y)_{\dR} \to \calL(Y_{\dR})$
	is an isomorphism for any stack $Y$. This follows immediately from the
	fact that $(-)_{\dR}$ preserves limits (\cite[Lemma 1.1.4]{GRCrystals}), 
	and the loop stack can be wrriten as $\calL Y \simeq Y \times_{Y \times Y} Y$.
	It is clear that all maps admit canonical $S^1$-equivariant structures. 
\end{proof} 
\begin{lem}\label{lem: unipotent loops commute with formal completion}
Let $Y\to X$ be any map of derived stacks. Then the natural $S^1$-equivariant map
\begin{align*}
	\mathcal{L}^u\left( \wih{X_Y} \right) \to \wih{\mathcal{L}^uX_{\mathcal{L}^uY}}
\end{align*} 
is an isomorphism.
\end{lem} 
\begin{proof}
	Let $L=(-)_{\text{red}}$ be the left adjoint to $R=(-)_{\text{dR}}$, and let $F=B\mathbb{G} _a \times -$ be the left adjoint to $\mathcal{L}^u=\Map(B\mathbb{G} _a,-)$. We claim that there is a canonical equivalence
	\begin{align*}
		L\circ F\simeq F\circ L.
	\end{align*} 
	This suffices to prove the claim. Indeed, passing to right adjoints gives an equivalence
	\begin{align*}
		\mathcal{L}^u\circ R\simeq R\circ \mathcal{L}^u.
	\end{align*} 
	Since $\mathcal{L}^u$ preserves limits, we have
	\begin{align*}
		\mathcal{L}^u(\wih{X_Y})\simeq& \mathcal{L}^uX\times _{\mathcal{L}^u(X_{\text{dR}})}\mathcal{L}^u(Y_{\text{dR}})\\
		\simeq&\mathcal{L}^uX\times _{(\mathcal{L}^uX)_{\text{dR}}}(\mathcal{L}^uY)_{\text{dR}}\\
		\simeq&\wih{(\mathcal{L}^uX)_{\mathcal{L}^uY}}.
	\end{align*} 
	Moreover, this composition of equivalences is precisely the natural map 
	\begin{align*}
	\mathcal{L}^u\left( \wih{X_Y} \right) \to \wih{\mathcal{L}^uX_{\mathcal{L}^uY}}.
	\end{align*} 
	We will now prove the equivalence $L\circ F\simeq F\circ L$. To see this, observe first that reduction commutes with products with $\mathbb{A}^n$. Indeed, for all $n$, the functors
	\begin{align*}
		Z\mapsto L(Z\times \mathbb{A}^n)\qquad \text{and}\qquad Z\mapsto L(Z)\times \mathbb{A}^n
	\end{align*} 
	preserve colimits. Thus, the canonical comparison map between them is an isomorphism if and only if it is an isomorphism on affine schemes. There, it is simply the identity
	\begin{align*}
		(H^0(A)[x_1,\ldots,x_n])_{\text{red}}=H^0(A)_{\text{red}}[x_1,\ldots,x_n].
	\end{align*} 
	Now, apply this in each degree to the bar presentation of $B\mathbb{G} _a$:
	\begin{align*}
		L(Z\times B\mathbb{G} _a)\simeq&\colim_{[n]\in\Delta^{\text{op}}}L(Z\times \mathbb{G} _a^n)\\
		\simeq&\colim L(Z)\times \mathbb{G} _a^n\\
		\simeq&L(Z)\times B\mathbb{G} _a.
	\end{align*} 
\end{proof}
\begin{lem}\label{lem: formal loops of open complement}
	Let $i:Z\to X$ be a closed substack of a stack $X$. Then $\hL i:\hL Z\to \hL X$ is a closed immersion. Moreover, if $U$ is the open complement to $Z$ in $X$, then $\hL U$ is the open complement to $\hL Z$ in $\hL X$, and the natural $S^1$-equivariant map
	\begin{align*}
		\hL(\wih{X_Z})\to \wih{(\hL X)_{\hL Z}}
	\end{align*} 
	is an isomorphism.
\end{lem} 
\begin{proof}
	Note that $Z_{\text{dR}}\to X_{\text{dR}}$ is a monomorphism, since the de Rham prestack only evaluates on underlying reduced, classical stacks, where this holds. Hence, by \zcref{cor: loops of complement}, $(\calL Z)_{\text{dR}}\to(\calL X)_{\text{dR}}$ is a monomorphism as well; thus,
    \begin{align*}
        Z_{\text{dR}}\simeq (\calL Z)_{\text{dR}}\times_{(\calL X)_{\text{dR}}}X_{\text{dR}}.
    \end{align*}
    In turn, this implies that 
    \begin{align*}
        \hL Z\simeq \calL Z\times_{\calL X} \hL X,
    \end{align*}
    so the map $\hL Z \to \hL X$ is the base change of a closed immersion hence also a closed immersion. We first show that the natural map
	\begin{align}\label{eq: natural map of formal loop spaces}
		\hL U\to \hL X\times _{\mathcal{L}X}\mathcal{L}U
	\end{align} 
	is an isomorphism. Indeed, we compute
	\begin{align*}
		\hL X\times _{\mathcal{L}X}\mathcal{L}U\coloneq& X_{\text{dR}}\times _{(\mathcal{L}X)_{\text{dR}}}\mathcal{L}X\times _{\mathcal{L}X}\mathcal{L}U\\
		\xrightarrow{\simeq}&\mathcal{L}U\times _{(\mathcal{L}X)_{\text{dR}}}X_{\text{dR}}\\
		\xleftarrow{\simeq}&\mathcal{L}U\times _{(\mathcal{L}U)_{\text{dR}}}(\mathcal{L}U)_{\text{dR}}\times _{(\mathcal{L}X)_{\text{dR}}}X_{\text{dR}}\\
		\simeq&\mathcal{L}U\times _{(\mathcal{L}U)_{\text{dR}}}\left( \mathcal{L}U\times _{\mathcal{L}X}X \right) _{\text{dR}},
	\end{align*} 
	where in the last line we use the fact that $(-)_{\text{dR}}$ preserves limits, since it is a right adjoint. Thus, we see that it suffices to show that the canonical map
	\begin{align*}
		u:U\to \mathcal{L}U\times _{\mathcal{L}X}X
	\end{align*} 
	is an isomorphism. To see this, note that there is also a natural map
	\begin{align*}
		e:\mathcal{L}U\times _{\mathcal{L}X}X\to \mathcal{L}U\xrightarrow{\ev_U} U,
	\end{align*} 
	where the first map is the projection to $\mathcal{L}U$ and the second is evaluation at the base point of $S^1$. We claim that the composition $e\circ u\simeq\id_U$. Indeed, the composition $e\circ u$ is naturally isomorphic to the composition $\ev_U\circ c_U$, where $c_U: U\to \mathcal{L}U$ is the constant loop map. By definition, $\ev_U\circ c_U\simeq\id_U$. We next claim that $u\circ e\simeq \id_P$, where $P\coloneq \mathcal{L}U\times _{\mathcal{L}X}X$. To see this, note that the maps $u,e$ are relative to $X$. Moreover, the projection $P\to X$ is an open immersion by \zcref{cor: loops of complement}, so it is a monomorphism. Thus, the composition $u\circ e$ is an endomorphism of $P$ relative to $X$; i.e., it is an object of
	\begin{align*}
		\mathsf{Map}_{/X}(P,P).
	\end{align*} 
	Now, since $P\to X$ is a monomorphism, for every $S$, the map of spaces
	\begin{align*}
		\mathsf{Map}(S,P)\to \mathsf{Map}(S,X)
	\end{align*} 
	has fibers which are either empty or contractible. For $S=P$, we see that the fiber over $P\to X$ is precisely
	\begin{align*}
		\mathsf{Map}_{/X}(P,P).
	\end{align*} 
	Since this fiber contains in particular the identity endomorphism of $P$, we conclude that $\id_P\simeq u\circ e$. Thus, we conclude that the natural map \zcref[noname]{eq: natural map of formal loop spaces} is an isomorphism as desired.

	Let now $S\to \hL Z$ be an $S$-point. We will show that
	\begin{align*}
		S\times _{\hL X}\hL U\simeq \emptyset,
	\end{align*} 
	which, together with the isomorphism  \zcref[noname]{eq: natural map of formal loop spaces} will prove that $\hL U$ is the open complement to $\hL Z $ in $\hL X$. For any stack $Y$, $Y\simeq \emptyset$ if and only if $Y_{\text{dR}}\simeq \emptyset$.
	Applying this to the stack $S\times _{\hL X}\hL U$, we find
	\begin{align*}
		\left(	S\times _{\hL X}\hL U\right)_{\text{dR}}\simeq &S_{\text{dR}}\times _{(\hL X)_{\text{dR}}}(\hL U)_{\text{dR}}\\
	\simeq&S_{\text{dR}}\times _{X_{\text{dR}}}U_{\text{dR}}\\
	\xleftarrow{\simeq}&S_{\text{dR}}\times _{Z_{\text{dR}}}Z_{\text{dR}}\times _{X_{\text{dR}}}U_{\text{dR}}\\
	\simeq&S_{\text{dR}}\times _{Z_{\text{dR}}}(Z\times _XU)_{\text{dR}}\\
	\simeq&S_{\text{dR}}\times _{Z_{\text{dR}}}\emptyset_{\text{dR}}\\
	\simeq&S_{\text{dR}}\times _{Z_{\text{dR}}}\emptyset\\
	\simeq&\emptyset,
	\end{align*} 
where in the first and fourth lines we use that $(-)_{\text{dR}}$ preserves limits, as it is a right adjoint, in the second line, we use that the de Rham prestack is nil-invariant, and in the fifth line we use that $U $ is the open complement to $Z$ in $X$.
	We now show that the natural map
	\begin{align*}
		\hL(\wih{X_Z})\to \wih{(\hL X)_{\hL Z}}
	\end{align*} 
	is an isomorphism. Indeed, we have the canonical isomorphisms
	\begin{align*}
		\hL(\wih{X_Z})\coloneq& \wih{(\mathcal{L}(\wih{X_Z}))_{\wih{X_Z}}}\\
			\simeq&\mathcal{L}(\wih{X_Z})\times _{(\mathcal{L}Z)_{\text{dR}}}Z_{\text{dR}}\\
			\simeq&\mathcal{L}X\times _{(\mathcal{L}X)_{\text{dR}}}(\mathcal{L}Z)_{\text{dR}}\times _{(\mathcal{L}Z)_{\text{dR}}}Z_{\text{dR}}\\
			\xrightarrow{\simeq}&\mathcal{L}X\times _{(\mathcal{L}X)_{\text{dR}}}Z_{\text{dR}}\\
			\simeq&\mathcal{L}X\times _{(\mathcal{L}X)_{\text{dR}}}(\hL Z)_{\text{dR}}\\
			\eqcolon &\wih{(\mathcal{L}X)_{(\hL Z)}},
	\end{align*} 
	where in the second and fifth lines we use that the functor $(-)_{\text{dR}}$ is nil-invariant. On the other hand, again by nil-invariance of $(-)_{\text{dR}}$, we have canonical isomorphisms
	\begin{align*}
		\wih{(\hL X)_{\hL Z}}\simeq &\hL X\times _{X_{\text{dR}}}Z_{\text{dR}}\\
		\simeq& \wih{(\hL X)_{Z}}.
	\end{align*} 
	Thus, it suffices to show that the natural map
	\begin{align*}
		\wih{(\hL X)_{Z}}\to \wih{(\mathcal{L}X)_{\hL Z}}
	\end{align*}
	is an isomorphism. But this follows from \zcref{lem: formal completion of formal completion is formal completion}.
\end{proof}

\section{Potent $\calD$-modules}
\label{sec:potent D-modules}

Following \cite{BZPotent}, we make the following defintion.

\begin{defn}
	Let $X$ satisfy the conditions of \zcref{cor: loop spaces of formal completions are ind-geometric}. The category of \emph{potent $\calD$-modules}
	on $X$ is defined to be
	\[\Pot(X) := \IndCoh(\calL X)^{\tT} \, (= \Ind(\Coh(\calL X)^{tS^1})).\]
\end{defn}

\begin{rem}
	Note that, for $X$ a coherent stack, the subcategory of
	coherent objects of both $\QCoh(X)$ and $\IndCoh(X)$
	is $\Coh(X)$. Thus, in defining the category of potent
	$\calD$-modules, it makes no difference whether
	we take $\QCoh(\calL X)^{\tT}$ or $\IndCoh(\calL X)^{\tT}$.
\end{rem}

\begin{rem}
	The name in the above definition takes inspiration from the equivalence
	under Koszul duality between cyclic sheaves on the unipotent loop
	space of $X$ and ordinary $\calD$-modules on $X$. Potent $\calD$-modules
	correspond to cyclic sheaves on the full---or \emph{potent}---loop space rather
	than just the \emph{uni}potent loop space of $X$.
\end{rem}

\begin{lem}
	\label{lem: regularization of adjunction is adjunction}
	Suppose that $\scrC, \scrD \in \bCoh_t$, and that $L: \scrC \leftrightarrows \scrD: R$ 
	are adjoint functors which each are left t-exact up to a shift, exact, and which preserve uniformly
	left t-bounded filtered colimits. Then the induced functors 
	$\scrR(L): \scrR(\scrC) \leftrightarrows \scrR(\scrD): \scrR(R)$
	are adjoint to one another.
\end{lem}

\begin{proof}
	The existence of the induced functors $\scrR(L)$ and $\scrR(R)$ is established
	by \cite[Lemma 4.6.1]{Preygel}, wherein it is also stated that they preserve colimits.
	
	We first note that $L$ preserves coherent objects. Indeed, let $c \in \Coh(\scrC)$,
	and let $\{d_i\}$ be a uniformly left t-bounded filtered diagram in $\scrD$. Then
	\begin{align}
		\Map_{\scrD}(Lc, \colim_i d_i) &\simeq \Map_{\scrC}(c, R(\colim_i d_i)) \\
															&\simeq \Map_{\scrC}(c, \colim_i R(d_i)) \label{colimitem1}\\
															&\simeq \colim_i \Map_{\scrC}(c, R(d_i)) \label{colimitem2}\\
															&\simeq \colim_i \Map_{\scrD}(Lc, d_i),
	\end{align}
	where \zcref[noname]{colimitem1} follows from the hypothesis that
	$R$ preserves uniformly left t-bounded filtered colimits, and \zcref[noname]{colimitem2}
	follows from $c$ being almost compact and the diagram $R(d_i)$ being uniformly left $t$-bounded. Thus $Lc$ is almost compact. 
	At the same time, $c$, being coherent,
	is left t-bounded (i.e. in $\scrC_{<\infty}$). The functor
	$L$ being left t-exact up to a shift means precisely that $L(\scrC_{<\infty}) \subset \scrD_{<\infty}$.
	Thus $Lc \in \scrD_{<\infty}$, hence coherent in the sense of \cite[Definition 4.2.2]{Preygel}.

	Since $L$ preserves coherent objects, the regularized
	functor $\scrR(L)$ preserves compact objects. 
	By \cite[Proposition 4.3.2]{Preygel},
	the canonical functors $\scrR(\scrC) \to \scrC$ and $\scrR(\scrD) \to \scrD$ restrict
	to equivalences $\scrR(\scrC)_{<\infty} \xrightarrow{\simeq} \scrC_{\infty}$ 
	and $\scrR(\scrD)_{<\infty} \xrightarrow{\simeq} \scrD_{<\infty}$. Thus, for any $c \in \Coh(\scrC)$
	and $d \in \Coh(\scrD)$, we have
	\begin{align*}
		\Map_{\scrR(\scrD)}(\scrR(L)c, d) &\simeq \Map_{\scrR(\scrD)_{<\infty}}(\scrR(L)c, d) \\
		&\simeq \Map_{\scrD_{<\infty}}(\scrR(L)c, d) \\
		&\simeq \Map_{\scrD_{<\infty}}(Lc, d) \\
		&\simeq \Map_{\scrC_{<\infty}}(c, Rd) \\
		&\simeq \Map_{\scrC_{<\infty}}(c, \scrR(R)d) \\
		&\simeq \Map_{\scrR(\scrC)_{<\infty}}(c, \scrR(R)d) \\
		&\simeq \Map_{\scrR(\scrC)}(c, \scrR(R)d)
	\end{align*}
	Since $\scrR(L)(c)$ is compact, and since $\scrR(R)$ preserves
	colimits by construction, both sides of the equivalence $\Map_{\scrR(\scrD)}(\scrR(L)c,d) \simeq
		\Map_{\scrR(\scrC)}(c, \scrR(R)d)$ commute with filtered colimits in $d$.
		On the other hand, since $\Map(-,-)$ commutes with
	colimits in the first argument, and since $\scrR(L)$ commutes with colimits
	in $c$, we deduce that both sides of the same equivalence 
	also commute with filtered colimits in $c$. 
	Thus, we get that $\Map_{\scrR(\scrC)}(\scrR(L)x, y) \simeq \Map_{\scrR(\scrD)}(x,\scrR(R)y)$,
	for any $x \in \scrR(\scrC)$ and $y \in \scrR(\scrD)$, as desired.
	It remains to observe that this equivalence is natural in both $x$ and $y$,
	since all the equivalences in the chain above are natural in $c$ and $d$.
\end{proof}

\begin{lem}
	\label{lem: naive Tate invariants preserves adjunction}
	Suppose that $L: \scrC \leftrightarrows \scrD: R$ is an adjunction in $S^1\bmod$. Then
	the induced functors $L^{tS^1}: \scrC^{tS^1} \leftrightarrows \scrD^{tS^1}: R^{tS^1}$
	form an adjunction in $\kb\bmod$.
\end{lem}

\begin{proof}
	It is known that $S^1$-invariants are a functor $(-)^{S^1}: S^1\bmod \to k[\b]\bmod$,
	and basechange is a functor $- \otimes_{k[\b]} \kb: k[\b]\bmod \to \kb\bmod$.
	This implies that an adjunction in each of the source categories
	is taken to an adjunction in each of the target categories,
	which proves the claim.
\end{proof}

\begin{cor}
	Suppose that $\scrC, \scrD \in \Fun(BS^1,\bCoh_t)$, and that $L: \scrC \leftrightarrows \scrD: R$ 
	are $S^1$-equivariant adjoint functors which each are left t-exact up to a shift, exact, and which preserve uniformly
	left t-bounded filtered colimits. Then the induced functors 
	$L^{\tT}: \scrC^{\tT} \leftrightarrows \scrD^{\tT}: R^{\tT}$
	are adjoint to one another.
\end{cor}

\begin{proof}
	This follows immediately from \zcref{lem: regularization of adjunction is adjunction}
	and \zcref{lem: naive Tate invariants preserves adjunction}.
\end{proof}

\subsection{Potent $\calD$-modules with support}

\begin{defn}
	Let $X$ be a derived stack. Given a closed embedding $Z \xhookrightarrow{i} X$,
	let $U \xhookrightarrow{j} X$ be the complementary open embedding. We define
	\[\Pot_Z(X) := \fib\left(\Pot(X) \xrightarrow{j^!} \Pot(U)\right)\]
	to be the subcategory of potent $\calD$-modules on $X$ supported on $Z$.
\end{defn}

\begin{lem}
	\label{lem:support-is-formal-completion}
	Let $X$ be an algebraic stack almost of finite type and with affine diagonal, and let $Z \hook X$ be a closed substack such that $Z\hook X$ is almost of finite presentation.
	Let $\wih{i}: \wih{X_Z} \to X$ denote induced map from the formal completion to $X$.
	Then $!$-pullback along $\wih{i}$ induces an equivalence,
	 \[i^!: \Pot_Z(X) \to \Pot(\wih{X_Z}),\]
	 with adjoint inverse $\wih{i}_*$.
\end{lem}

\begin{proof}
	We will deduce the result from the corresponding statement
	for ind-coherent sheaves.
	By \cite[Theorem 5.2.8]{PreygelT}, it is know that $!$-pullback along
	the induced map $\calL\wih{i}: \wih{\calL X_{\calL Z}} \hookrightarrow \calL X$
	induces an equivalence,
	\[
		(\calL\wih{i})^!: \mathsf{IndCoh}(\calL X)_{\calL Z}\xrightarrow{\simeq} \mathsf{IndCoh}\left(\wih{\calL X_{\calL Z}}\right).
	\]
	By \zcref{lem:loops-commute-with-formal-completion}, there is an $S^1$-equivariant isomorphism,
	$\calL(\wih{X_Z}) \simeq \wih{\calL X_{\calL Z}}$,
	implying that 
	\[\mathsf{IndCoh}\left(\wih{\calL X_{\calL Z}}\right)^{\tT} \simeq \Pot(\wih{X_Z}).\]
	It therefore suffices to show that
	$\left(\mathsf{IndCoh}(\calL X)_{\calL Z}\right)^{\tT} \simeq \fib(j^!_{\pot})$,
	where we use $j^!_{\pot}$ to distinguish the potent $\calD$-module
	pullback from the $\IndCoh$ pullback.
	
	Let $j: U \hook X$ denote the complement to $i: Z \hook X$.
	By \zcref{cor: loops of complement}, the inclusions
	$\calL Z \xhookrightarrow{Li} \calL X \xhookleftarrow{\calL j} \calL U$
	give a decomposition of $\calL X$ into a complementary pair of open and closed substacks.
	The inclusion $\calL j$ is canonically $S^1$-equivariant, so $\calL j^!$ is an $S^1$-equivariant functor,
	meaning we have an $S^1$-equivariant fiber sequence
	\[
		\fib\left(\calL j^!\right) \to \mathsf{IndCoh}(\calL X) \xrightarrow{\calL j^!} \mathsf{IndCoh}(\calL U).
	\]
	By \cite[Proposition 2.3.3]{BeraldoDiffOp}, the compact objects of $\mathsf{IndCoh}(\calL X)$ 
	(resp. $\mathsf{IndCoh}(\calL U)$) are given by $\mathsf{Coh}(\calL X)$ 
	(resp. $\mathsf{Coh}(\calL U)$). Since $\calL j$ is an open immersion,
	$j^! \simeq j^*$. Thus, $\calL j^!$ has right adjoint $\calL j_*$, which is automatically
	continuous, implying that $\calL j^!$ preserves compact objects. 
	We therefore obtain a fiber sequence of small categories:
	\[
		\fib\left(\calL j^!|_{\Coh(\calL X)}\right)\to \mathsf{Coh}(\calL X) \xrightarrow{\calL j^!|_{\Coh(\calL X)}} \mathsf{Coh}(\calL U).
	\]
	By \cite[Proposition 2.5.2]{BeraldoDiffOp}, for any closed immersion of laft quasi-compact stacks $\frakY \hook \frakX$,
	the image of $\Coh(\frakY)$ under pushforward along the induced map $\fraki: \frakY \to \wih{\frakX_{\frakY}}$
	generates $\IndCoh(\wih{\frakX_{\frakY}})$. At the same time, we know from \cite[Theorem 5.2.8]{PreygelT}
	that $\IndCoh(\wih{\frakX_{\frakY}}) \simeq \IndCoh_{\frakY}(\frakX)$.
	Hence, the proof of \cite[Chapter II.1 Proposition 6.1.3(d)]{GR17}
	goes through in our setting, showing that
	\[\IndCoh_{\calL Z}(\calL X) := \fib\left(\calL j^!\right) \simeq \Ind\left(\fib\left(\calL j^!_{\Coh(\calL X)}\right)\right) =: \Ind(\Coh_{\calL Z}(\calL X)).\]
Thus, we have a localization sequence of objects of $\Fun(BS^1,\bCoh_t)$,
\begin{align*}
	\mathsf{IndCoh}(\mathcal{L}(\wih{X_Z}))\xleftarrow{\simeq}\mathsf{IndCoh}_{\mathcal{L}Z}(\mathcal{L}X)\to \mathsf{IndCoh}(\mathcal{L}X)\to \mathsf{IndCoh}(\mathcal{L}U).
\end{align*} 
Moreover, the right adjoints in this sequence are $t$-bounded by \zcref{lem: loops of closed immersion is closed immersion}, \zcref{lem: open immersion pushfoward is t-bounded} and \zcref{lem: pullback along formal completion inclusion is t-bounded}. Thus, by \zcref{thm: tTate preserves localization sequences}, we obtain a localization sequence after applying $(-)^{\tT}$; hence,
\begin{align*}
	\Pot_Z(X)\xrightarrow{\simeq}\Pot(\wih{X_Z})
\end{align*} 
as desired.
\end{proof}

\begin{cor}
\label{cor: support is formal completion for L^}
	In the setting of \zcref{lem:support-is-formal-completion}, $!$-pullback along $\wih{i}$ induces an equivalence
\begin{align*}
	\mathsf{IndCoh}_{\hL Z}(\hL X)^{\tT}\xrightarrow{\simeq} \mathsf{IndCoh}(\hL (\wih{X_Z}))^{\tT}
\end{align*} 
\end{cor} 
\begin{proof}
	We will show that the proof of \zcref{lem:support-is-formal-completion} applies to the map $\hL \wih{i}$. Indeed, to obtain the desired localization sequence, it is sufficient, by \zcref{lem: formal loops of open complement}, to show that the natural map
	\begin{align}\label{eq: formal completion commutes}
		\hL(\wih{X_Z})\to \wih{(\hL X)_{\hL Z}}
	\end{align} 
	is an isomorphism and that the ind-coherent sheaf categories in question are compactly generated. We have that \zcref[noname]{eq: formal completion commutes} is an isomorphism by \zcref{lem: formal loops of open complement}. We now seek to apply \cite[Proposition 5.30]{CautisWilliams} to show that all categories in question are compactly generated by $\Coh$. For this, note that by nil-invariance of $(-)_{\text{dR}}$, we have canonical isomorphisms
    \begin{align*}
	    \wih{\hL X_{\hL Z}}\simeq&\calL X\times_{(\calL X)_{\text{dR}}}X_{\text{dR}}\times _{X_{\text{dR}}}Z_{\text{dR}}\\
	    \simeq&\mathcal{L}X\times _{(\mathcal{L}X)_{\text{dR}}}Z_{\text{dR}}\\
	    \eqcolon &\wih{(\mathcal{L}X)_Z}.
    \end{align*}
    Since $\mathcal{L}X$ is an algebraic stack almost of finite type and with affine diagonal by assumption, we see that by \zcref{lem: formal completions of aft algebraic stacks are coherent ind-geometric} $\wih{(\mathcal{L}X)_Z}$ is coherent ind-geometric, so \emph{loc. cit.} applies to it. For the stacks $\hL X$ and $\hL U$, we again use \zcref{lem: formal completions of aft algebraic stacks are coherent ind-geometric}. Thus, we obtain a localization sequence of objects of $\Fun(BS^1,\bCoh_t)$
    \begin{align*}
	    \mathsf{IndCoh}(\hL(\wih{X_Z})\xleftarrow{\simeq}\mathsf{IndCoh}_{\hL Z}(\hL X)\to \mathsf{IndCoh}(\hL X)\to \mathsf{IndCoh}(\hL U).	
    \end{align*} 
    It remains to check that the right adjoints in this localization sequence are $t$-bounded. For the map $\hL\wih{i}:\hL(\wih{X_Z})\to \hL X$, we use the natural equivalence
    \begin{align*}
	    \hL(\wih{X_Z})\simeq\wih{(\mathcal{L}X)_Z},
    \end{align*} 
    so that by \zcref{lem: pullback along formal completion inclusion is t-bounded}, the $!$-pullback along $\hL(\wih{X_Z})\to \mathcal{L}X$ is $t$-bounded. Now, the map $\hL(\wih{X_Z})\to \mathcal{L}X$ factors as
    \begin{align*}
	    \hL(\wih{X_Z})\xrightarrow{\hL\wih{i}} \hL X\xrightarrow{\iota_X}\mathcal{L}X,
    \end{align*} 
    and the $!$-pullback along $\hL X\to \mathcal{L}X$ is $t$-bounded by \emph{loc. cit.}. Moreover, by \cite[Theorem 5.2.8]{PreygelT}, the $*$-pushforward along $\iota_X:\hL X \to \mathcal{L}X$ is fully faithful and $t$-exact. Thus, the unit $\id\xrightarrow{\simeq}\wih{i}^!\wih{i}_*$ is an isomorphism. We can thus write 
    \begin{align*}
	    (\hL\wih{i})^!\simeq(\hL\wih{i})^!\iota_X^!\iota_{X*}\simeq(\iota_X\circ\hL\wih{i})^!\iota_{X*}.
    \end{align*} 
    The right-hand side is $t$-bounded; hence, we conclude that $(\hL\wih{i})^!$ is $t$-bounded as well. 
    Now, we will show that $*$-pushforward along $\hL j:\hL U\to \hL X$ is also $t$-bounded. Consider the Cartesian square
    \[
    \begin{tikzcd}
	    \hL U\arrow[r,"\hL j"]\arrow[d,"\iota_U"]&\hL X\arrow[d,"\iota_X"]\\
	    \mathcal{L}U\arrow[r,"\mathcal{L}j"]&\mathcal{L}X
    \end{tikzcd}.
    \] 
    Base change along this square shows that
    \begin{align*}
    	\iota_X^!(\mathcal{L}j)_*\simeq(\hL j)_*\iota_U^!.
    \end{align*} 
    Again, by \cite[Theorem 5.2.8]{PreygelT}, $\iota_{U*}$ is fully faithful and $t$-exact, so $\iota_U^!\iota_{U*}\simeq\id$. Thus,
    \begin{align*}
	    (\hL j)_*\simeq \iota_X^!(\mathcal{L}j)_*\iota_{U*}.
    \end{align*} 
    Now, all of the functors on the right-hand side are $t$-bounded. Indeed, $(\mathcal{L}j)_*$ is $t$-bounded by \zcref{lem: open immersion pushfoward is t-bounded}, and $i_X^!$ is $t$-bounded by \zcref{lem: pullback along formal completion inclusion is t-bounded}. This allows the same method of proof as in \zcref{lem:support-is-formal-completion}.
\end{proof}

\section{Nil-invariance of cyclic sheaves on the $\a$-potent loop space}
\label{sec:nil-invariance of cyclic sheaves on the a-unipotent loop space}

\subsection{$\a$-potent loop space}
The free loop space of a quotient by a reductive group admits a 
kind of ``Jordan decomposition,"
described in \cite[\S 3]{ChenEquiv},
that allows one to work over equivalence classes 
of semisimple elements in $G/\!\!/G$ one at a time.
Namely, given a global quotient $X/G$, with $G$ a
reductive group, there is a canonical map,
\[\mu: \calL(X/G) \to \calL(BG) \simeq G/G \to G/\!\!/G,\]
where $G/\!\!/G$ is the affinization of $G/G$---a variety
parametrizing semisimple conjugacy classes of $G$.
For each semisimple conjugacy class $\a \in G/\!\!/G$,\footnotemark
\footnotetext{In \cite{BZChenHelmNadler}, $\a$ denotes
a semisimple element in $G$ whose conjugacy class is denoted
$[\a]$. Since the definitions therein indexed by $\a$
only depend on the equivalence class $[\a]$, 
we choose in this paper to use $\a$ to denote the
conjugacy class rather than a representative thereof.}
one defines the $\a$-potent loop space
to be the formal completion of $L(X/G)$ along the fiber
of $\mu: L(X/G) \to G/\!\!/G$ over $\a$:
\[\calL^{\alpha}(X/G) := \wih{\calL(X/G)_{\mu^{-1}(\a)}}.\]
From definition of the formal completion, there is a canonical map
$i_{\a}: \calL^{\a}(X/G) \to \calL(X/G)$.

We have the following analog of \zcref{lem:loops-commute-with-formal-completion}. 
\begin{lem}
	\label{lem:unipotent-loops-commute-with-formal-completion}
	Let $\calX \simeq X/G$ be a global quotient stack with $G$ a linear group with 
	reductive neutral component acting on a prestack $X$.
	Let $\calZ \hook \calX$ be a closed immersion with complement
    $\calU \hook \calX$. Then the induced map
    $\calL^{\a}(\calZ) \to \calL^{\a}(\calX)$ is a closed immersion
    whose open complement is $\calL^{\a}(\calU) \to \calL^{\a}(\calX)$.
    Moreover, we have an equivalence,
	\begin{align*}
		\calL^{\alpha}\left(\wih{\calX_{\calZ}}\right) \simeq \wih{(\calL^{\alpha} \calX)_{\calL^{\alpha} \calZ}},
	\end{align*} 
	for each $\alpha \in G/\!\!/G$.
\end{lem}

\begin{proof}
	By \zcref{lem: formal completion is quotient}, $\wih{X_Z}/G \xrightarrow{\simeq} \wih{\calX_{\calZ}}$.
	Denote by $\bbO_\alpha(\wih{\calX_{\calZ}})$ (resp. $\bbO_\alpha(\calZ)$, $\bbO_\alpha(\calX)$, $\bbO_\alpha(\mathcal{U})$) 
	the derived fiber of $\mu_{\wih{\calX_{\calZ}}}: \calL(\wih{\calX_{\calZ}}) \simeq \calL(\wih{X_Z}/G) \to G/\!\!/G$ over $\a \in G/\!\!/G$ (resp. the derived fibers of $\mu_{\calX}$, $\mu_{\calZ}$, and $\mu_{\mathcal{U}}$ over
	$\a \in G/\!\!/G$). 
    
    Let $S\to \calL^{\alpha} \mathcal{X}$ be an $S$-point. We will first show that 
    \begin{align*}
    	\calL^{\alpha} \mathcal{U}\xrightarrow{\simeq} \calL^{\alpha} \mathcal{X}\times _{\mathcal{LX}}\mathcal{LU}.
    \end{align*} 
    Indeed, 
    \begin{align}
	    \calL^{\alpha} \mathcal{U}\xrightarrow{\simeq}&\mathcal{LU}\times _{(\mathcal{LU})_{\text{dR}}}\left( \mathcal{LU}\times _{\mathcal{LX}}\bbO_\alpha(\mathcal{X}) \right) _{\text{dR}}\label{eq: orbit of open complement}\\
	    \xrightarrow{\simeq}&\mathcal{LU}\times _{(\mathcal{LX})_{\text{dR}}}\bbO_\alpha(\mathcal{X})\label{eq: de rham of orbit of open complement}\\
	    \xrightarrow{\simeq}&\mathcal{LU}\times _{\mathcal{LX}}\mathcal{LX}\times _{(\mathcal{LX})_{\text{dR}}}\bbO_\alpha(\mathcal{X})\nonumber\\
	    \eqcolon &\mathcal{LU}\times _{\mathcal{LX}}\calL^{\alpha} \mathcal{X},\nonumber
    \end{align} 
    where in \zcref[noname]{eq: orbit of open complement}, we use the Cartesian diagram \[
    \begin{tikzcd}
		\bbO_\alpha(\mathcal{U})\arrow[r]\arrow[d]&\mathcal{LU}\arrow[d]\\
		\bbO_\alpha(\mathcal{X})\arrow[r]&\mathcal{LX}
	\end{tikzcd}
    \]
    and in \zcref[noname]{eq: de rham of orbit of open complement} we use that $(-)_{\text{dR}}$ preserves limits by \cite[Lemma 1.1.4]{GRCrystals}
    Thus, to see that $\calL^{\alpha}\mathcal{U}$ 
    is the open complement to $\calL^{\alpha}\mathcal{Z}$, it is sufficient to show that 
	\begin{align*}
		S\times _{\calL^{\alpha} \mathcal{X}}\calL^{\alpha} \mathcal{Z}\simeq \emptyset\iff S\times _{\mathcal{LX}}\mathcal{LZ}\simeq \emptyset.
	\end{align*} 
	In fact, we claim that
	\begin{align*}
		S\times _{\calL^{\alpha} \mathcal{X}}\calL^{\alpha} \mathcal{Z}\simeq S\times _{\mathcal{LX}}\mathcal{LZ}.
	\end{align*} 
	Indeed, by using the Cartesian square
	\begin{equation}\label{cd: cartesian square for a-orbits}
	\begin{tikzcd}
		\bbO_\alpha(\mathcal{Z})\arrow[r]\arrow[d]&\mathcal{LZ}\arrow[d]\\
		\bbO_\alpha(\mathcal{X})\arrow[r]&\mathcal{LX}
	\end{tikzcd},
\end{equation}
	we see that
	\begin{align*}
		S\times _{\calL^{\alpha} \mathcal{X}}\calL^{\alpha} \mathcal{Z}\coloneq& S\times _{\calL^{\alpha} \mathcal{X}}\mathcal{LZ}\times _{(\mathcal{LZ})_{\text{dR}}}\bbO_\alpha(\mathcal{Z})_{\text{dR}}\\
		\xrightarrow{\simeq}&S\times _{\calL^{\alpha} \mathcal{X}}\mathcal{LZ}\times _{(\mathcal{LZ})_{\text{dR}}}(\mathcal{LZ})_{\text{dR}}\times _{(\mathcal{LX})_{\text{dR}}}\bbO_\alpha(\mathcal{X})_{\text{dR}}\\
		\xrightarrow{\simeq}&S\times _{\calL^{\alpha} \mathcal{X}}\bbO_\alpha(\mathcal{X})_{\text{dR}}\times _{(\mathcal{LX})_{\text{dR}}}\mathcal{LZ}\\
		\xleftarrow{\simeq}&S\times _{\calL^{\alpha} \mathcal{X}}\calL^{\alpha} \mathcal{X}\times _{\mathcal{LX}}\mathcal{LZ}\\
		\xrightarrow{\simeq}&S\times _{\mathcal{LX}}\mathcal{LZ}.
	\end{align*}

	Unraveling the definitions, we have,
	\begin{align*} 
		\calL^u_\a \calX \times _{(\calL^u_\a \calX)_{\dR}}(\calL^u_\a \calZ)_{\dR} &\simeq \left(\calL\calX \times_{(\calL\calX)_{\dR}} \bbO_{\a}(\calX)_{\dR}\right) \times_{\left(\calL\calX \times_{(\calL\calX)_{\dR}} \bbO_{\a}(\calX)_{\dR}\right)_{\dR}} \left(\calL\calX \times_{(\calL\calZ)_{\dR}} \bbO_{\a}(\calZ)_{\dR}\right)_{\dR} \\
		&\xrightarrow{\simeq} \left(\calL\calX \times_{(\calL\calX)_{\dR}} \bbO_{\a}(\calX)_{\dR}\right) \times_{(\calL\calX)_{\dR} \times_{(\calL\calX)_{\dR}} \bbO_{\a}(\calX)_{\dR}} \left((\calL\calZ)_{\dR} \times_{(\calL\calZ)_{\dR}} \bbO_{\a}(\calZ)_{\dR}\right) \\
		&\xrightarrow{\simeq} \left(\calL\calX \times_{(\calL\calX)_{\dR}} (\calL\calZ)_{\dR}\right) \times_{(\calL\calX)_{\dR}  \times_{(\calL\calX)_{\dR}} (\calL\calZ)_{\dR}} \left(\bbO_{\a}(\calX)_{\dR} \times_{\bbO_{\a}(\calX)_{\dR}} \bbO_{\a}(\calZ)_{\dR}\right) \\
		&\xleftarrow{\simeq} \left(\calL\calX \times_{(\calL\calX)_{\dR}} (\calL\calZ)_{\dR}\right) \times_{\left(\calL\calX  \times_{(\calL\calX)_{\dR}} (\calL\calZ)_{\dR}\right)_{\dR}} \left(\bbO_{\a}(\calX) \times_{\bbO_{\a}(\calX)_{\dR}} \bbO_{\a}(\calZ)_{\dR}\right)_{\dR}\\
		&=: \wih{\calL\calX_{\calL\calZ}} \times_{\left(\wih{\calL\calX_{\calL\calZ}}\right)_{\dR}} \wih{\bbO_{\a}(\calX)_{\bbO_{\a}(\calZ)}}_{\dR}.
	\end{align*} 
	On the other hand, by \zcref{lem:loops-commute-with-formal-completion}, 
	$\calL(\wih{\calX_{\calZ}}) \simeq \wih{\calL\calX_{\calL\calZ}}$, 
	so we have
	\[
	\wih{(\calL^{\alpha} \calX)_{\calL^{\alpha} \calZ}} \simeq 
	\wih{\calL\calX_{\calL\calZ}} \times_{\left(\wih{\calL\calX_{\calL\calZ}}\right)_{\dR}} \wih{\bbO_{\a}(\calX)_{\bbO_{\a}(\calZ)}}_{\dR}
	\simeq \calL(\wih{\calX_{\calZ}}) \times_{\calL(\wih{\calX_{\calZ}})_{\dR}} \wih{\bbO_{\a}(\calX)_{\bbO_{\a}(\calZ)}}_{\dR}.
	\]
	It therefore suffices to show that $\wih{\bbO_{\a}(\calX)_{\bbO_{\a}(\calZ)}} \simeq \bbO_{\a}(\wih{\calX_{\calZ}})$. This follows from
	applying \zcref{lem: formal completion commutes with fiber products} to the Cartesian square \zcref[noname]{cd: cartesian square for a-orbits}
	to obtain $\wih{\bbO_{\a}(\calX)_{\bbO_{\a}(\calZ)}} \simeq \wih{\calX_{\calZ}} \times_{\calX} \bbO_{\a}(\calX)$,
	and then observing that the latter objects is exactly $\bbO_{\a}(\wih{\calX_{\calZ}})$ since
	$\mu_{\wih{\calX_{\calZ}}}: \wih{\calX_{\calZ}} \to G/\!\!/G$ factors as $\wih{\calX_{\calZ}} \to \calX \xrightarrow{\mu_{\calX}} G/\!\!/G$.
\end{proof}

\begin{cor}\label{cor: support is formal completion for L^u}
	Suppose in the setting of \zcref{lem:support-is-formal-completion} that $X =U /G$ is a global quotient stack where $U$ is a derived scheme, $G$ is an algebraic group with reductive neutral component, and we denote $\wih{i}: \wih{X_Z}\to X$. Then $!$-pullback along $\calL^u\wih{i}$ induces an equivalence
	\begin{align*}
	\mathsf{IndCoh}_{\mathcal{L}^uZ}(\mathcal{L}^uX)^{\tT}\xrightarrow{\simeq}\mathsf{IndCoh}(\mathcal{L}^u(\wih{X_Z}))^{\tT}.
\end{align*}
\end{cor} 
\begin{proof}
	We have natural isomorphisms
	\begin{align*}
		\mathcal{L}^u(\wih{X_Z})\simeq&\wih{(\mathcal{L}^uX)_{\mathcal{L}^uZ}}\\
		\simeq&\mathcal{L}X\times _{(\mathcal{L}X)_{\text{dR}}}\bbO_{[1]}(Z)_{\text{dR}}\\
		\coloneq &\wih{(\mathcal{L}X)_{\bbO_{[1]}(Z)}},
	\end{align*} 
	where we use nil-invariance of $(-)_{\text{dR}}$ and the equivalence $\mathcal{L}^u\simeq\mathcal{L}^u_{[1]}$ in the second line. Thus, the proof of \zcref{cor: support is formal completion for L^} applies.
\end{proof}

\subsection{$A$-trivializations}
We recall the notions of $A$-trivialization for prestacks and
categories equipped with an action of a group prestack.

\begin{defn}(\cite[Definition A.16]{BZChenHelmNadler})
	Let $G$ be a linear algebraic group with reductive 
	neutral component acting on a prestack $X$, and let $A \subset G$ be a normal subgroup 
	with $G'=G/A$. An \emph{$A$-trivialization} (of the $G$-action on $X$) 
	is a $G'$-action on $X$ and an equivalence
	\begin{align*}
		X /G \simeq X/G' \times_{BG'} BG.
	\end{align*} 
	When $A =\langle \alpha \rangle$, for $\alpha \in Z(G)$ a central element, we call an $A$-trivialization 
	an \emph{$\alpha$-trivialization}.
\end{defn}

\begin{lem}
	\label{lem: loop space has A-trivialization}
	Let $G$ be an affine algebraic group, and let $A \subset G$ be 
	a closed central subgroup.\footnote{This is to ensure that $G/A$ is an affine algebraic group as well.} Denote the quotient $G' = G/A$. Suppose that $X$ is a derived stack 
	equipped with a $G$-action and an $A$-trivialization $X/G \simeq X/G' \times_{BG'} BG$. Then $\mathcal{L}^u(X/G)$
	has a canonical $A$-trivialization coming from that on $X/G$.
\end{lem}

\begin{proof}
	Since $G' = G/A$ is the quotient of $G$ by a central subgroup, there is
	an equivalence $G/G \simeq G/G' \times_{BG'} BG$, which gives an $A$-trivialization of 
	$\calL (BG) \simeq G/G$. This $A$-trivialization induces an $A$-trivialization
	of $\calL (X/G)$ as follows: 
	\begin{align*}
		\left(X \times_{X \times X} (X \times G)\right)/G &\xrightarrow{\simeq} (X/G) \times_{X/G \times_{BG} X/G} (X/G \times_{BG} G/G) \\																		
		&\simeq \left(X/G' \underset{BG'}{\times} BG\right) \underset{\left(X/G' \underset{BG'}{\times} BG\right) \underset{BG}{\times} \left(X/G' \underset{BG'}{\times} BG\right)}{\times} \left(X/G' \underset{BG'}{\times} BG\right) \underset{BG}{\times} \left(G/G' \underset{BG'}{\times} BG\right) \\
		&\xrightarrow{\simeq} (X/G' \times_{BG'} BG) \times_{(X/G' \times_{BG'} X/G') \times_{BG'} BG} \left((X/G' \times_{BG'} \times G/G') \times_{BG'} BG\right) \\
		&\xrightarrow{\simeq} \left(X/G' \times_{XG' \times_{BG'} X/G'} X/G' \times_{BG'} G/G'\right) \times_{BG'} BG \\
		&\xleftarrow{\simeq} \left(X \times_{X \times X} X \times G\right)/G' \times_{BG'} BG.
	\end{align*}
	
	Let $U \subset G$ denote the unipotent cone of $G$.
	There is a similar $A$-trivialization on $U/G$ from
	the equivalence $U/G \simeq U/G' \times_{BG'} BG$.
	Combining this with the $A$-trivialization of $\calL (BG)$ identified above also
	induces an $A$-trivialization of $\calL^u(BG)$:
	\begin{align*}
		\calL (BG) \times_{\calL (BG)_{\dR}} (U/G)_{\dR} &\simeq (G/G' \times_{BG'} BG) \times_{(G/G' \times_{BG'} BG)_{\dR}} (U/G' \times_{BG'} BG)_{\dR} \\
		&\xrightarrow{\simeq} (G/G' \times_{BG'} BG) \times_{(G/G')_{\dR} \times_{(BG')_{\dR}} BG_{\dR}} ((U/G')_{\dR} \times_{BG'_{\dR}} BG_{\dR}) \\
		&\simeq (G/G' \times_{(G/G')_{\dR}} (U/G')_{\dR}) \times_{(BG' \times_{BG'_{\dR}} BG'_{\dR})} (BG \times_{BG_{\dR}} BG_{\dR}) \\ 
		&\simeq (G/G' \times_{(G/G')_{\dR}} (U/G')_{\dR}) \times_{BG'} BG \\
		&\simeq (G/G' \times_{G_{\dR}/G'_{\dR}} U_{\dR}/G'_{\dR}) \times_{BG'} BG	\\
		&\simeq (G \times_{G_{\dR}} U_{\dR})/(G' \times_{G'_{\dR}} G'_{\dR}) \times_{BG'} BG \\
		&\simeq \wih{G_U}/G' \times_{BG'} BG.								
	\end{align*}
	Finally, using the $A$-trivializations on $\calL^u(BG)$, $\calL (BG)$, and $\calL (X/G)$ above, 
	we obtain a canonical $A$-trivialization of $\calL^u(X/G)$. Indeed, applying $\mathcal{L}^u\coloneq \Map(B\mathbb{G} _a,-)$, which preserves limits, to the $A$-trivialization of $X/G$, we obtain
	\begin{align*}
		\mathcal{L}^u(X /G)\simeq& \mathcal{L}^u(X /G')\times _{\mathcal{L}^u(BG')}\mathcal{L}^u(BG)\\
		\simeq&\left( \mathcal{L}^u(X /G')\times _{\mathcal{L}^u(BG')}\wih{G_U}/G' \right) \times _{BG'}BG.
	\end{align*} 
	We claim that this is an $A $-trivialization. Indeed, denoting
	\begin{align*}
		W\coloneq \mathcal{L}^u(X /G)\times _{\mathcal{L}^u(BG')}\wih{G_U}/G',
	\end{align*} 
	we need to show that $W \simeq Y /G'$ and $\mathcal{L}^u(X /G)\simeq Y /G$ for some stack $Y$. We claim this holds for $Y\coloneq W\times _{BG'}\pt$. To see this, observe first that the base change of $W\times _{BG'}BG$ along $\pt \to BG$ gives a $G$-torsor
	\begin{align*}
		W\times _{BG'}\pt \xlongrightarrow{} W\times _{BG'}BG\simeq \mathcal{L}^u(X /G).
	\end{align*} 
	Thus, we see that
	\begin{align*}
		\mathcal{L}^u(X /G)\simeq (W\times _{BG'}\pt) /G.
	\end{align*} 
	On the other hand, $W\times _{BG'}\pt$ is the base change of $W$ along the universal $G'$-torsor $\pt \to BG'$; hence,
	\begin{align*}
		W\simeq (W\times _{BG'}\pt) /G'
	\end{align*} 
	as desired.
\end{proof}

\subsubsection{Canonical $A$-trivialization on the fixed point locus}
	If $A \subset G$ is a central subgroup, the centralizer $G(A)$ clearly
	acts on the fixed point locus $X^A$ of the induced $A$-action on $X$, and
	\cite[Proposition A.17]{BZChenHelmNadler}\footnotemark 
	\footnotetext{Although stated in \cite{BZChenHelmNadler} only for derived schemes, the 
	proof therein is equally valid for $X$ a general prestack.}
	states that $X^A$ possesses a canonical $A$-trivialization.
	In the sequel, it will be important to know that, for a $G$-equivariant
	map of prestacks $f: X \to Y$, the
	formal completion along the induced map on fixed points,
	$f^A: X^A \to Y^A$ also possesses a canonical $A$-trivialization.
	This follows from \cite[Proposition A.17]{BZChenHelmNadler} and
	the following lemma.

\begin{lem}
	\label{lem:fixed-points-commute-with-formal-completion}
	Let $G$ be any group prestack, and let $f: X \to Y$ be any $G$-equivariant map of prestacks. 
	Let $A \subset G$ be a reduced central subgroup.
	Then there is a natural equivalence $(\wih{Y_X})^A \simeq \wih{(Y^A)_{X^A}}$.\footnotemark
	\footnotetext{It is worth nothing that there is no natural morphism either direction
		between $(\wih{Y_X})^A$ and $\wih{(Y^A)_{(X^A)}}$. The equivalence stated is
		obtained by composing a universal map with the inverse of another.}
\end{lem}	

\begin{proof}
	First note that the $G$-equivariant structure on $f: X \to Y$
	restricts to a canonical $A$-equivariant, so we obtain map of fixed points,
	$X^A \to U^A$.

	Recall the Cartesian diagram defined the fixed point locus of $\wih{Y_X}$,
	\begin{equation}
		\label{eqn: fixed points}
		\begin{tikzcd}
			{(\wih{Y_X})^A} & {\MapStk(BA, \wih{Y_X}/A)} \\
			{\{\id_{BA}\}} & {\MapStk(BA,BA)}
			\arrow[from=1-1, to=1-2]
			\arrow[from=1-1, to=2-1]
			\arrow["\lrcorner"{anchor=center, pos=0.125}, draw=none, from=1-1, to=2-2]
			\arrow[from=1-2, to=2-2]
			\arrow[from=2-1, to=2-2].
		\end{tikzcd}
	\end{equation}
	By definition, $\wih{Y_X}/A = (Y \times_{Y_{\dR}} X_{\dR})/A$. Since the
	quotient by the group action of $A$ is a sifted colimit, it commutes
	with finite products, so we obtain $(Y \times_{Y_{\dR}} X_{\dR})/A 
	\xrightarrow{\simeq} Y/A \times_{Y_{\dR}/A} X_{\dR}/A$.
	Mapping stacks preserve limits in the second argument, so we
	therefore have the Cartesian diagram
	\[\begin{tikzcd}
		{\MapStk(BA, \wih{Y_X}/A)} & {\MapStk(BA, Y/A)} \\
		{\MapStk(BA, (X/A)_{\dR})} & {\MapStk(BA, (Y/A)_{\dR})}
		\arrow[from=1-1, to=1-2]
		\arrow[from=1-1, to=2-1]
		\arrow[from=1-2, to=2-2]
		\arrow[from=2-1, to=2-2].
	\end{tikzcd}\]
	Limits commute with one another, so we may evaluate the top-left term of
	the fixed point diagram \labelcref{eqn: fixed points} as the limit
	of the cospan whose terms are the top-left terms of the
	pullback diagrams
	\[
	\begin{tikzcd}
		{S^A} & {\MapStk(BA, S/A)} \\
		{\{\id_{BA}\}} & {\MapStk(BA,BA)}
		\arrow[from=1-1, to=1-2]
		\arrow[from=1-1, to=2-1]
		\arrow["\lrcorner"{anchor=center, pos=0.125}, draw=none, from=1-1, to=2-2]
		\arrow[from=1-2, to=2-2]
		\arrow[from=2-1, to=2-2].
	\end{tikzcd}
	\]
	for $S = Y$, $Y_{\dR}$, and $X_{\dR}$.
	This gives precisely $\wih{Y_X}^A \xrightarrow{\simeq} Y^A \times_{(Y_{\dR})^A} (X_{\dR})^A$.
	
	On the other hand, there is a natural map $\wih{(Y^A)_{(X^A)}} = Y^A \times_{(Y^A)_{\dR}} (X^A)_{\dR}
	\to Y^A \times_{(Y_{\dR})^A} (X_{\dR})^A$ induced by the universal morphism $(S^A)_{\dR} \to (S_{\dR})^A$ that
	exists for any stacks $S$ with $A$-action, so
	it remains to show that $(S^A)_{\dR} \xrightarrow{\simeq} (S_{\dR})^A$. By definition, for a test affine $T$,
	\begin{align*}
		S^A(T)\simeq\lim_{[n]\in\Delta}S(T\times A^n).
	\end{align*} 
	Now, $A$ is reduced; hence, $(T\times A^n)_{\text{red}}\simeq T_{\text{red}}\times A^n$. Thus, we compute
	\begin{align*}
		(S_{\text{dR}})^A(T)\simeq&\lim_{[n]\in\Delta}S_{\text{dR}}(T\times A^n)\\
		\simeq&\lim_{[n]\in\Delta}S(T_{\text{red}}\times A^n)\\
		\simeq&S^A(T_{\text{red}})\\
		\simeq&(S^A)_{\text{dR}}(T).
	\end{align*} 
	These isomorphisms are clearly natural in $T$ and compatible with the group operations and induced action. Hence, we find
	\begin{align*}
		(S^A)_{\text{dR}}\to (S_{\text{dR}})^A
	\end{align*} 
	is an isomorphism as desired.
\end{proof}

In particular, \zcref{lem:fixed-points-commute-with-formal-completion} and 
\cite[Proposition A.17]{BZChenHelmNadler} together imply that the formal 
completion along a map of fixed points has a canonical $A$-trivialization.

\begin{cor}
	\label{lem: canonical a-trivialization of formal completion}
	Let $G$ be a group prestack, and let $X \to Y$ be a $G$-equivariant map of 
	prestacks equipped with $G$-actions. Let $A \subset G$ be a reduced central subgroup.
	Then $\wih{{Y^A}_{X^A}}$ has a canonical $A$-trivialization.
\end{cor}

\subsection{$\a$-shifted action on the unipotent loop space}
In addition to the standard $S^1$-action on $\calL^{\alpha}(-)$
given by loop rotation, there is the so-called $\a$-shifted $S^1$-action,
which we describe below. In order to do so, we first recall the following
definition.

If $X$ is a prestack equipped with an action of a group prestack $G$,
and $\a \in G$ is a central element, then we obtain a canonical $\a$-trivialization,
\begin{equation}
	\label{eqn: canonical a-trivialization}
	X^{\alpha}/G(\alpha) \simeq (X^{\alpha}/G'(\alpha)) \times_{BG'(\alpha)} BG(\alpha),
\end{equation}
as a special case of \cite[Proposition A.17]{BZChenHelmNadler},
where $G(\alpha)$ here denotes the centralizer of $\a$, 
and $G'(\alpha) := G(\alpha)/\langle \alpha \rangle$ denotes the quotient.
The $\a$-shifted loop rotation action on $\calL(X^{\a}/G(\a))$ 
(resp. $\calL^u(X^{\a}/G(\a))$ and $\calL^{\alpha}(X^{\a}/G(\a))$) is defined in the
following construction:

\begin{con}
	\label{con: a-shifted loop rotation}
There is an $S^1$-action on $BG(\alpha)$, denoted by $B(sh_{\a})$, 
given by the delooping of the $\bbZ$-action $sh_{\a}$ on $G(\a)$ given by
multiplication by $\alpha$. This induces an $S^1$-action on $X^{\a}/G(\a)$,
which we denote by $\sigma_{\a}$ via its canonical $\a$-trivialization: $\sigma_{\a}$ is the diagonal
to the $S^1 \times S^1$-action obtained from the trivial action on the first factor
and the $B(sh_{\a})$-action on the the second factor in the 
product on the right-hand side of \zcref[noname]{eqn: canonical a-trivialization}. 
Composing this action $\sigma_{\a}:BS^1 \to \PreStk$ with the loop space functor $\calL(-): \PreStk \to \PreStk$,
we obtain an $S^1$-action on $\calL(X^{\a}/G(\a))$ called
the \emph{$\a$-twisted action} on $\calL(X^{\a}/G(\a))$.
Since the ordinary $S^1$-action $\rho$ on $\calL(X^{\a}/G(\a))$ given by loop rotation
commutes with $\calL(\sigma_{\a})$, the \emph{$\alpha$-shifted loop rotation} action is defined
to be the diagonal to the $S^1\times S^1$-action $\rho \times \calL(\sigma_\alpha)$. We denote
this final $S^1$-action by $\rho(\a)$.

One may also take the functor $\calL^{\alpha}(-): \PreStk_{/BG} \to \PreStk_{/BG}$ in place of the usual loop space functor
in the above construction to obtain the $\a$-shifted loop rotation action on $\calL^{\alpha}(X^{\a}/G(\a))$.
This action is also denoted by $\rho({\a})$.
\end{con}

\subsection{Fixed point and shifted localization}
Let us now specialize to the case when $X$ is a \emph{derived scheme} equipped
with an action of a reductive group $G$.
For a chosen semisimple element $\a \in G$, 
let $X^{\a}$ denote the homotopy fixed points
of the action of $\langle \a \rangle$, and
let $G(\a) \subset G$ denote the centralizer of
$\a$. Clearly $G(\a)$ acts on $X^{\a}$,
and there is a canonical map,
$X^{\a}/G(\a) \to X/G$.
By \cite[Theorem 2.9]{BZChenHelmNadler}, if $X$ is locally Noetherian,
the above map induces an $S^1$-equivariant equivalence on $\a$-potent loop spaces,
\[\ell_{\a}^u: \calL^{\alpha}(X^{\a}/G(\a)) \xrightarrow{\simeq} \calL^{\alpha}(X/G).\]
There is an automorphism of $sh_{\a}: \calL(X^{\a}/G(\a)) \to \calL(X^{\a}/G(\a))$
defined as follows. From the canonical $\a$-trivialization on
$X^{\a}/G(\a)$ we obtain an induced product decomposition
$\calL(X^{\a}/G(\a)) \simeq \calL(X^{\a}/G'(\a)) \times_{G'(\a)/G'(\a)} G(\a)/G(\a)$,
where $G'(\a)$ is the quotient $G(\a)/\langle \a \rangle$.
Let $sh_{\a}$ be the automorphism defined on this product by multiplication
by $\a$ on the $G(\a)/G(\a)$ factor, and the identity on the other factors.
This automorphism restricts to a map,
\[sh_{\a}: \calL^u(X^{\a}/G(\a)) \to \calL^{\alpha}(X^{\a}/G(\a)),\]
which is $S^1$-equivariants for the $\a$-shifted loop rotation action (defined
in \zcref{con: a-shifted loop rotation})
on the source and the ordinary loop rotation action on the right.
The \emph{shifted localization map},
\[s\ell^u_{\a}: \calL^u(X^{\a}/G(\a)) \to \calL^{\alpha}(X/G),\]
defined to be the composition $\ell_{\a}^u\circ sh_{\a} $,
is an $S^1$-equivariant equivalence for the $\rho(\a)$
action on the source and the $\rho$ action on the target 
by \cite[Corollary A.18]{BZChenHelmNadler}.

In what follows, we will require an
extension of the shifted localization results
of \cite{BZChenHelmNadler} to the case
when $X$ is the formal completion of
a map of derived schemes,
which is not a derived scheme itself in general.

\begin{lem}
	\label{lem: shifted localization for formal completions}
	Let $G$ be a linear algebraic group with reductive neutral component
	acting on locally Noetherian derived schemes $X$ and $Y$, and let $\a \in G$
	be a semisimple element. If $f: X \to Y$ is a $G$-equivariant map,
	then the shifted localization map 
	\[s\ell^u_{\a}: \calL^u\left(\left(\wih{Y_X}\right)^{\a}/G(\a)\right) \to \calL^{\alpha}\left(\wih{Y_X}/G\right)\]
	is an equivalence.
\end{lem}

\begin{proof}
	By construction, the shifted localization map $s\ell^u_{\a}: \calL^u(U^{\a}/G(\a)) \to \calL^{\alpha}(U/G)$ is natural along
	maps of derived $G$-schemes $U$. Since $X$ and $Y$ themselves satisfy the hypotheses of \cite[Corollary A.18]{BZChenHelmNadler},
	this means we obtain a commutative diagram of $S^1$-equivariant maps,
	\[\begin{tikzcd}
		{\calL^u(X^{\a}/G(\a))} & {\calL^{\alpha}(X/G)} \\
		{\calL^u(Y^{\a}/G(\a))} & {\calL^{\alpha}(Y/G)}
		\arrow["\simeq", from=1-1, to=1-2]
		\arrow["{\calL^{\alpha}(f^{\a}/G(\a))}"', from=1-1, to=2-1]
		\arrow["{\calL^{\alpha}(f/G)}", from=1-2, to=2-2]
		\arrow["{\simeq }", from=2-1, to=2-2],
	\end{tikzcd}\]
	where $f^{\a}: X^{\a} \to Y^{\a}$ is the induced map on fixed points.
	From this, we deduce that the induced map
	\begin{equation}
		\label{eqn: localization for formal completions}
	\wih{\calL^u(Y^{\a}/G(\a))_{\calL^u(X^{\a}/G(\a))}} \xrightarrow{} \wih{\calL^{\alpha}(Y/G)_{\calL^{\alpha}(X/G)}}
	\end{equation}
	is an equivalence. On the other hand,
	\begin{itemize}
		\item by \zcref{lem:fixed-points-commute-with-formal-completion}, \zcref{lem:unipotent-loops-commute-with-formal-completion}, \zcref{lem: formal completion is quotient}, and the fact that $\calL^u=\calL^u_{[1]}$ is a limit, hence commutes with limits, we see---by a direct fiber product calculation---that
		the left- and right-hand side of \zcref[noname]{eqn: localization for formal completions}
		are equivalent to 
		\[\calL^u\left(\wih{Y^{\a}/G(\a)_{X^{\a}/G(\a)}}\right) \;\;\; \text{and} \;\;\; \calL^{\alpha}\left(\wih{Y/G_{X/G}}\right),\]
		respectively;
		\item by \zcref{lem: formal completion is quotient}, $\wih{Y^{\a}/G(\a)_{X^{\a}/G(\a)}} \simeq \wih{Y^{\a}_{X^{\a}}}/G(\a)$
		and $\wih{Y/G_{X/G}} \simeq \wih{Y_X}/G$;
		\item and by \zcref{lem:fixed-points-commute-with-formal-completion}, $\wih{Y^{\a}_{X^{\a}}}
		\simeq \left(\wih{Y_X}\right)^{\a}$.
	\end{itemize} 
	All together, the equivalence \zcref[noname]{eqn: localization for formal completions} becomes
	the equivalence
	\[s\ell^u_{\a}: \calL^u\left(\wih{Y_X}^{\a}/G(\a)\right) \xrightarrow{\simeq} \calL^{\alpha}\left(\wih{Y_X}/G\right),\]
	as desired
\end{proof}

There is a corresponding notion of $A$-trivialization for $G$-categories, which we recall below.

\begin{defn}
If $G$ is an affine algebraic group and $A\subset Z(G)$ is a central subgroup, 
then an \emph{$A$-trivialization} of a $\mathsf{QCoh}(BG)$-module category 
$\mathscr{C}$ is a category $\mathscr{C}'$ and an equivalence
\begin{align*}
	\mathscr{C}\simeq \mathscr{C}'\otimes_{\mathsf{QCoh}(B(G /A))}\mathsf{QCoh}(BG).
\end{align*} 
Then the category of \emph{$A$-trivial objects} of $\mathscr{C}$ is denoted $\mathscr{C}_A$, 
which is defined to be the essential image in $\mathscr{C}$ of the natural functor 
$\mathscr{C}'\to \mathscr{C}$. When $A=\langle\alpha\rangle$ for $\alpha\in Z(G)$,
we write $\mathscr{C}_\alpha \coloneq \mathscr{C}_A $.
\end{defn}

\begin{rem}
	If $X$ is a prestack with $G$-action, an $A$-trivialization of $X$ induces
	a $A$-trivialization of the $G$-category $\QCoh(X/G)$.
\end{rem}

We now show that Tate invariants with respect to the $\alpha$-twisted loop rotation 
action on $\calL^u(X^\alpha /G(\alpha))$ kills all objects which are not $\alpha$-trivial. 
In fact, this already happens at the level of $S^1$-invariants, as is shown in the proof of the following result.

\begin{prop}
	\label{prop: Tate invariants kill everything except a-trivial block}
	Let $G$ be a reductive affine algebraic group and $X$ a laft prestack with $G $-action, 
	and let $\alpha\in Z(G)$ be a central semisimple element. 
	Suppose further that $X /G$ has an $\alpha$-trivialization and $\mathcal{L}^u(X /G)$ is coherent ind-geometric in the sense of \cite[Definitions 4.1 and 4.3]{CautisWilliams}. Then
	the inclusion of the $\a$-trivial block of $\IndCoh(L^u(X/G))$ induces
	an equivalence,
	\begin{align}
		\tag{1} \label{eq:tate-alpha} \IndCoh(\calL^u(X/G))^{\tT(\alpha)}_\alpha &\xrightarrow{\simeq} \IndCoh(\calL^u(X/G))^{\tT(\alpha)}
	\end{align} 
	where $\tT(\alpha)$ denotes the renormalized 
	Tate invariants with respect to the 
	$\alpha$-twisted loop rotation action.
\end{prop}

\begin{proof}
	We will investigate the $S^1 $-invariants of the category $\Coh(\calL^u (X /G))$.
	Let $X^*(\langle \alpha \rangle)$ denote the character lattice of $\langle \alpha \rangle$. 
	The category $\mathscr{C}=\mathsf{Coh}(\calL^u(X /G))$ has a direct sum decomposition
	\[\mathscr{C} \simeq \bigoplus _{\chi\in X^*(\langle \alpha \rangle)}\mathscr{C}_\chi,\]
	arising from the $\alpha$-trivialization of the stack $\calL^u(X/G)$ coming from
	\zcref{lem: loop space has A-trivialization}.
	The $\alpha$-shifted loop rotation action on $L^u(X/G)$ induces a corresponding
	$S^1$-action on $\mathscr{C}$. Moreover, since any $S^1$-action on $\scrC$
	preserves the decomposition $\mathscr{C} \simeq \oplus _\chi \mathscr{C}_\chi$,\footnotemark
	\footnotetext{Indeed, the $S^1$-action on $\scrC$ restricts to an $S^1$-action
	on \emph{any} full subcategory of $\scrC$ (cf. proof of \zcref{lem: S^1 invariants commutes with heart}).}
	we have
	\begin{align*}
		\mathscr{C}^{S^1(\alpha)} \simeq \bigoplus _{\chi\in X^*(\langle \alpha \rangle)}\mathscr{C}_\chi^{S^1(\alpha)}.
	\end{align*}

	We claim that $\mathscr{C}_\chi^{S^1(\alpha)} \simeq 0$ whenever $\chi(\alpha)\neq 1$. 
	Let $\chi \in X^*(\langle \alpha \rangle)$ be any character such that $\chi(\alpha)\neq 1$. 
	Let $\tilde c\in \mathscr{C}_\chi^{S^1(\alpha)}$ be an $S^1(\alpha)$-equivariant object, and denote by $c\in \mathscr{C}_\chi$ 
	its underlying object. Let $\rho: \bbZ \to \Aut_{\Aut(\scrC)}(\id_{\scrC})$ denote ordinary loop rotation action, and $\rho_c: c \to c$ denote
	the component at $c$ of the natural automorphism $\rho(1)$.
	By \zcref{prop: BG_a-actions are unipotent}, $\rho_c$ is a unipotent operator,
	hence has the form $\rho_c \simeq 1 + N_c$ for some nilpotent operator $N_c \in \Omega^{\infty}(\End_{\scrC}(c))$.
	Furthermore, the restriction to $\scrC_{\chi}$ of the $\alpha$-shifted rotation 
	action on $\scrC$ is given by multiplication by the scalar $\chi(\alpha)$. 
	Thus, the component at $c$ of the natural automorphism of the identity endofunctor given by $\rho(\a)$ is
	\begin{align*}
		\rho(\alpha)_c \simeq \chi(\alpha)\rho_c \simeq \chi(\alpha)(1+N_c).
	\end{align*} 
	The $S^ 1(\alpha)$-equivariant structure on $\tilde c$ supplies a nullhomotopy $1-\rho(\alpha)_c \simeq 0$. 
	We claim, however, that $1-\rho(\alpha)_c$ is invertible. 
	Indeed, let $n \in \mathbb{N} $ be such that $N_c^n \simeq 0$. Then
	\begin{align*}
		(1-\rho(\alpha)_c)^{-1}&\simeq \left( (1-\chi(\alpha)(1+N_c) \right)^{-1}\\
											   &\simeq \frac{1}{1-\chi(\alpha)}\sum_{i=0}^{n-1}\left( \frac{\chi(\alpha)}{1-\chi(\alpha)}N_c \right) ^i.
	\end{align*} 
	Thus, the identity automorphism must be nullhomotopic itself, $\id_c \simeq 0$, which implies
	that $c \simeq 0$. Since the forgetful functor $\mathscr{C}_\chi^{S^1(\alpha)}\to \mathscr{C}_\chi$ is conservative,\footnotemark
	\footnotetext{cf. the proof of \zcref{lem:fully-faithful-functors-are-ff-after-tate}. 
	It is shown there that an $S^1$-equivariant map of
	anima with $S^1$-actions is an natural equivalence if and only if
	the underlying map on objects is an isomorphism.} we conclude that $\tilde c \simeq 0$. Thus
	$\mathscr{C}_\chi^{S^1(\alpha)}\cong 0$
	for $\chi(\alpha)\neq 1$. To complete the proof of
	equivalence \zcref[noname]{eq:tate-alpha}, it suffices to observe 
	that the $\alpha$-trivial block of $\mathscr{C}$ is 
	precisely the summand $\mathscr{C}_{\chi_1}$ for the
	character $\chi_1$ such that $\chi_1(\alpha)=1$.
\end{proof}

\subsection{Nil-invariance}
We now prove the main result of this section: that the category of cyclic sheaves on the $\a$-potent
loop space of a global quotient stack is invariant under nil-isomorphism. Such a result
follows from a form of Koszul duality for cyclic sheaves appearing in the future work \cite{BZCHN2},
although our proof does not rely on this, instead utilizing the fixed point localization results of \cite{Preygel}. 
We do however utilize the following theorem of \cite{BZCHN2}.
\begin{thm}\label{thm: BZCHN Tate sheaves on unipotent equals formal}
	Let $X /G$ be a quotient stack of a derived scheme of finite type by an affine algebraic group. Then the canonical inclusion $i_X:\wih{\mathcal{L}}(X /G)\to \mathcal{L}^u(X /G)$ induces a pair of adjoint equivalences
	\[
		\begin{tikzcd}[column sep= large]
		\mathsf{IndCoh}(\wih{\mathcal{L}}(X /G))^{\tT}\arrow[r,"i_{X*}^{\tT}",shift left]&\mathsf{IndCoh}(\mathcal{L}^u(X /G))^{\tT}\arrow[l,shift left, "(i_X^!)^{\tT}"].
	\end{tikzcd}
	\] 
\end{thm} 

For use in the below proposition, we record the following corollary.
\begin{cor}\label{cor: Tate sheaves on formal unipotent equals formal}
	In the setup of \zcref{thm: BZCHN Tate sheaves on unipotent equals formal}, suppose that $Z\to X$ is a $G$-invariant closed subscheme and $G$ has a reductive neutral component. If $Z\to X$ is almost finitely presented and $X$ has affine diagonal, then the canonical inclusion $\wih{i}:\wih{\mathcal{L}	}(\wih{X_Z}/G)\to \mathcal{L}^u(\wih{X_Z}/G)$ induces a pair of adjoint equivalences
\[
\begin{tikzcd}[column sep= large]
	\mathsf{IndCoh}(\wih{\mathcal{L}}(\wih{X_Z}/G))^{\tT}\arrow[r,"\wih{i}_*^{\tT}",shift left]&\mathsf{IndCoh}(\mathcal{L}^u(\wih{X_Z}/G))^{\tT}\arrow[l,shift left, "(\wih{i}^!)^{\tT}"].
\end{tikzcd}
\] 
\end{cor} 
\begin{proof}
	For ease of notation, we drop the $(-)^{\tT}$ when denoting the relevant pullbacks or pushforwards; for example, we write $\wih{i}^!$ for $(\wih{i}^!)^{\tT}$. Denote the open complement to $Z$ in $X$ by $U$, and let $j:U /G\hookrightarrow X /G$ and $\iota: Z /G\hookrightarrow X /G$ denote the inclusions. We claim that we have a commutative diagram
	\begin{equation}\label{cd: open complement in formal and unipotent loops}
	\begin{tikzcd}
	\wih{\mathcal{L}}(U /G)\arrow[r,hook,"\wih{j}"]\arrow[d,"i_U"]&\wih{\mathcal{L}}(X /G)\arrow[d,"i_X"]\\
	\mathcal{L}^u(U /G)\arrow[r,"j^u",hook]&\mathcal{L}^u(X /G)
	\end{tikzcd},
\end{equation}
	where $ \wih{j}$ (resp. $j^u$) denotes $\wih{\mathcal{L}}j$ (resp. $\mathcal{L}^uj$). Indeed, by \zcref{lem: formal loops of open complement} (resp. \zcref{lem:unipotent-loops-commute-with-formal-completion} with $\alpha=1$), the open complement to $\hL (Z /G)$ in $\hL (X /G)$ (resp. to $\mathcal{L}^u (Z /G)$ in $\mathcal{L}^u(X /G)$) is $\hL (U /G)$ (resp. $\mathcal{L}^u(U /G)$). Applying $\mathsf{IndCoh}(-)^{\tT}$ to \zcref[noname]{cd: open complement in formal and unipotent loops}, we obtain a commutative diagram
	\[
	\begin{tikzcd}
		\mathsf{IndCoh}(\mathcal{L}^u(X /G))^{\tT}\arrow[r,"i_X^!"]\arrow[d,"(j^u)^!"]&\mathsf{IndCoh}(\hL (X /G))^{\tT}\arrow[d,"\wih{j}^!"]\\
		\mathsf{IndCoh}(\mathcal{L}^u(U /G))^{\tT}\arrow[r,"i_U^!"]&\mathsf{IndCoh}(\hL (U /G))^{\tT}
	\end{tikzcd}.
	\] 
	Thus, we see that if $\mathcal{F}\in \mathsf{IndCoh}_{\mathcal{L}^u(Z /G)}(\mathcal{L}^u(X /G))^{\tT}$, then $i_X^!\mathcal{F}\in \mathsf{IndCoh}_{\hL (Z /G)}(\hL (X /G))^{\tT}$. This implies that the equivalence $i_X^!$ of \zcref{thm: BZCHN Tate sheaves on unipotent equals formal} restricts to a fully faithful functor
	\begin{align*}
 \mathsf{IndCoh}_{\mathcal{L}^u(Z /G)}(\mathcal{L}^u(X /G))^{\tT}\hook\mathsf{IndCoh}_{\hL (Z /G)}(\hL (X /G))^{\tT}.
	\end{align*} 
	We claim this functor is an equivalence. To see this, let $\mathcal{G}\in \mathsf{IndCoh}_{\hL (Z /G)}(\hL (X /G))^{\tT}$. By \zcref{thm: BZCHN Tate sheaves on unipotent equals formal}, there is an $\mathcal{F}\in \mathsf{IndCoh}(\mathcal{L}^u(X /G))^{\tT}$ such that $i_X^! \mathcal{F}=\mathcal{G}.$
	By definition of the category with support, $\wih{j}^! \mathcal{G}\simeq 0.$
	Thus, 
	\begin{align*}
		i_U^!(j^u)^!\mathcal{F}\simeq&\wih{j}^!i_X^!\mathcal{F}\\
		\simeq& 0.
	\end{align*} 
	Since $i_U^!$ is also an equivalence by \zcref{thm: BZCHN Tate sheaves on unipotent equals formal} applied to $U /G$, we find that
		$(j^u)^!\mathcal{F}\simeq 0,$
	so that $\mathcal{F}\in \mathsf{IndCoh}_{\mathcal{L}^u(Z /G)}(\mathcal{L}^u(X /G))^{\tT}$. Thus, we have shown that $i_X^!$ restricts to an equivalence
	\begin{align}\label{eq: support equivalence for unipotent and formal}
 \mathsf{IndCoh}_{\mathcal{L}^u(Z /G)}(\mathcal{L}^u(X /G))^{\tT}\xrightarrow{\simeq}\mathsf{IndCoh}_{\hL (Z /G)}(\hL (X /G))^{\tT}.
\end{align} 
We next claim that its left adjoint $i_{X*}$ restricts to the inverse of this equivalence. For this, it is sufficient to show that $i_{X*}$ restricts to a functor between the categories with support. Thus, let $\mathcal{G}\in \mathsf{IndCoh}_{\hL (Z /G)}(\hL (X /G))^{\tT}$ as before. A similar argument to that in the proof of \zcref{lem: formal loops of open complement} shows that the diagram \zcref[noname]{cd: open complement in formal and unipotent loops} is Cartesian. Thus, base change along this diagram gives that
\begin{align*}
	(j^u)^!i_{X*}\mathcal{G}\simeq i_{U*}\wih{j}^!\mathcal{G}\simeq 0,
\end{align*} 
so that $i_{X*}$ restricts as required.

We will now show that the functors $\wih{i}_*$ and $\wih{i}^!$ are adjoint equivalences. To see this, we apply \zcref{cor: support is formal completion for L^} and \zcref{cor: support is formal completion for L^u} to each of the respective inclusions $\iota^u: \wih{(\mathcal{L}^u(X /G))_{\mathcal{L}^u(Z /G)}}\to \mathcal{L}^u(X /G)$ and $\iota^\wedge:\wih{\mathord{(\hL (X /G))_{\hL (Z /G)}}}\to \hL (X /G)$. We thus obtain adjoint equivalences
	\[
	\begin{tikzcd}
		\mathsf{IndCoh}_{\mathcal{L}^u(Z /G)}(\mathcal{L}^u(X /G))^{\tT}\arrow[r,"(\iota^u)^!",shift left]&\mathsf{IndCoh}(\mathcal{L}^u(\wih{X_Z}/G))^{\tT}\arrow[l,"\iota^u_*",shift left]
	\end{tikzcd}
	\] 
	and
	\[
	\begin{tikzcd}
		\mathsf{IndCoh}_{\hL (Z /G)}(\hL (X /G))^{\tT}\arrow[r,"(\iota^\wedge)^!",shift left]&\mathsf{IndCoh}(\hL (\wih{X_Z}/G))^{\tT}\arrow[l,"\iota^\wedge_*",shift left].
	\end{tikzcd}
	\] 
	Combining these with the natural adjoint equivalences of \zcref[noname]{eq: support equivalence for unipotent and formal}, we obtain a commutative diagram
	\[
	\begin{tikzcd}
		\mathsf{IndCoh}_{\mathcal{L}^u(Z /G)}(\mathcal{L}^u(X /G))^{\tT}\arrow[r,shift left,"i_X^!"]\arrow[d,shift left, "(\iota^u)^!"]&\mathsf{IndCoh}_{\hL (Z /G)}(\hL (X /G))^{\tT}\arrow[l,shift left,"i_{X*}"]\arrow[d,shift left,"(\iota^\wedge)^!"]\\
		\mathsf{IndCoh}(\mathcal{L}^u(\wih{X_Z}/G))^{\tT}\arrow[u,"\iota^u_*",shift left]\arrow[r,shift left,"\wih{i}^!"]&\mathsf{IndCoh}(\hL (\wih{X_Z}/G))^{\tT}\arrow[u,"\iota^\wedge_*",shift left]\arrow[l,shift left,"\wih{i}_*"]
	\end{tikzcd}
	\] 
	Since all functors in the diagram except $\wih{i}^!$ and $\wih{i}_*$ are equivalences, we deduce that those functors are equivalences, too. 
\end{proof}

We now give the proof of nil-invariance of $\Coh(-)^{tS^1}$ on the $\a$-potent loop spaces.
\begin{prop}
	\label{prop: Kashiwara's lemma for unipotent loop spaces}
	Let $U$ be a finite type derived scheme with affine diagonal and equipped with an action of a reductive group $G$, and let $Z$ be a $G$-invariant closed subscheme of $U$ such that the inclusion $Z\to U$ is almost finitely presented. Let $\calL^{\alpha}(i): \calL^{\alpha}(Z/G) \to \calL^{\alpha}(\wih{U_Z}/G)$ denote the standard
	inclusion. Then $(\calL^{\alpha}(i)_*)^{tS^1}: \Coh(\calL^{\alpha}(Z/G))^{tS^1} \to \Coh(\calL^{\alpha}(\wih{U_Z}/G))^{tS^1}$
	is an equivalence.
\end{prop}

\begin{proof}
	We begin by showing that the related functor,
	\[(\calL^{\alpha}(i)_*)^{\tT}: \IndCoh(\calL^{\alpha}(Z/G))^{\tT} \to \IndCoh(\calL^{\alpha}(\wih{U_Z}/G))^{\tT},\]
	is an equivalence. Afterwards, we will show that $\calL^{\alpha}(i)_*$ preserves coherent objects. Since $\calL^{\alpha}(i)$ is $S^1$-equivariant, we can restrict this equivalence to compact objects to conclude, by using \cite[Proposition 5.30]{CautisWilliams} and \zcref{cor: loop spaces of formal completions are ind-geometric}. 
	
	As noted above, by \cite[Corollary A.18]{BZChenHelmNadler} and \zcref{lem: shifted localization for formal completions}, respectively,
	the shifted localization maps $s\ell_Z: \calL^u(Z^{\a}/G(\a)) \to \calL^{\alpha}(Z/G)$
	and $s\ell_{\wih{U_Z}}: \calL^u((\wih{U_Z})^{\a}/G(\a)) \simeq \calL^u(\wih{U^{\a}_{Z^{\a}}}/G(\a)) \to \calL^{\alpha}(\wih{U_Z}/G)$\footnotemark
	\footnotetext{cf. \zcref{lem: formal completion is quotient} and \zcref{lem:fixed-points-commute-with-formal-completion}.}
	are $S^1$-equivariant equivalences for the $\a$-twisted loop rotation $S^1$-action $\rho(\a)$ on the source
	and the ordinary loop rotation $S^1$-action $\rho$ on target.
	It is straightforward to see that these equivalences fit into a commutative diagram
	of $S^1$-equivariant maps,
	\[\begin{tikzcd}
		{\calL^u(Z^{\a}/G(\a)) } & {\calL^u(\wih{U^{\a}_{Z^{\a}}}/G(\a))} \\
		{\calL^{\alpha}(Z/G)} & {\calL^{\alpha}(\wih{U_Z}/G)}
		\arrow["\calL^u(i^{\a})", from=1-1, to=1-2]
		\arrow["{s\ell_Z}"', from=1-1, to=2-1]
		\arrow["{s\ell_U}", from=1-2, to=2-2]
		\arrow["{\calL^{\alpha}(i)}", from=2-1, to=2-2],
	\end{tikzcd}\]
	where $i^{\a}: Z^{\a}/G(\a) \to \wih{U^{\a}_{Z^{\a}}}$ is map on
	fixed point loci induced by the $S^1$-equivariant map $i$.
	
	After applying $\IndCoh(-)$ and the $(-)^{\tT}$ construction
	(which respects composition of functors by \cite[Lemma 4.6.1]{Preygel}), this commutative diagram yields
	the following commutative diagram of functors,
	\[\begin{tikzcd}
		{\IndCoh(\calL^u(Z^{\a}/G(\a)))^{\tT(\a)}} &&& {\IndCoh(\calL^u(\wih{U^{\a}_{Z^{\a}}}/G(\a)))^{\tT(\a)}} \\
		{\IndCoh(\calL^{\alpha}(Z/G))^{\tT}} &&& {\IndCoh(\calL^{\alpha}(\wih{U_Z}/G))^{\tT}}
		\arrow["{(\calL^u(i^{\a})_*)^{\tT(\a)}}", from=1-1, to=1-4]
		\arrow["{(\calL^{\alpha}(i)_*)^{\tT}}", from=2-1, to=2-4]
		\arrow["\simeq", from=1-1, to=2-1]
		\arrow["\simeq", from=1-4, to=2-4],
	\end{tikzcd}\]
	where $\tT(\a)$ and $\tT$ denote the renormalized Tate invariants taken
	with respect to the $\rho(\a)$ and $\rho$ $S^1$-actions, respectively. Indeed, \cite[Proposition 5.30]{CautisWilliams} applies, as all stacks in questions are formal completions of coherent geometric stacks by \zcref{cor: loop spaces of formal completions are ind-geometric}, hence are coherent ind-geometric stacks by \zcref{lem: formal completions of aft algebraic stacks are coherent ind-geometric}.
	Thus, it suffices to show that $(\calL^u(i^{\a})_*)^{\tT(\a)}$ is an equivalence.

	By \zcref{prop: Tate invariants kill everything except a-trivial block},\footnotemark
    \footnotetext{The stack $\mathcal{L}^u(\wih{U^\a_{Z^\a}}/G(\alpha))$ is coherent ind-geometric by \zcref{lem:unipotent-loops-commute-with-formal-completion}, \zcref{lem: formal completion is quotient}, and \zcref{cor: loop spaces of formal completions are ind-geometric}, where we note that $\mathcal{L}^1(Y /G)=\mathcal{L}^u(Y /G)$ for any $G$-stack $Y$. Indeed, we have the natural isomorphisms
		\begin{align*}
			\mathcal{L}^u(\wih{U^\a_{Z^\a}}/G(\alpha))\simeq&\wih{(\mathcal{L}^u(U^\a /G(\a)))_{\mathcal{L}^u(Z^\a /G(\alpha))}}\\
			\simeq&\mathcal{L}(U^\a /G(\a))\times _{(\mathcal{L}(U^\a /G(\a)))_{\text{dR}}}\bbO_1(U^\a /G(\a))_{\text{dR}}\times _{\bbO_1(U^\a /G(\a))_{\text{dR}}}\bbO_1(Z^\a /G(\a))_{\text{dR}}\\
			\xrightarrow{\simeq}&\mathcal{L}(U^\a /G(\a))\times _{(\mathcal{L}(U^\a /G(\a)))_{\text{dR}}}\bbO_1(Z^\a /G(\a))_{\text{dR}}\\
			\coloneq &\wih{(\mathcal{L}(U^\a /G(\a)))_{\bbO_1(Z^\a /G(\a))}},
		\end{align*} 
		where in the second line we use nil-invariance of $(-)_{\text{dR}}$. We also note that the inclusion $\bbO_1(Z^\a /G(\alpha))\hookrightarrow \mathcal{L}(U^\a /G(\a))$ is an almost finitely presented closed immersion, because it is the composition of a base change of the inclusion of the point $\alpha$ into $G/\!\!/G$ with the almost finitely presented closed immersion $\calL Z\to \calL X$.
	}
    it
	in fact suffices to show that the induced functor on $\a$-trivial blocks,
	\[(\calL^u(i^{\a})_*)^{\tT(\a)}_{\a}: \IndCoh(\calL^u(Z^{\a}/G(\a)))^{\tT(\a)}_{\a}
		\to \IndCoh(\calL^u(\wih{U^{\a}_{Z^{\a}}}/G(\a)))^{\tT(\a)}_{\a}\]
	is an equivalence.
	In order to do so, we note that on the $\a$-trivial blocks, $\IndCoh(\calL^u(Z^{\a}/G(\a)))_{\a}$
	and $\IndCoh(\calL^u(\wih{U^{\a}_{Z^{\a}}}/G(\a)))_{\a}$,
	the $\a$-twisted and ordinary $S^1$-actions coincide. Thus, there are equivalences 
	that fit into a commutative diagram as follows:
		\[\begin{tikzcd}
		{\IndCoh(\calL^u(Z^{\a}/G(\a)))^{\tT}_{\a}} &&& {\IndCoh(\calL^u(\wih{U^{\a}_{Z^{\a}}}/G(\a)))^{\tT}_{\a}} \\
		{\IndCoh(\calL^u(Z^{\a}/G(\a)))^{\tT(\a)}_{\a}} &&& {\IndCoh(\calL^u(\wih{U^{\a}_{Z^{\a}}}/G(\a)))^{\tT(\a)}_{\a}}
		\arrow["{(\calL^u(i^{\a})_*)^{\tT}}", from=1-1, to=1-4]
		\arrow["{(\calL^u(i^{\a})_*)^{\tT(\a)}}", from=2-1, to=2-4]
		\arrow["\simeq", from=1-1, to=2-1]
		\arrow["\simeq", from=1-4, to=2-4].
	\end{tikzcd}\]
	Because $\calL^u(i^{\a})_*$ respects the decomposition of its source and
	target given by the $\a$-trivialization,\footnotemark
	\footnotetext{This follows from the fact that $\calL^u(i^{\a})$ preserves the
	canonical $\a$-trivializations of the unipotent loop spaces (cf. \zcref{lem: loop space has A-trivialization}).} 
	and the $S^1$-action 
	preserves full subcategories, it therefore
	suffices to show the stronger statement that
	\[(\calL^u(i^{\a})_*)^{\tT}: \IndCoh(\calL^u(Z^{\a}/G(\a)))^{\tT} \to \IndCoh(\calL^u(\wih{U^{\a}_{Z^{\a}}}/G(\a)))^{\tT}\]
	is an equivalence. By \zcref{thm: BZCHN Tate sheaves on unipotent equals formal} and \zcref{cor: Tate sheaves on formal unipotent equals formal}, we see that it suffices to show that
	\[(\hL (i^{\a})_*)^{\tT}: \IndCoh(\hL (Z^{\a}/G(\a)))^{\tT} \to \IndCoh(\hL (\wih{U^{\a}_{Z^{\a}}}/G(\a)))^{\tT}\]
	is an equivalence. To this end, let $\wih{p}: \hL (Z^{\a}/G(\a)) \to (Z^{\a}/G(\a))_{\dR}$
	and $\wih{p'}: \hL (\wih{U^{\a}_{Z^{\a}}}/G(\a)) \to (\wih{U^{\a}_{Z^{\a}}}/G(\a))_{\dR}$
	denote the canonical maps from the formal loop space to the de Rham stack,
	where our notation $\wih{p}$ and $\wih{p'}$ is chosen to agree with that in \cite{Preygel}.
	We note that $\wih{p}$ and $\wih{p'}$ are canonically $S^1$-equivariant for the
	loop rotation action on the source and the trivial $S^1$-action on the target.
	Hence, we have a commutative diagram,
	\begin{equation}
		\label{eqn: formal loop de Rham cyclic sheaves diagram}
	\begin{tikzcd}
		{\IndCoh\left(\hL (Z^{\a}/G(\a))\right)^{\tT}} &&& {\IndCoh\left(\hL \left(\wih{U^{\a}_{Z^{\a}}}/G(\a)\right)\right)^{\tT}} \\
		{\IndCoh\left((Z^{\a}/G(\a))_{\dR}\right)^{\tT}} &&& {\IndCoh\left(\left(\wih{U^{\a}_{Z^{\a}}}/G(\a)\right)_{\dR}\right)^{\tT}}
		\arrow["{(\hL (i^{\a})_*)^{\tT}}", from=1-1, to=1-4]
		\arrow["{(\wih{p}_*)^{\tT}}"', from=1-1, to=2-1]
		\arrow["{(\wih{p'}_*)^{\tT}}", from=1-4, to=2-4]
		\arrow["{(i^{\a}_{\dR *})^{\tT}}", from=2-1, to=2-4],
	\end{tikzcd}
	\end{equation}
	coming from the identification $\wih{p'} \circ \hL (i^{\a}) \simeq i^{\a}_{\dR} \circ \wih{p}$
	and the fact that $\tT$ is a functor on $\Fun(BS^1,\bCoh_t)$.
	The vertical arrows in \zcref[noname]{eqn: formal loop de Rham cyclic sheaves diagram} 
	are equivalences by \cite[Theorem 5.1.3]{Preygel}.\footnotemark
	\footnotetext{Although the results of \cite{Preygel} are stated only for
	derived algebraic spaces, \cite[Theorem 5.1.3]{Preygel} holds for stacks as
indicated by the remark at the top of page 21 of \textit{op. cit.} Indeed, the required statement is that for the stack $\mathfrak{X}=\wih{U^\alpha_{Z_\alpha}}/G(\alpha)$, the two maps $c_S: \mathfrak{X}\to \wih{\mathfrak{X}^S}$, where $S\in \{\emptyset,S^1\} $, are required to have $!$-pullbacks and $*$-pushforwards which are compatible with base change. Further, the $!$-pullbacks $c_S^!$ are required to be monadic. Now, the $c_S$ are inf-schematic nil-isomorphisms. Thus, by \cite[Chapter 3 Proposition 3.1.2]{GR17II}, the $*$-pushforwards $c_{S*}$ exist and the functors $c_S^!$ are conservative, hence monadic by Barr-Beck.}
	Meanwhile, $i^{\a}_{\dR}$ is an isomorphism because the map
	$i^{\a}: Z^{\a}/G(\a) \hook \wih{U^{\a}_{Z^{\a}}}/G(\a)$
	is a formal completion by \zcref{lem: formal completion is quotient}, 
	hence induces an isomorphism on underlying reduced stacks.
	Thus, $(i^{\a}_{\dR *})^{\tT}$ is an equivalence, so
	$(\hL (i^{\a})_*)^{\tT}$, hence $(\calL^{\alpha}(i)_*)^{\tT}$,
	is an equivalence as desired.

    It remains to show then
	that $(\calL^{\alpha}i)_*$ preserves $\Coh(-)$. 
	We begin by observing that $\calL^{\alpha}i$ is ind-inf-schematic and ind-proper. Indeed,
	the map $i: Z/G \to \wih{U_Z}/G$ is
	clearly an ind-schematic and ind-proper map of stacks.\footnotemark
	\footnotetext{$i$ is induced by a $G$-equivariant closed immersion of schemes $Z \to U$.}
	By the stability of ind-schematic and ind-proper morphisms under
	base change, we deduce that $\calL(Z/G) \to \calL(\wih{U_Z}/G)$ is an ind-schematic map of stacks, which
	implies, again by stability under base change,
	that the induced map of formal completions $\calL^{\alpha}i: \calL^{\alpha}(Z/G) \to \calL^{\alpha}(\wih{U_Z}/G)$
	is ind-schematic and ind-proper. In particular, it is ind-inf-schematic and ind-proper.
	Thus, by \cite[Chapter 3 Proposition 3.2.4]{GR17II}, $(\calL^{\alpha}i)_*$ admits a continuous
	right adjoint, which implies that $(\calL^{\alpha}i)_*$ preserves compact objects, which
	for the stacks involved means $\Coh(-)$.
\end{proof}

\section{The global quotient case}
\label{sec:the global quotient case}

\begin{lem}\label{lem: base change for homs}
	Let $X\to S$ be a reasonable ind-geometric stack over an affine scheme $S= \Spec A$. Let $S'=\Spec B\to S$ be a flat map. Then for $\calF,\calG\in\mathsf{Coh}(X)$, the canonical map
	\begin{align*}
		\Hom_{\mathsf{Coh}(X)}(\mathcal{F},\mathcal{G})\otimes _AB\to \Hom_{\mathsf{Coh}(X\times _SS')}(p^*\mathcal{F},p^*\mathcal{G})
	\end{align*} 
	is an isomorphism.
\end{lem} 
\begin{proof}
	Consider the Cartesian square
	\[
	\begin{tikzcd}
		X_n \times _SS'\arrow[r,"p_n"]\arrow[d,"\hat p_n"]&X_n \arrow[d,"f_n"]\\
		S'\arrow[r,"g"]&S
	\end{tikzcd}.
	\] 
	All objects in the square are geometric stacks, and both horizontal maps are flat. First, suppose that $\mathcal{F},\mathcal{G}\in \mathsf{Coh}(X_n)$. By \cite[Propositions 7.2 and 7.3]{CautisWilliams}, we have
	\begin{align*}
		p_n^*\sHom_{\mathsf{QCoh}(X_n)}(\mathcal{F},\mathcal{G})\simeq\sHom_{\mathsf{QCoh}(X_n \times _SS')}(p_n^*\mathcal{F},p_n^*\mathcal{G}),
	\end{align*} 
	since $\mathcal{F}$ is almost perfect, $\mathcal{G}$ is bounded below, and $p_n$ is flat. Now, Proposition 3.8 of \emph{op. cit.} gives
	\begin{align*}
		g^*f_{n*}\sHom_{\mathsf{QCoh}(X_n)}(\mathcal{F},\mathcal{G})\simeq \hat p_{n*}p_n^*\sHom_{\mathsf{QCoh}(X_n)}(\mathcal{F},\mathcal{G}).
	\end{align*} 
	Thus, we obtain
	\begin{align*}
		\Hom_{X_n}(\mathcal{F},\mathcal{G})\otimes _AB\simeq\Hom_{X_n \times _SS'}(p_n^*\mathcal{F},p_n^*\mathcal{G}).
	\end{align*} 
	By Proposition 5.5 of \emph{op. cit.},
	\begin{align*}
		\mathsf{Coh}(X)\simeq&\colim_n \mathsf{Coh}(X_n)\\
		\mathsf{Coh}(X\times _SS')\simeq&\colim_n \mathsf{Coh}(X_n \times _S S').
	\end{align*} 
	Now, let $\mathcal{F},\mathcal{G}\in \mathsf{Coh}(X)$, and suppose they are represented by $\mathcal{F}_n,\mathcal{G}_n\in \mathsf{Coh}(X_n)$. For $m\ge n$, write $\mathcal{F}_m,\mathcal{G}_m$ for their pushforwards to $X_m$. We compute
	\begin{align*}
		\Hom_{\mathsf{Coh}(X)}(\mathcal{F},\mathcal{G})\otimes _A B\simeq&\colim_{m\ge n}\Hom_{\mathsf{Coh}(X_m)}(\mathcal{F}_m,\mathcal{G}_m)\otimes _A B\\
		\simeq&\colim_{m\ge n}\Hom_{\mathsf{Coh}(X_m \times _S S')}(p_m^*\mathcal{F}_m,p_m^*\mathcal{G}_m)\\
		\simeq&\Hom_{\mathsf{Coh}(X\times _S S')}(p^*\mathcal{F},p^*\mathcal{G}).
	\end{align*} 
\end{proof}

\begin{lem}\label{lem: base change of S^1-invariants}
	Let $M\in A\mod$ be equipped with an $S^1$-action, and suppose $M$ is cohomologically bounded below. If $A\to B$ is flat, then the natural map
	\begin{align*}
		M^{tS^1}\otimes_{A[u^{\pm 1}]} B[u^{\pm 1}] \to (M\otimes_A B)^{tS^1}
	\end{align*} 
	is an isomorphism. Here the $S^1$-action on $M \otimes_A B$ is the one induced by the trivial $S^1$-actions
    on $A$ and $B$. 
\end{lem} 
\begin{proof}
    We begin by proving that $M^{S^1} \otimes_{A[u]} B[u] \xrightarrow{\simeq} (M \otimes_A B)^{S^1}$.
    
	By definition, $M^{S^1}\simeq\Hom_{A[\varepsilon]}(A,M)$, where $\deg \varepsilon=-1$. We now use the standard resolution of $A$ as a $A[\varepsilon]$-module,
	\begin{align*}
	P\coloneq \Tot(	\cdots \to A[\varepsilon][2]\to A[\varepsilon][1]\to A[\varepsilon]),
	\end{align*} 
	where the map $A[\varepsilon][j]\to A[\varepsilon][j-1]$ is multiplication by $\varepsilon$. Applying $\Hom_{A[\varepsilon]}(-,M)$ to this resolution we find
	\begin{align*}
		K(M)\coloneq \Hom_{A[\varepsilon]}(A,M)\simeq M[\![u]\!],
	\end{align*} 
	where the differential is $d=d_M+u\varepsilon$, $\deg u = 2$, and $d_M$ is the differential of $M$. This is the usual Koszul duality for mixed complexes. Note that $P\simeq A[\varepsilon,v]$ with $\deg v =-2$ and differential given by $d (a v^j)=\varepsilon a v^{j-1}$. Consider the fiber sequence
	\begin{align*}
		M[\![u]\!][-2n-2]\to M[\![u]\!]\to M[u] /u^{n+1}M[u].
	\end{align*} 
	Suppose the lower bound on $M$ is $b$. Then the fiber
	\begin{align*}
		M[\![u]\!][-2n-2]
	\end{align*} 
	is bounded below by $b+2n+2$. There is a commutative square
	\[
	\begin{tikzcd}
		K(M)\otimes _AB\arrow[r]\arrow[d]&K(M\otimes _AB)\arrow[d]\\
		K_n(M)\otimes _AB\arrow[r,"\simeq"]&K_n(M\otimes _AB)
	\end{tikzcd},
	\] 
	where $K_n(M)=(M[u]/u^{n-1},d)$. The map
	\begin{align*}
		K_n(M)\otimes _AB\to K_n(M\otimes _AB)
	\end{align*} 
	is an isomorphism because it involves only finitely many summands, and tensoring preserves direct sums. Since $ B$ is flat over $A$, the bound $b+2n+2$ is preserved. Thus, both vertical maps in the square have fibers bounded below by $b+2n+2$. Now, fixing $j$, choose $n$ such that $b+2n+2>j+1$, in which case both vertical maps give isomorphisms on $H^j$. Since we can do this for arbitrary $j$, we find
	\begin{align*}
		K(M)\otimes _AB\xrightarrow{\simeq}K(M\otimes _AB).
	\end{align*}
    Now observe that $K(M) \otimes_A B \simeq K(M) \otimes_{A[u]} B[u]$, and
    we conclude that $M^{S^1} \otimes_{A[u]} B[u] \simeq (M \otimes_A B)^{S^1}$, as desired.
    To conclude, we apply base change $- \otimes_{k[u]} \ku$ to each side of the equivalence
    to obtain,
    \[M^{tS^1} \otimes_{A[u^{\pm 1}]} B[u^{\pm 1}] \simeq (M^{S^1} \otimes_{A[u]} B[u]) \otimes_{k[u]} \ku \simeq (M \otimes_A B)^{S^1} \otimes_{k[u]} \ku,\]
    as desired.
\end{proof}

\begin{prop}\label{prop: Tate base change of Coh is fully faithful}
	Let $X \to S$ be a stack over $S=\Spec A$ such that $X$ is reasonable ind-geometric, and let $S'=\Spec B\to S$ be a flat map. Suppose that $X$ has an $S^1$-action; $S$ and $S'$ are equipped with the trivial $S^1$-action; and the maps $X\to S$ and $S\to S'$ are $S^1$-equivariant. The projection $p: X \times_S S' \to X$ is then
    canonically $S^1$-equivariant, so $p^*$ is $S^1$-equivariant and its target is $B$-linear. Let $\theta$ denote the natural $\Perf B[u^{\pm 1}]$-linear functor
	\begin{align*}
		\theta:\mathsf{Coh}(X)^{tS^1}\otimes _{A} B\to& \mathsf{Coh}(X\times _S S')^{tS^1}
	\end{align*} 
	obtained by extension of scalars. Then $\theta$ is fully faithful.
\end{prop} 
\begin{proof}
	Let $\mathcal{F}',\mathcal{G}'\in \mathsf{Coh}(X)^{S^1}$ whose underlying object in $\Coh(X)$ are denoted by $\mathcal{F}$ and $\mathcal{G}$, respectively. 
    Set $M \coloneq \Hom^A_{\mathsf{Coh}(X)}(\mathcal{F},\mathcal{G})$, considered with its induced $S^1$-action. Let $\mathcal{F}'_B$ and $\mathcal{G}'_B$ denote 
    the image of $\mathcal{F}',\mathcal{G}'$, respectively, in $\Coh(X)^{S^1} \otimes_A B$ under extension of scalars. Then
	\begin{align}
		\Hom^{B[u^{\pm 1}]}_{\Coh(X)^{S^1} \otimes_A B}(\mathcal{F}'_B,\mathcal{G}'_B) &\simeq (M^{S^1} \otimes_{A[u]} A[u^{\pm 1}])\otimes _{A[u^\pm 1]}B[u^\pm 1] \label{eq:application of S^1-invariants of homs}\\
		&\xrightarrow{\simeq} (M\otimes _AB)^{tS^1} \label{eq:application of base change for S^1 invariants}\\
		&\xrightarrow{\simeq} \left(\Hom^B_{\Coh(X \times_S S')}(p^*\mathcal{F},p^*\mathcal{G})\right)^{S^1} \otimes_{B[u]} B[u^{\pm 1}] \label{eq:application of base change for Homs}\\
		&\simeq \Hom^{B[u^{\pm 1}]}_{\Coh(X \times_S S')^{tS^1}}(\theta (\mathcal{F}_B),\theta (\mathcal{G}_B))\nonumber,
	\end{align} 
	where equivalence \zcref[noname]{eq:application of S^1-invariants of homs} follows from \zcref{lem: morphisms in C^G}; 
    \zcref[noname]{eq:application of base change for S^1 invariants} follows from \zcref{lem: base change of S^1-invariants}
    since objects in $\Coh(X)$ are bounded; 
    and \zcref[noname]{eq:application of base change for Homs} follows from \zcref{lem: base change for homs}.
\end{proof}

We first prove Kashwara's lemma for $\Pot(-)$ in the case of a
global quotient of a derived scheme.
The proof in this case is broken up into two propositions shown
below, proving fully faithfulness of $i_*$ and essential surjectivity,
respectively.

\begin{thm}
		\label{thm: Kashiwara's lemma for G-projective quotients}
		Let $U$ be a Noetherian scheme almost of finite type and with affine diagonal 
		equipped with an action by a reductive group $G$.
		Let $Z \hook U$ be an almost finitely presented inclusion of a $G$-invariant closed derived subscheme of $U$.
		Then the pushforward along $i: Z/G \to U/G$ induces an
		equivalence,
		\[i_*: \Pot(Z/G) \to \Pot_{Z/G}(U/G).\]
	\end{thm}

	\begin{proof}
		\zcref{prop: fully faithfulness in projective case} and \zcref{prop: essential surjectivity in projective case}
		together imply that $(\calL i_*)^{tS^1}: \Coh(\calL (Z/G))^{tS^1} \xrightarrow{\simeq} \Coh_{\calL (Z/G)}(\calL (U/G))^{tS^1}$ is an equivalence.
		Taking the ind-extension of this equivalence gives the desired result.
	\end{proof}

\subsubsection{Fully faithfulness}
We first need two technical lemmas, to prepare for the proof of fully faithfulness below.
\begin{lem}\label{lem: tensoring Homs with perfect complex commutes with S^1-invariants}
	Let $A \in \CAlg_k$, and let $\mathscr{C} \in (\dgCat_A)^{BS^1}$. Let $a',b'\in \mathscr{C}^{S^1}$, 
    and let $a,b\in \mathscr{C}$ denote their underlying objects. Denote by $\eta^*: \scrC^{S^1} \to \scrC^{tS^1}$
    the canonical functor given by extension of scalars.
    Let $P\in \mathsf{Perf}(A)$, and let $P^\vee=\Hom_A(P,A)\in \mathsf{Perf}(A)$ 
    denote the dual of $P$. Then there is a natural isomorphism
	\begin{align*}
		\Hom^{A}_{\mathscr{C}^{tS^1}}(\eta^*(a'\otimes_{A[u]} P[u]), \eta^*(b'\otimes_{A[u]} P[u])) \simeq \Hom^{A}_{\mathscr{C}^{tS^1}}(a',b') \otimes_A P^\vee \otimes_A P,
	\end{align*} 
	where for any $x,y\in \mathscr{C}^{tS^1}$, $\Hom^{A}_{\mathscr{C}^{tS^1}}(x,y)$ denotes the restriction of the $A[u^{\pm 1}]$-module $\Hom^{A[u^{\pm 1}]}_{\mathscr{C}^{tS^1}}(x,y)$ along $A\to A[u^{\pm 1}]$.
\end{lem}
\begin{proof}
	We first show that there is a natural isomorphism of $A[u]$-modules
	\begin{align*}
		\Hom_{\mathscr{C}}^{A}(a,b)^{S^1}\otimes _AP\otimes _AP^\vee\simeq \Hom_{\mathscr{C}}^{A}(a\otimes_A P,b\otimes _AP)^{S^1}.
	\end{align*} 
	Indeed,
	\begin{align*}
		\Hom_{\mathscr{C}}^A(x,y)^{S^1}\coloneq \lim_{BS^1}\Hom_{\mathscr{C}}^A(x,y)
	\end{align*} 
	for any $x,y\in \mathscr{C}$. Since $P,P^\vee$ are dualizable $A$-modules, the functors $-\otimes _AP$ and $-\otimes _AP^\vee$ are both left and right adjoints; in particular, they preserve limits, which gives the desired isomorphism. We next claim that the underlying object under the forgetful functor $\mathscr{C}^{S^1}\to \mathscr{C}$ of $a'\otimes _{A[u]}P[u]$ is precisely $a\otimes _AP$, and similarly for $b$. Indeed, the forgetful functor is $A$-linear, and $-\otimes _{A[u]}P[u]$ is canonically isomorphic to $-\otimes _AP$. Thus, by \zcref{lem: morphisms in C^G}, we obtain a natural, $A[u ]$-linear isomorphism
	\begin{align*}
		\Hom_{\mathscr{C}}^{A}(a,b)^{S^1}\otimes _AP\otimes _AP^\vee\simeq\Hom_{\mathscr{C}^{S^1}}^{A[u]}(a'\otimes _{A[u]}P[u],b'\otimes _{A[u]}P[u]).
	\end{align*} 
	Given $x',y'\in\mathscr{C}^{S^1}$, by \cite[Lemma 2.7]{Perry}, the functor $\eta^*$ induces a natural isomorphism of $A[u^{\pm 1}]$-modules
	\begin{align*}
		\Hom_{\mathscr{C}^{S^1}}^{A[u]}(x',y')\otimes _{A[u]}A[u^{\pm 1}]\xrightarrow{\simeq}\Hom_{\mathscr{C}^{tS^1}}^{A[u^{\pm 1}]}(\eta^*x',\eta^*y').
	\end{align*} 
	We thus have the following chain of natural isomorphisms, where $\res$ is the restriction functor along $A\to A[u^{\pm 1}]$.
	\begin{align*}
		\Hom_{\mathscr{C}^{tS^1}}^{A}(\eta^*a',\eta^*b')  \otimes _AP\otimes _AP^\vee\simeq&\res\left(\Hom^{A[u]}_{\mathscr{C}^{S^1}}(a',b')\otimes _AP\otimes _AP^\vee\otimes _{A[u]}A[u^{\pm 1}]\right)\\
		\simeq& \res\left( \Hom_{\mathscr{C}}^{A}(a,b)^{S^1} \otimes _AP\otimes _AP^\vee\otimes _{A[u]}A[u^{\pm 1}]\right) \\
		\simeq &\res\left( \Hom_{\mathscr{C}}^A (a\otimes _AP,b\otimes _AP)^{S^1}\otimes _{A[u]}A[u^{\pm 1}]\right) \\
		\simeq &\res\left(\Hom^{A[u]}_{\mathscr{C}^{S^1}}(a'\otimes _{A[u]}P[u],b'\otimes _{A[u]}P[u])\otimes _{A[u]}A[u^{\pm 1}]\right)\\
		\simeq&\Hom_{\mathscr{C}^{tS^1}}^A(\eta^*(a'\otimes _{A[u]}P[u]),\eta^*(b'\otimes _{A[u]}P[u])).
	\end{align*} 
\end{proof}
\begin{lem}
\label{lem: rewriting morphisms of Tate}
    Let $A \in \CAlg_k$, and let $\scrC, \scrD \in (\dgCat_A)^{BS^1}$.
    Suppose that $F: \scrC \to \scrD$ is an $S^1$-equivariant, $A$-linear functor, 
    and let $a', b' \in \scrC^{S^1}$ whose underlying objects of $\scrC$
    are denoted by $a, b \in \scrC$, respectively. Let $\eta^*_{\mathscr{C}}$ (resp. $\eta^*_{\mathscr{D}}$)
    denote the canonical extension of scalars functor,
    $\scrC^{S^1} \to \scrC^{S^1} \otimes_{A[u]} A[u^{\pm 1}]$
    (resp. $\scrD^{S^1} \to \scrD^{S^1} \otimes_{A[u]} A[u^{\pm 1}]$).
    Then there is a commutative diagram of $A[u^{\pm 1}]$-modules,
    \[\begin{tikzcd}
	    {\Hom^{A[u^{\pm 1}]}_{\scrC^{tS^1}}(\eta^*_{\mathscr{C}}a', \eta^*_{\mathscr{C}}b')} & & {\Hom^{A[u^{\pm 1}]}_{\scrD^{tS^1}}(F^{tS^1}\eta^*_{\mathscr{C}}a', F^{tS^1}\eta^*_{\mathscr{C}}b')} \\
	{\left(\Hom^A_{\scrC}(a,b)\right)^{tS^1}} & & {\left(\Hom^A_{\scrD}(Fa,Fb)\right)^{tS^1}}
	\arrow["{F^{tS^1}_{\eta^*a', \eta^*b'}}", from=1-1, to=1-3]
	\arrow["\simeq"', from=1-1, to=2-1]
	\arrow["\simeq"', from=1-3, to=2-3]
	\arrow["{F_{a,b}}", from=2-1, to=2-3],
\end{tikzcd}\]
    in which the vertical arrows are isomorphisms,
    as indicated.
\end{lem}

\begin{proof}
    First, observe that there is a commutative square
    \[
    \begin{tikzcd}
	    \mathscr{C}^{S^1}\arrow[d,"\eta^*_{\mathscr{C}}"]\arrow[r,"F^{S^1}"]&\mathscr{D}^{S^1}\arrow[d,"\eta^*_{\mathscr{D}}"]\\
	    \mathscr{C}^{tS^1}\arrow[r,"F^{tS^1}"]&\mathscr{D}^{tS^1}
    \end{tikzcd}.
    \] 
    Indeed, functoriality of the tensor product gives that 
    \begin{align*}
	    \eta^*_{\mathscr{C}}=\id_{\mathscr{C}}\otimes \eta,\qquad\eta^*_{\mathscr{D}}=\id_{\mathscr{D}}\otimes \eta,\qquad F^{S^1}\otimes_{k[u]}\ku=F^{S^1}\otimes \id_{\mathsf{Perf}(\ku)},
    \end{align*} 
    where $\eta: \mathsf{Perf}(k[u])\to \mathsf{Perf}(\ku)$ is the usual extension of scalars. Thus, we have the equivalence
    \begin{align*}
	    \eta^*_{\mathscr{D}}\circ F^{S^1} &= (\id_{\mathscr{D}}\otimes \eta)\circ F^{S^1}\\
	    &\simeq F^{S^1}\otimes \eta\\
	    &\simeq F^{S^1}\circ(\id_{\mathscr{C}}\otimes \eta),
    \end{align*} 
    as desired. Thus, we obtain the following commutative diagram
    \[
    \begin{tikzcd}
	    \Hom_{\mathscr{C}^{S^1}}^{A[u]}(a',b')\arrow[r,"F^{S^1}"]\arrow[d,"\eta^*_{\mathscr{C}}"]&\Hom_{\mathscr{D}^{S^1}}^{A[u]}(F^{S^1}a',F^{S^1}b')\arrow[d,"\eta^*_{\mathscr{D}}"]\\
	    \Hom_{\mathscr{C}^{tS^1}}^{A[u^{\pm 1}]}(\eta^*_{\mathscr{C}}a',\eta^*_{\mathscr{C}}b')\arrow[r,"F^{tS^1}"]&\Hom_{\mathscr{D}^{tS^1}}^{A[u^{\pm 1}]}(\eta^*_{\mathscr{D}}F^{S^1}a',\eta^*_{\mathscr{D}}F^{S^1}b')
    \end{tikzcd}.
    \] 
    By \cite[Lemma 2.10]{Perry}, this diagram induces a commutative diagram
    \begin{equation}\label{cd: S^1-invariants to Tate invariants}
	    \begin{tikzcd}[column sep=large]
	    \Hom_{\mathscr{C}^{S^1}}^{A[u]}(a',b')\otimes _{A[u]}A[u^{\pm 1}]\arrow[r,"{F^{S^1}}"]\arrow[d,"\simeq"]&\Hom_{\mathscr{D}^{S^1}}^{A[u]}(F^{S^1}a',F^{S^1}b')\otimes _{A[u]}A[u^{\pm 1}]\arrow[d,"\simeq"]\\
	    \Hom_{\mathscr{C}^{tS^1}}^{A[u^{\pm 1}]}(\eta^*_{\mathscr{C}}a',\eta^*_{\mathscr{C}}b')\arrow[r,"F^{tS^1}"]&\Hom_{\mathscr{D}^{tS^1}}^{A[u^{\pm 1}]}(\eta^*_{\mathscr{D}}F^{tS^1}a',\eta^*_{\mathscr{D}}F^{tS^1}b')
    \end{tikzcd}
\end{equation} 
    Now, by \zcref{lem: morphisms in C^G}, we also have a commutative diagram
    \[
    \begin{tikzcd}
	    \Hom_{\mathscr{C}}^{A}(a,b)^{S^1}\arrow[r,"F"]\arrow[d,"\simeq"]&\Hom_{\mathscr{D}}^A(Fa,Fb)^{S^1}\arrow[d,"\simeq"]\\
	    \Hom_{\mathscr{C}^{S^1}}^{A[u]}(a',b')\arrow[r,"F^{S^1}"]&\Hom_{\mathscr{D}^{S^1}}^{A[u]}(F^{S^1}a',F^{S^1}b')
    \end{tikzcd}.
    \] 
    Applying the functor $-\otimes _{A[u]}A[u^{\pm 1}]$ to this diagram and combining the resulting diagram with \zcref[noname]{cd: S^1-invariants to Tate invariants}, we obtain the commutative diagram below, in which the desired diagram is the outer square.
    \[
	    \begin{tikzcd}[column sep = large]
		    \Hom_{\mathscr{C}}^{A}(a,b)^{tS^1}\arrow[r,"F"]\arrow[d,equal]&\Hom_{\mathscr{D}}^{A}(Fa,Fb)^{tS^1}\arrow[d,equal]\\
		    \Hom_{\mathscr{C}}^{A}(a,b)^{S^1}\otimes _{A[u]}A[u^{\pm 1}]\arrow[r,"F"]\arrow[d,"\simeq"]&\Hom_{\mathscr{D}}^A(Fa,Fb)^{S^1}\otimes _{A[u]} A[u^{\pm 1}]\arrow[d,"\simeq"]\\
    \Hom_{\mathscr{C}^{S^1}}^{A[u]}(a',b')\otimes _{A[u]}A[u^{\pm 1}]\arrow[r,"{F^{S^1}}"]\arrow[d,"\simeq"]&\Hom_{\mathscr{D}^{S^1}}^{A[u]}(F^{S^1}a',F^{S^1}b')\otimes _{A[u]}A[u^{\pm 1}]\arrow[d,"\simeq"]\\
	    \Hom_{\mathscr{C}^{tS^1}}^{A[u^{\pm 1}]}(\eta^*_{\mathscr{C}}a',\eta^*_{\mathscr{C}}b')\arrow[r,"F^{tS^1}"]\arrow[d,equal]&\Hom_{\mathscr{D}^{tS^1}}^{A[u^{\pm 1}]}(\eta^*_{\mathscr{D}}F^{S^1}a',\eta^*_{\mathscr{D}}F^{S^1}b')\arrow[d,"\simeq"]\\
	    \Hom_{\mathscr{C}^{tS^1}}^{A[u^{\pm 1}]}(\eta^*_{\mathscr{C}}a',\eta^*_{\mathscr{C}}b')\arrow[r,"F^{tS^1}"]&\Hom_{\mathscr{D}^{tS^1}}^{A[u^{\pm 1}]}(F^{tS^1}\eta_{\mathscr{C}}^*a',F^{tS^1}\eta^*_{\mathscr{C}}b')\\
    \end{tikzcd}
    \] 
\end{proof}

We now prove fully faithfulness.

\begin{prop}
	\label{prop: fully faithfulness in projective case}
	Suppose that $U$ is a derived scheme almost of finite type, with affine diagonal, and equipped with
	an action by a reductive group $G$.
	Let $Z \hook U$ be an almost finitely presented inclusion of a $G$-invariant closed derived subscheme of $U$. Let $i: \calL (Z/G) \to \calL (\wih{U_Z}/G)$ denote
	the map induced by the inclusion. Then
	the functor $(i_*)^{tS^1}: \Coh(\calL (Z/G))^{tS^1} \to \Coh(\calL (\wih{U_Z}/G))^{tS^1}$ is fully faithful.
\end{prop}
\begin{proof}
    Let $\eta: k[u] \to \ku$ denote the canonical localization map,
    and let $\eta^*: \Coh(\calL (Z/G))^{S^1} \to \Coh(\calL (Z/G))^{tS^1}$ denote the
    functor given by extension of scalars. Let $\calF', \calG' \in \Coh(\calL (Z/G))^{S^1}$ be
    arbitrary $S^1$-equivariant objects whose images under $\eta^*$ in $ \Coh(\calL (Z/G))^{tS^1}$ are denoted by $\calF$ and $\calG$, respectively. Since $i$
    commutes with the maps $\mu_Z$ and $\mu_{\wih{U_Z}}$,
    the functor $(i_*)^{tS^1}: \Coh(\calL (Z/G))^{tS^1} \to \Coh(\wih{U_Z}/G))^{tS^1}$
    is canonically $A$-linear, where $A \coloneq \calO(G/\!\!/G)$. With this in
    mind, we define the following $A$-module,
    \[F \coloneq \ker\left(\Hom^A_{\Coh(\calL (Z/G))^{tS^1}}(\mathcal{F}, \mathcal{G})\to 
        \Hom^A_{\Coh(\calL (\wih{U_Z}/G))^{tS^1}}\left((i_*)^{tS^1}\mathcal{F},(i_*)^{tS^1}\mathcal{G}\right)\right),\]
    where, if $\scrC$ is an $A[u^{\pm 1}]$-linear category, $\Hom^A_{\scrC}(-,-)$ denotes
    the $A$-module obtained by the restriction of scalars functor $\res_{A[u^{\pm 1}]/A}: A[u^{\pm 1}]\mod \to A\mod$ 
    given as the right adjoint to $- \otimes_A A[u^{\pm 1}]$.
    Since $\res_{A[u^{\pm 1}]/A}$ is a right adjoint it preserves limits, meaning
    \begin{align*}
        F &\simeq \res_{A[u^{\pm 1}]/A}\left(\ker\left(\Hom^{A[u^{\pm 1}]}_{\Coh(\calL (Z/G))^{tS^1}}(\calF,\calG)\to \Hom^{A[u^{\pm 1}]}_{\Coh(\calL (\wih{U_Z}/G))^{tS^1}}\left((i_*)^{tS^1}\calF, (i_*)^{tS^1}\calG\right)\right)\right).
	\end{align*}
    In order to show that $(i_*)^{tS^1}$ is fully faithful, we begin by showing that $F \simeq 0$.
    Since $A$ is Noetherian, it suffices by \cite[Lemma 2.12]{neeman-bockstedt} 
    to show that $F \otimes_A k(\frakp) \simeq 0$
    for every prime ideal $\frakp \subset A$, where $k(\frakp)$ denotes the residue field of $\frakp$.
    In fact, since the algebraic closure $k(\frakp) \to \overline{k(\frakp)}$ is faithfully flat,
    it suffices to show that $F \otimes_A \overline{k(\frakp)} \simeq 0$ for each prime $\frakp \subset A$.

    For any stack $X/k$, let $X_{\overline{k(\frakp)}} \coloneq X \times_{\Spec k} \Spec \overline{k(\frakp)}$ denote
    the base change of $X$ along $\Spec \overline{k(\frakp)} \to \Spec k$. Our key observation is
    that the point of $(G/\!\!/G)_{\overline{k(\frakp)}}$ given by the
    canonical map $\a: \Spec \overline{k(\frakp)} \to G/\!\!/G_{\overline{k(\frakp)}}$ is closed, 
    even though the original point $\frakp \in G/\!\!/G$ might not be.
    Let $\frakm_{\a} \subset A \otimes_k \overline{k(\frakp)} =: A_{\overline{k(\frakp)}}$
    denote the maximal ideal associated to the closed point $\a \in (G/\!\!/G)_{\overline{k(\frakp)}}$.
    Crucially, the statement of \zcref{prop: Kashiwara's lemma for unipotent loop spaces}
    is formulated over an arbitrary algebraically closed base field $k$ of characteristic zero, which was fixed
    by convention in \zcref{sec:conventions and notation}---in particular it
    applies to the case of the $G_{\overline{k(\frakp)}}$-equivariant
    closed immersion $i_{\overline{k(\frakp)}}: Z_{\overline{k(\frakp)}} \hook U_{\overline{k(\frakp)}}$.

    Since $A_{\overline{k(\frakp)}}$ is Noetherian (as $A$ is a finitely-generated
    $k$-algebra), $\frakm_{\a}$ is finitely generated. Choose a presentation $\frakm_{\a} = (f_1, \ldots, f_n)$
    by a finite set of generators, and let $P \coloneq K_{A_{\overline{k(\frakp)}}}(f)$ be the associated Koszul complex,
    which, as a finite sequence of free $A_{\overline{k(\frakp)}}$-modules, belongs to $\mathsf{Perf}(A_{\overline{k(\frakp)}})$.
    On the other hand, since $H^0(K_{A_{\overline{k(\frakp)}}}(f)) \simeq (A_{\overline{k(\frakp)}})/\frakm_{\a}$,
    it is in fact supported on the closed point $\a$. In particular,
    this means that the $A_{\overline{k(\frakp)}}$-module,
	  \[(F_{\overline{k(\frakp)}} \otimes _{A_{\overline{k(\frakp)}}}P^\vee) \otimes_{A_{\overline{k(\frakp)}}} P\]
    is supported at $\alpha\in (G/\!\!/G)_{\overline{k(\frakp)}}$. Now we establish a sequence
    of equivalences that allows us to conclude that $F_{\overline{k(\frakp)}}\otimes _{A_{\overline{k(\frakp)}}}P\otimes _{A_{\overline{k(\frakp)}}}P^\vee \simeq 0$:

\begin{center}
\resizebox{\linewidth}{!}{%
\begin{minipage}{1.4\textwidth}
\begin{align}
F_{\overline{k(\frakp)}} 
&=
\left(
\res_{A[u^{\pm 1}] / A}
\left(
\ker\left(
\Hom^{A[u^{\pm 1}]}_{\Coh(\calL (Z/G))^{tS^1}}(\calF,\calG)
\to
\Hom^{A[u^{\pm 1}]}_{\Coh(\calL (\wih{U_Z}/G))^{tS^1}}
\left((i_*)^{tS^1}\calF,(i_*)^{tS^1}\calG\right)
\right)
\right)
\right)_{A_{\overline{k(\frakp)}}}
\nonumber
\\
&\simeq
\res_{A_{\overline{k(\frakp)}}[u^{\pm 1}] / A_{\overline{k(\frakp)}}}
\left(
\ker\left(
\Hom^{A[u^{\pm 1}]}_{\Coh(\calL (Z/G))^{tS^1}}(\calF,\calG)
\to
\Hom^{A[u^{\pm 1}]}_{\Coh(\calL (\wih{U_Z}/G))^{tS^1}}
\left((i_*)^{tS^1}\calF,(i_*)^{tS^1}\calG\right)
\right)_{A_{\overline{k(\frakp)}}[u^{\pm 1}]}
\right)
\label{eq:fiber-2}
\\
&\simeq
\res_{A_{\overline{k(\frakp)}}[u^{\pm 1}] / A_{\overline{k(\frakp)}}}
\left(
\ker\left(
\left(\Hom^{A[u^{\pm 1}]}_{\Coh(\calL (Z/G))^{tS^1}}(\calF,\calG)\right)_{A_{\overline{k(\frakp)}}[u^{\pm 1}]}
\to
\left(\Hom^{A[u^{\pm 1}]}_{\Coh(\calL (\wih{U_Z}/G))^{tS^1}}
((i_*)^{tS^1}\calF,(i_*)^{tS^1}\calG)\right)_{A_{\overline{k(\frakp)}}[u^{\pm 1}]}
\right)
\right)
\label{eq:fiber-3}\\
&\simeq 
\res_{A_{\overline{k(\frakp)}}[u^{\pm 1}] / A_{\overline{k(\frakp)}}}
\left(\ker\left( \Hom^{A_{\overline{k(\frakp)}}[u^{\pm 1}]}_{\mathsf{Coh}\left((\mathcal{L}(Z /G))_{A_{\overline{k(\frakp)}}}\right)^{tS^1}}\left(\mathcal{F}_{A_{\overline{k(\frakp)}}},\mathcal{G}_{A_{\overline{k(\frakp)}}}\right)\to \Hom^{A_{\overline{k(\frakp)}}[u^{\pm 1}]}_{\mathsf{Coh}\left((\mathcal{L}(\wih{U_Z} /G))_{A_{\overline{k(\frakp)}}}\right)^{tS^1}}\left(\left((i_*)^{tS^1}\mathcal{F}\right)_{A_{\overline{k(\frakp)}}},
\left((i_*)^{tS^1}\mathcal{G}\right)_{A_{\overline{k(\frakp)}}}\right) \right)  \right) \label{eq:fiber-4} \\
&\simeq 
\res_{A_{\overline{k(\frakp)}}[u^{\pm 1}] / A_{\overline{k(\frakp)}}}
\left(\ker\left( \Hom^{A_{\overline{k(\frakp)}}[u^{\pm 1}]}_{\mathsf{Coh}\left((\mathcal{L}(Z /G))_{A_{\overline{k(\frakp)}}}\right)^{tS^1}}\left(\mathcal{F}_{A_{\overline{k(\frakp)}}},\mathcal{G}_{A_{\overline{k(\frakp)}}}\right)\to \Hom^{A_{\overline{k(\frakp)}}[u^{\pm 1}]}_{\mathsf{Coh}\left((\mathcal{L}(\wih{U_Z} /G))_{A_{\overline{k(\frakp)}}}\right)^{tS^1}}\left(({i_{\frakp}}_*)^{tS^1} (\mathcal{F}_{A_{\overline{k(\frakp)}}}),
({i_{\frakp}}_*)^{tS^1}(\mathcal{G}_{A_{\overline{k(\frakp)}}})\right) \right)  \right) \label{eq:fiber-5}
\end{align}
\end{minipage}
}
\end{center}
where
\begin{itemize}[label=--]
    \item the objects $(\mathcal{L}(Z/G))_{A_{\overline{k(\frakp)}}}$ and $(\mathcal{L}(\wih{U_Z} /G))_{A_{\overline{k(\frakp)}}}$ appearing in \zcref[noname]{eq:fiber-4} 
    denote the base changes of $\calL(Z/G),\calL(\wih{U_Z}/G) \in \mathsf{Stk}_{/A}$, respectively, along the map $\Spec A_{\overline{k(\frakp)}} \to \Spec A$;
    \item the map $i_{\frakp}: \calL(Z/G)_{A_{\overline{k(\frakp)}}}\to \calL(\wih{U_Z}/G)_{A_{\overline{k(\frakp)}}}$ appearing in \zcref[noname]{eq:fiber-5}
    denotes the base change of $i$ along $\Spec A_{\overline{k(\frakp)}} \to \Spec A$;
\end{itemize}
and where
\begin{itemize}[label=--]
    \item equivalence \zcref[noname]{eq:fiber-2} holds because restriction of scalars commutes
    with base change;\footnotemark
    \footnotetext{More precisely, suppose that $R \to R'$ and $R \to S$ are maps in $\CAlg_k$, and let $S' \coloneq R' \otimes_R S$.
    There is a natural equivalence of functors, $(- \otimes_R R') \circ \res_{S / R} \simeq \res_{S' / R'} \circ (- \otimes_R S)$.}
    \item equivalence \zcref[noname]{eq:fiber-3} holds because $A\mod$ is a stable category (meaning
    kernels are also finite colimits), and $- \otimes_{A[u^{\pm 1}]} A_{\overline{k(\frakp)}}[u^{\pm 1}]$ is a left adjoint;
    \item equivalence \zcref[noname]{eq:fiber-4} follows from \zcref{prop: Tate base change of Coh is fully faithful} and the fact that restriction of scalars $\res_{A_{\overline{k(\frakp)}}[u^{\pm 1}] / A_{\overline{k(\frakp)}}}$
	    satisfies the projection formula 
	    \begin{align*}
		    \res_{A_{\overline{k(\frakp)}}[u^{\pm 1}] / A_{\overline{k(\frakp)}}}(M)\otimes _AN\simeq \res_{A_{\overline{k(\frakp)}}[u^{\pm 1}] / A_{\overline{k(\frakp)}}}(M\otimes _{A[u^{\pm 1}]}Q[u^{\pm 1}]);
	    \end{align*} 
    \item and equivalence \zcref[noname]{eq:fiber-5} follows from the naturality of $\theta$ in \zcref{prop: Tate base change of Coh is fully faithful}.
    
\end{itemize}

For any stack $X\in \mathsf{Stk}_{/A}$, every square in the commutative diagram
	\[
	\begin{tikzcd}
		X_{A_{\overline{k(\frakp)}}}\arrow[r]\arrow[d]&\Spec A_{\overline{k(\frakp)}}\arrow[r]\arrow[d]&\Spec\overline{k(\frakp)}\arrow[d]\\
		X\arrow[r]&\Spec A\arrow[r]&\Spec k
	\end{tikzcd}
	\] 
is Cartesian. Thus, we have a canonical isomorphism
	\begin{align*}
		X_{A_{\overline{k(\frakp)}}}\simeq X_{\overline{k(\frakp)}}.
	\end{align*} 
Now, for any $X \in \Stk_k$, we have $(X_{\overline{k(\frakp)}})_{\dR} \simeq (X_{\dR})_{\overline{k(\frakp)}}$,
from which we deduce
\[(\wih{U_Z})_{\overline{k(\frakp)}} \simeq \wih{(U_{\overline{k(\frakp)}})_{Z_{\overline{k(\frakp)}}}}.\]
To compute the loop spaces of these stacks, we need to distinguish between the loop space functor $\mathcal{L}$ with respect to the base $\Spec k$ and that with respect to the base $\Spec \overline{k(\frakp)}$. Denote the former by $\mathcal{L}\coloneq\mathcal{L}_k$ and the latter by $\mathcal{L}_{\overline{k(\frakp)}}$. Using that $ \calL_k$ and $\mathcal{L}_{\overline{k(\frakp)}}$ preserve products, we therefore obtain equivalences
\begin{align*}
\calL_k(Z/G)_{A_{\overline{k(\frakp)}}}\simeq\calL_k(Z/G)_{\overline{k(\frakp)}} \simeq \calL_{\overline{k(\frakp)}}((Z/G)_{\overline{k(\frakp)}}) \simeq \calL_{\overline{k(\frakp)}}(Z_{\overline{k(\frakp)}}/G_{\overline{k(\frakp)}}), \;\text{and} \\
\calL_k(\wih{U_Z}/G)_{A_{\overline{k(\frakp)}}}\simeq\calL_k(\wih{U_Z}/G)_{\overline{k(\frakp)}} \simeq \calL_{\overline{k(\frakp)}}((\wih{U_Z}/G)_{\overline{k(\frakp)}}) \simeq \calL_{\overline{k(\frakp)}}\left(\wih{{U_{\overline{k(\frakp)}}}_{Z_{\overline{k(\frakp)}}}}/G_{\overline{k(\frakp)}}\right).
\end{align*}
With these equivalences in hand, we continue by 
computing the expression $\zcref[noname]{eq:fiber-4} \otimes _{A_{\overline{k(\frakp)}}}P\otimes _{A_{\overline{k(\frakp)}}}P^\vee$. In what follows,
\[\mathscr{C}\coloneq \Coh\left(\calL_{\overline{k(\frakp)}}\left(Z_{\overline{k(\frakp)}}/G_{\overline{k(\frakp)}}\right)\right),\quad\text{and}\quad
\mathscr{D}\coloneq \Coh\left(\calL_{\overline{k(\frakp)}}\left(\wih{(U_{\overline{k(\frakp)}})_{Z_{\overline{k(\frakp)}}}}/G_{\overline{k(\frakp)}}\right)\right):\]

\begin{center}
\resizebox{\linewidth}{!}{%
\begin{minipage}{1.3\textwidth}
\begin{align}
&\simeq 
\res_{A_{\overline{k(\frakp)}}[u^{\pm 1}] / A_{\overline{k(\frakp)}}}
\left(\ker\left( 
\Hom^{A_ {\overline{k(\frakp)}}[u^{\pm 1}]}_{\mathscr{C}^{tS^1}}\left(\mathcal{F}_{ {\overline{k(\frakp)}}},\mathcal{G}_{ {\overline{k(\frakp)}}}\right)\to\Hom^{A_ {\overline{k(\frakp)}}[u^{\pm 1}]}_{\mathscr{D}^{tS^1}}\left(({i_{\frakp}}_*)^{tS^1}\mathcal{F}_{ {\overline{k(\frakp)}}},({i_{\frakp}}_*)^{tS^1}\mathcal{G}_{ {\overline{k(\frakp)}}}\right) \right)  \right) \otimes _{A_{\overline{k(\frakp)}}}P\otimes _{A_{\overline{k(\frakp)}}}P^\vee
\label{eq:fiber-6}
\\
&\simeq 
\ker
\left(\res_{A_{\overline{k(\frakp)}}[u^{\pm 1}] / A_{\overline{k(\frakp)}}}\left( \Hom^{A_ {\overline{k(\frakp)}}[u^{\pm 1}]}_{\mathscr{C}^{tS^1}}\left(\mathcal{F}_{ {\overline{k(\frakp)}}},\mathcal{G}_{ {\overline{k(\frakp)}}}\right)\to \Hom^{A_ {\overline{k(\frakp)}}[u^{\pm 1}]}_{\mathscr{D}^{tS^1}}\left(({i_{\frakp}}_*)^{tS^1}\mathcal{F}_{ {\overline{k(\frakp)}}},({i_{\frakp}}_*)^{tS^1}\mathcal{G}_{ {\overline{k(\frakp)}}}\right) \right)  \right) \otimes _{A_{\overline{k(\frakp)}}}P\otimes _{A_{\overline{k(\frakp)}}}P^\vee
\label{eq:fiber-7}
\\
\begin{split}
&\simeq
\ker
\biggr(
\res_{A_{\overline{k(\frakp)}}[u^{\pm 1}] / A_{\overline{k(\frakp)}}}
\biggr(
\left(
\Hom^{A_{\overline{k(\frakp)}}[u^{\pm 1}]}_{\mathscr{C}^{tS^1}}
(\eta^*(\calF'_{\overline{k(\frakp)}}),\eta^*(\calG'_{\overline{k(\frakp)}}))
\right)
\otimes_{A_{\overline{k(\frakp)}}[u^{\pm 1}]}
\eta^*P[u]
\otimes_{A_{\overline{k(\frakp)}}[u^{\pm 1}]}
\eta^*P^{\vee}[u]\\
&\quad\longrightarrow
\left(
\Hom^{A_{\overline{k(\frakp)}}[u^{\pm 1}]}_{\mathscr{D}^{tS^1}}
\left(({i_{\frakp}}_*)^{tS^1}\eta^*(\calF'_{\overline{k(\frakp)}}),({i_{\frakp}}_*)^{tS^1}\eta^*(\calG'_{{\overline{k(\frakp)}}})
\right)
\right)
\otimes_{A_{\overline{k(\frakp)}}[u^{\pm 1}]}
\eta^*P[u]
\otimes_{A_{\overline{k(\frakp)}}[u^{\pm 1}]}
\eta^*P^{\vee}[u]
\biggr)
\biggr)
\end{split} 
\label{eq:fiber-8}
\\
&\simeq
\ker
\left(
\Hom^{A_{\overline{k(\frakp)}}}_{\mathscr{C}^{tS^1}}
(\eta^*\calF'_{A_{\overline{k(\frakp)}}},\eta^*\calG'_{A_{\overline{k(\frakp)}}})
\otimes_{A_{\overline{k(\frakp)}}}
P
\otimes_{A_{\overline{k(\frakp)}}}
P^{\vee}
\to
\Hom^{A_{\overline{k(\frakp)}}}_{\mathscr{D}^{tS^1}}
\left(\eta^*({i_{\frakp}}_*)^{S^1}\calF'_{A_{\overline{k(\frakp)}}},\eta^*({i_{\frakp}}_*)^{S^1}\calG'_{A_{\overline{k(\frakp)}}}\right)
\otimes_{A_{\overline{k(\frakp)}}}
P
\otimes_{A_{\overline{k(\frakp)}}}
P^{\vee}
\right)
\label{eq:fiber-9}
\\
\begin{split}
&\simeq   
\ker\biggr(
\Hom^{A_{\overline{k(\frakp)}}}_{\mathscr{C}^{tS^1}}
\left(\eta^*(\mathcal{F}'_{\overline{k(\frakp)}}) 
\otimes_{A_{\overline{k(\frakp)}}[u^{\pm 1}]} \eta^*P[u],\eta^*(\mathcal{G}'_{\overline{k(\frakp)}}) \otimes_{A_{\overline{k(\frakp)}}[u^{\pm 1}]} \eta^*P[u]
\right)\\
&\quad\longrightarrow
\Hom^{A_{\overline{k(\frakp)}}}_{\mathscr{D}^{tS^1}}
\left(({i_{\frakp}}_*)^{S^1}(\eta^*(\mathcal{F}'_{\overline{k(\frakp)}}) 
\otimes_{A_{\overline{k(\frakp)}}[u^{\pm 1}]} \eta^*P[u]),({i_{\frakp}}_*)^{S^1}(\eta^*(\mathcal{G}'_{\overline{k(\frakp)}}) \otimes_{A_{\overline{k(\frakp)}}[u^{\pm 1}]} \eta^*P[u])
\right)
\biggr),
\end{split} 
\label{eq:fiber-10}
\end{align}
\end{minipage}
}
\end{center}
where 
\begin{itemize}[label=--]
    \item equivalence \zcref[noname]{eq:fiber-6} follows from the preceding discussion;
    \item equivalence \zcref[noname]{eq:fiber-7} holds because, as a right adjoint, $\res_{A[u^{\pm 1}] / A}$ commutes with
    limits (e.g. kernels)
\item \zcref[noname]{eq:fiber-8} holds because restriction of scalars along $B\to C$ is tensoring with the $(B,C)$-bimodule $C$, hence commutes with tensor products over the base $B$;
    \item equivalence \zcref[noname]{eq:fiber-9} follows from the fact that the $A_{\overline{k(\frakp)}}$-linear $\Hom$ objects
	    are equivalent to image of the $A_{\overline{k(\frakp)}}[u^{\pm 1}]$-linear $\Hom$ objects under restriction of scalars;
    \item and equivalence \zcref[noname]{eq:fiber-10} follows from \zcref{lem: tensoring Homs with perfect complex commutes with S^1-invariants}.
\end{itemize}

Because $P,P^\vee$ are bounded, the sheaves $\mathcal{F}'_{\overline{k(\frakp)}}\otimes _{A_{\overline{k(\frakp)}}}P,\mathcal{G}'_{\overline{k(\frakp)}}\otimes _{A_{\overline{k(\frakp)}}}P$ are coherent,
as well as supported on the fiber over $\a$ by the discussion above,
so they belong to $\mathsf{Coh}((\mathcal{L}_{\overline{k(\frakp)}})^u_\alpha(Z_{\overline{k(\frakp)}}/G_{\overline{k(\frakp)}}))^{tS^1}$.\footnote{This can be easily seen by an argument similar to that in \zcref{lem:support-is-formal-completion}.}
Thus, we can apply \zcref{prop: Kashiwara's lemma for unipotent loop spaces} to obtain that
	 \begin{equation}
     \label{eqn: acyclic}
	 	F_{\overline{k(\frakp)}}\otimes _{A_{\overline{k(\frakp)}}}P\otimes _{A_{\overline{k(\frakp)}}}P^\vee\simeq 0.
	 \end{equation} 
	 Base changing this $A_{\overline{k(\frakp)}}$-module along $A_{\overline{k(\frakp)}}\to {\overline{k(\frakp)}}$, we obtain an
  object that belongs to the thick subcategory generated by $F_{\overline{k(\frakp)}}$, again by perfectness of $P$ over $A_{\overline{k(\frakp)}}$.
  By definition of the Koszul complex $P$, base changing it along $A_{\overline{k(\frakp)}}\to {\overline{k(\frakp)}}$ sends the Koszul differential to zero. Thus, we see that the complex \zcref[noname]{eqn: acyclic} is a direct sum of shifts of $F\otimes_A k_{\overline{k(\frakp)}}$. 
  We thus conclude that
	 \begin{align*}
	 	 F \otimes _A{\overline{k(\frakp)}} = F_{\overline{k(\frakp)}} \simeq 0,
	 \end{align*} 
  as desired.

  Finally, in order to conclude, we apply \zcref{lem: fully faithful iff fully faithful on pure tensors}.
\end{proof}

\subsubsection{Essential surjectivity}

\begin{prop}
	\label{prop: essential surjectivity in projective case}
	Let $i:Z\hookrightarrow X$ be a closed substack of a stack $X$ such that $X$ is an algebraic stack almost of finite type with affine diagonal and $i$ is almost of finite presentation. 
	Denote by $\wih{i}: Z \to \wih{X_Z}$ the inclusion of $Z$ into the formal completion,
	and by $\calL \wih{i}: \calL Z \to \calL \wih{X_Z}$ the corresponding map on loop spaces.
	Then the image of the functor $(\calL \wih{i}_*)^{tS^1}: \Coh(\calL Z)^{tS^1} \to \Coh(\calL \wih{X_Z})^{tS^1}$ thickly generates
    the target.
\end{prop}

\begin{proof}
	We begin by showing that $\Coh(\calL \wih{X_Z})^{S^1}$ is generated as a thick category
	by the image of the restriction of $(\calL \wih{i}_*)^{\tT}$ to $\Coh(\calL Z)^{S^1}$. For this,
	it suffices by \zcref{lem: thick generation preserved by base change} to show that
	$\Coh(\calL \wih{X_Z})^{S^1}$ is generated as a thick category by the image
	of the functor $\Coh(\calL Z)^{S^1} \to \Coh(\calL \wih{X_Z})^{S^1}$.
	
	From the natural isomorphism $\calL \wih{X_Z} \to \wih{\calL X_{\calL Z}}$ 
	(\zcref{lem:loops-commute-with-formal-completion}), we have
	the equivalence $\Coh(\calL \wih{X_Z}) \simeq \Coh(\wih{\calL X_{\calL Z}})$.
	We recall that pushforward along the map $\wih{\calL X_{\calL Z}} \to \calL X$ induces an
	equivalence $\Coh(\wih{\calL X_{\calL Z}}) \xrightarrow{\simeq} \Coh_{\calL Z}(\calL X) \subset \Coh(\calL X)$,
	so it suffices to show that the target of $\Coh(\calL Z)^{S^1} \to \Coh_{\calL Z}(\calL X)^{S^1}$
	induced by $\calL Z \to \calL X$ is generated as a thick category by the image.
	
	The $t$-structures on both $\Coh(\calL Z)^{S^1}$ and
	$\Coh_{\calL Z}(\calL X)^{S^1} \simeq \Coh(\wih{\calL X_{\calL Z}})^{tS^1}$ are bounded since they're induced
	by those on $\Coh(\calL Z)$ and $\Coh_{\calL Z}(\calL X)$, respectively, so it follows
	from \zcref{lem: formal lemma} that it suffices just to show that the heart $(\Coh_{\calL Z}(\calL X)^{S^1})^{\heart}$ is generated
	by the image of $\Coh(\calL Z)^{S^1}$.
	
	Note that \zcref{lem: S^1 invariants commutes with heart} and \zcref{lem: heart of derived is heart of classical} 
	together imply that
	\[(\Coh_{\calL Z}(\calL X)^{S^1})^{\heart} \simeq (\Coh_{(\calL Z)^{\cl}}((\calL X)^{\cl})^{\heart})^{S^1}.\]
	There is, of course, a functor $\oblv: (\Coh_{(\calL Z)^{\cl}}((\calL X)^{\cl})^{\heart})^{S^1} \to \Coh_{(\calL Z)^{\cl}}((\calL X)^{\cl})^{\heart}$
	which forgets the $S^1$-equivariant structure of an object on the left. Since $X$ is a $1$-stack,
	$\calL X$ is also a $1$-stack, meaning that the restriction of $\calL X: \CAlg_k^{\cn} \to \Spc$
	to classical rings $\CAlg_k^{\cl} \to \Spc$ is an ordinary stack in the sense
	of e.g. \cite{OlssonStacks}. In this case, $(\calL X)^{\cl}$ possesses a structure
	sheaf of commutative rings, $\calO_{(\calL X)^{\cl}}$. 
	The $t$-structure on $\Coh((\calL X)^{\cl})$, hence
	on $\Coh_{(\calL Z)^{\cl}}((\calL X)^{\cl})$, is compatible with the fully faithful embedding 
	$\Coh((\calL X)^{\cl}) \hook \calO_{(\calL X)^{\cl}}\mod$.
	The $t$-structure on the latter is the canonical one on any category of modules over a sheaf of commutative rings,
	so an object of $\F \in \Coh_{(\calL Z)^{\cl}}((\calL X)^{\cl})^{\heart}$ is a discrete module over $\calO_{(\calL X)^{\cl}}$.
	Moreover, because $\F$ is supported on $(\calL Z)^{\cl}$, the defining ideal of $(\calL Z)^{\cl}$,
	\[I := \ker\left(\calO_{(\calL X)^{\cl}} \to (\calL  i)_*\calO_{(\calL Z)^{\cl}}\right),\]
	is contained in $\sqrt{\Ann \F}$.
	By \zcref{lem: loop space is Noetherian}, the classical stack $(\calL X)^{\cl}$
	is Noetherian, so $I$ is locally finitely generated. For a local chart $U_a$, let $I_a = (f_{a1}, \ldots, f_{am})$ be a
	finite local presentation of $I$. Since each $f_{ai}$ belongs to $\sqrt{\Ann \F}$, there exists
	in a neighborhood of each point $p \in U_a$ an integer $n_{ai} \geq 1$ such that
	${f_{ai}}^{n_{ai}} \F = 0$. Set $N_a=\sum_i n_{ai}$, so that on the chart $U_a$, $I^N\F=0$. Since $(\calL X)^{\cl}$ is quasi-compact, we may pick a maximum such integer $N=\max N_a$
	for a finite open cover of $X$,
	so that globally $I^N \F = 0$.
	
	Thus, $\F$ possesses a finite filtration $\F \supset I \F \supset \ldots \supset I^N \F$, each
	of whose associated graded pieces are annihilated by $I$. Thus, each associated graded
	piece $I^i \F/I^{i+1} \F$ belongs to the image of $i_*: \Coh((\calL Z)^{\cl}) \to \Coh((\calL X)^{\cl})$.
	Since each term of the filtration on $\F$ fits into an exact sequence,
	\[0 \to I^{i+1} \F \to I^i \F \to I^i \F/I^{i+1}\F \to 0,\]
	we see by induction that $\F$ itself belongs to the thick subcategory
	of $\Coh((\calL X)^{\cl})$ generated by the image of $i_*$.
	
	Now, $I$ is an $S^1$-equivariant subsheaf of $\calO_{(\calL X)^{\cl}}$.
	Thus, the argument above can be performed $S^1$-equivariantly to
	show that every $\wit{F} \in (\Coh_{(\calL Z)^{\cl}}((\calL X)^{\cl})^{\heart})^{S^1}$ is in the thick subcategory generated
	by the image of $(\calL i^{\cl}_*)^{S^1}: \Coh((\calL Z)^{\cl})^{S^1} \to \Coh((\calL X)^{\cl})^{S^1}$.
	This in turn finishes the argument that $\Coh_{\calL Z}(\calL X)^{S^1}$ is thickly
	generated by the image of $\Coh(\calL Z)^{S^1}$ under $(\calL i_*)^{S^1}$.
	
	To conclude the proof, we apply \zcref{lem: thick generation preserved by base change}
	to base change along the map $k[u] \to \ku$ and the functor
	$(\calL i_*)^{S^1}: \Coh(\calL Z)^{S^1} \to \Coh_{\calL Z}(\calL X)^{S^1}$.
\end{proof}

\begin{lem}
	\label{lem: thick generation preserved by base change}
	Let $R$ be an $\bbE_{\infty}$-algebra, and let $R \to R'$ be an $R$-algebra.
	Suppose that $F: \scrC \to \scrD$ is an $R$-linear functor such that $\scrD$
	is thickly generated by the image of $F$. Then $\scrD \otimes_R R'$ is thickly
	generated by the image of $F \otimes_R R'$.
\end{lem}

\begin{proof}
	Indeed, since $-\otimes_R R'$ commutes with colimits, we see that the thick subcategory generated by the image of $F\otimes_R R'$ is precisely $\mathscr{D}\otimes_R R'$, since $\mathscr{D}$ is the thick subcategory generated by the image of $F$.
\end{proof}

\begin{lem}
	\label{lem: heart of derived is heart of classical}
	Let $X$ be a coherent derived algebraic $1$-stack,
	and let $Z \hook X$ be a closed derived substack.
	Then, pushforward along the natural map $j:\iota(X^{\cl})\footnotemark 
	\footnotetext{Here $\iota$ denote the functor from classical
		stacks to derived stacks given by left Kan extension along
		the embedding $\CAlg_k^{\cl} \hook \CAlg^{\cn}$.} \to X$ 
	induces an equivalence 
	\begin{equation}
		\label{eqn: equivalence of Cohs}
		\Coh_{Z^{\cl}}(X^{\cl})^{\heart} \xrightarrow{\simeq} (\Coh_Z(X))^{\heart}.
	\end{equation}
	Moreover, if $X$ is equipped with an $S^1$-action such that
	$Z$ is an $S^1$-invariant substack,
	then $X^{\cl}$ possesses a natural induced $S^1$-action from that on $X$; the canonical $t$-structures
	on $\Coh(X)$ and $\Coh(X^{\cl})$ are compatible with the induced $S^1$-actions on $\Coh(X)$
	and $\Coh(X^{\cl})$; and the equivalence \zcref[noname]{eqn: equivalence of Cohs} is canonically $S^1$-equivariant.
\end{lem}

\begin{proof}
		We first check that pushforward induces an equivalence $j_*:\mathsf{Coh}(X^{\cl})^\heartsuit\xrightarrow{\cong}\mathsf{Coh}(X)^\heartsuit.$ 
		The proof of \cite[Proposition 2.5.9.1]{SAG} applies to the map $j:X^{\cl}\to X$ locally
		in the smooth topology on the target. 
		Indeed, if $U=\Spec A \to X$ is a smooth affine chart, 
		then its base change $j_A: \Spec H^0(A)\to \Spec A$ has 1-connective fiber. This implies that 
		$j_{A*}:\QCoh(\Spec A)^\heart \to \QCoh(\Spec H^0(A))^\heart$ is an equivalence. 
		Moreover, it is apparent from
		the proof of \cite[Proposition 2.5.9.1]{SAG} that $j_{A*}$ 
		preserves the subcategory of finitely generated objects, 
		so the equivalence restricts to coherent objects. Since being coherent and
		being in the heart are local properties in the smooth topology, 
		it follows that $j_*$ is an equivalence, as desired.\footnotemark
		\footnotetext{The functors $\Psi_X,\Psi_{X^{\cl}}$ of \cite[Chapter 4 Section 1.1.2]{GR17} 
		(extended to stacks in Chapter 5 of \textit{op. cit.}) are $t$-exact 
		equivalences when restricted to the hearts by Proposition 1.2.2 of \textit{loc. cit.}}
	
	We now check that we obtain the induced equivalence (\ref{eqn: equivalence of Cohs}). We claim to have a Cartesian diagram
	\begin{equation}\label{cd:classical-open-complement}
		\begin{tikzcd}
			U^{\cl}\arrow[r,"i_{U^{\cl}}"]\arrow[d,"j_U"]&X^{\cl}\arrow[d,"j"]\\
			U\arrow[r,"i_U"]&X
		\end{tikzcd}.
	\end{equation}
	Indeed, there is a natural map $U^{\cl}\to U\times _X X^{\cl}$. To see that it is an equivalence,
	it is sufficient to check the following. Let $x:S\to X$ be an $S$-point of $X$. 
	Then $S\times _X Z=\emptyset$ if and only if $S\times _X Z^{\cl}=\emptyset$. 
	One direction is clear, to see the other, it suffices to observe that a stack $Y$ is empty 
	if and only if its classical truncation is. Indeed, any $T $-point $T\to Y$ restricts to a point 
	$T^{\cl}\to Y^{\cl}$, and $T^{\cl}$ is nonempty whenever $T$ is nonempty. Thus,
	\begin{align*}
		U^{\cl}\xrightarrow{\simeq}U\times _X X^{\cl}
	\end{align*} 
	is an equivalence as desired. We can then use base change for the diagram (\ref{cd:classical-open-complement}) 
	to see that if $\calF\in \mathsf{Coh}_{Z^{\cl}}(X^{\cl})^\heart$, then $i_U^!j_*\calF\cong 0$. 
	Thus, we see that $j_*$ defines a functor
	\begin{align*}
		\mathsf{Coh}_{Z^{\cl}}(X^{\cl})^\heart\to \mathsf{Coh}_Z(X)^\heart.
	\end{align*} 
	Clearly, the functor is fully faithful, because in the commutative diagram
	\[
	\begin{tikzcd}
		\mathsf{Coh}_{Z^{\cl}}(X^{\cl})^\heart\arrow[r]\arrow[d]& \mathsf{Coh}_Z(X)^\heart\arrow[d]\\
		\mathsf{Coh}(X^{\cl})^\heart\arrow[r,"\cong"]&\mathsf{Coh}(X)^\heart
	\end{tikzcd},
	\]
	the vertical arrows are fully faithful inclusions and the bottom horizontal arrow is an equivalence. 
	To see it is essentially surjective, suppose $\calF\in \mathsf{Coh}_Z(X)^\heart$. 
	Under the inclusion $\mathsf{Coh}_Z(X)^\heart \hookrightarrow \mathsf{Coh}(X)^\heart$ and the equivalence 
	\begin{align*}
		\mathsf{Coh}(X^{\cl})^\heart\xrightarrow{\cong}\mathsf{Coh}(X)^\heart,
	\end{align*} 
	there is an $\calF_0\in \mathsf{Coh}(X^{\cl})^\heart$ such that $j_*\calF_0= \calF$. 
	It is sufficient to show that $j_U^!j_*\calF_0\cong 0$, which follows from base change along (\ref{cd:classical-open-complement}).
	
	If $X$ is equipped with an $S^1$-action such that $Z $ is an $S^1$-invariant substack, 
	then there is an induced $S^1 $-action on $ X^{\cl}$ by pulling back the action map 
	$S^1\times X\to X$ to $X^{\cl}$. Clearly, the map $X^{\cl}\to X$ is $S^1$-equivariant. 
	The truncation functors on $\mathsf{QCoh}(X)$ and $\mathsf{QCoh}(X^{\cl})$ 
	are $S^1$-equivariant by \zcref{lem:truncation functors are equivariant}. Thus, the $S^1$-equivariant functor $j_*$ preserves the hearts as desired.
\end{proof}

\begin{lem}
	\label{lem: formal lemma}
	Suppose that $\scrC$ and $\scrD$ are stable $\infty$-categories such that
	$\scrD$ is equipped with a bounded $t$-structure. Then $\scrD$ is thickly generated by the image of
	a functor $F: \scrC \to \scrD$ if and only if its heart $\scrD^{\heart}$ is.
\end{lem}

\begin{proof}
	One direction is clear, so suppose $\mathscr{D}^\heart$ is generated by $\im(F)$. 
	Denote by $\mathscr{D}'$ the category thickly generated by $\im(F)$. 
	Let $d\in \mathscr{D}$. Since the $t$-structure on $\mathscr{D}$ is bounded, 
	there are integers $\ell \le k$ such that $d\in \mathscr{D}^{\ge \ell}\cap \mathscr{D}^{\le k}$.
	We proceed by induction on $n=k-\ell$. If $n=0$, then $d \simeq H^{\ell}(d)[-\ell]$.
	Since $H^{\ell}(d)\in \mathscr{D}^\heart \subset \mathscr{D}'$ 
	and $\mathscr{D}'$ is closed under shifts, we see that $d\in \mathscr{D}'$. For $n>0$, use the triangle
	\begin{align*}
		\tau^{\le k-1}d\to d\to H^k(d)[-k].
	\end{align*} 
	The first term lies in $\mathscr{D}'$ by induction. 
	The last term also lies in $\mathscr{D}'$ because 
	$H^k(d)\in \mathscr{D}^\heart$. Since $\mathscr{D}'$ 
	is closed under extensions, we conclude.
\end{proof}

\begin{lem}
	\label{lem: loop space is Noetherian}
	Suppose that $X$ is a Noetherian derived algebraic $1$-stack. Then
	$(\calL X)^{\cl}$ is a Noetherian classical stack. In particular, $(\calL X)^{\cl}$
	is quasi-compact.
\end{lem}

\begin{proof}
	Denote $X^{\cl}\eqcolon X_0$. Classical truncation preserves limits, as it is a right adjoint. Hence,
	\begin{align*}
		(\mathcal{L}X)^{\cl}\simeq X_0\times _{X_0\times X_0}X_0,
	\end{align*} 
	and the canonical map
	\begin{align*}
		p:(\mathcal{L}X)^{\cl}\to X_0
	\end{align*} 
	is a base change of the diagonal $\Delta_0$ of $X_0$. Now, $X_0$ is a classical algebraic stack; hence, its diagonal is representable by algebraic spaces and is locally of finite type. Moreover, $X_0$ is Noetherian; hence, its diagonal is qcqs, since $X_0$ is in particular quasi-separated. This immediately gives that $(\mathcal{L}X)^{\cl}$ is locally Noetherian. It is also qcqs, since this holds for $p$ and $X_0$. Thus, it is Noetherian as desired.
\end{proof}

\section{The linearly scalloped case}
\label{sec:the linearly scalloped case}

\subsection{Nisnevich descent}
\label{ssec:nisnevich descent}
In this section, we prove that potent
$\calD$-modules on linearly scalloped stacks satisfy
Kashiwara's Lemma.
Our proof relies heavily on the fact that $\Pot(-)$ satisfies
Nisnevich descent, which we prove below after
recalling the following technical lemma. Let $\mathsf{Rep}_{/X}$ denote the category of derived stacks $Y$ with affine diagonal over $X$ for which the structure map $Y\to X$ is representable by algebraic spaces and of finite presentation. 
\begin{rem}
	The category $\mathsf{Rep}_{/X}$ we define here differs slightly from the one in \cite{KhanRavi}, where they do not require objects to have affine diagonal.
\end{rem}

\begin{prop}\label{prop:potent-d-modules-nis-descend}
	Let $X$ be an algebraic stack almost of finite type and with affine diagonal\footnote{This condition is necessary for potent $\mathcal{D}$-modules to define a presheaf on $\Rep_{/X}$ \cite{ES-6FF}.} Then the presheaf $\Pot: \Rep_{/X} \to \DGCat_{\ku}$ satisfies descent
	with respect to the Nisnevich topology on $\Rep_{/X}$.
\end{prop}

\begin{proof}
	It suffices to show that $\Pot(-)$ is Nisnevich excisive, i.e. 
	takes Nisnevich squares\footnotemark
	\footnotetext{For the definition of Nisnevich square, see, \cite[Definition 2.4]{KhanRavi}.} 
	to pullback diagrams, by \cite[Theorem 2.2.7]{KhanLocalization}. 

	Now, given a Nisnevich square
	\[
	\begin{tikzcd}
		W\arrow[r,"j'"]\arrow[d,"p'"]&V\arrow[d,"p"]\\
		U\arrow[r,"j"]&S
	\end{tikzcd},
	\] 
	where we recall that $j$ is an immersion and $p$ is an \'etale
    map that restricts to an isomorphism along the complement to $j$. From this, we obtain another Nisnevich square
	\begin{equation}
    \label{eqn: Nisnevich descent}
	\begin{tikzcd}
			\mathcal{L}W\arrow[r,"\mathcal{L}j'"]\arrow[d,"\mathcal{L}p'"]&\mathcal{L}V\arrow[d,"\mathcal{L}p"]\\
		\mathcal{L}U\arrow[r,"\mathcal{L}j"]&\mathcal{L}S
	\end{tikzcd},
	\end{equation}
	since $\mathcal{L}$ is a Nisnevich cosheaf by \cite{BouazizKhan}. We note that each morphism
    in \zcref[noname]{eqn: Nisnevich descent} is canonically $S^1$-equivariant. Since $\mathsf{IndCoh}$ satisfies fppf, 
    hence Nisnevich, descent (\cite[Corollary 10.4.5]{GaitsgoryIndCoh}), we obtain
    an $S^1$-equivariant equivalence,
	\begin{align*}
		\mathsf{IndCoh}(\mathcal{L}S)\xrightarrow{\simeq} \mathsf{IndCoh}(\mathcal{L}V)\times _{\mathsf{IndCoh}(\mathcal{L}W)}\mathsf{IndCoh}(\mathcal{L}U).
	\end{align*}

    We conclude by checking that the set-up above satisfies the conditions of \zcref{lem: tTate of fiber product}.
    By \cite[Proposition 5.30]{CautisWilliams}, the $t$-structure on $\IndCoh(\calS)$ for any coherent ind-geometric stack
    $\calS$ is regular, meaning in particular that $\IndCoh(\calS)$ is compactly generated.
    We claim that the functors $(\calL j')^!:\IndCoh(\calL V) \to \IndCoh(\calL W)$ and $(\calL p')^!: \IndCoh(\calL U) \to \IndCoh(\calL W)$
    are each $t$-exact and preserve compact objects. The morphism $j'$ is an open immersion as the base change of an open immersion, so $\calL j'$ is an open immersion
    by \zcref{lem: loops of closed immersion is closed immersion}. In particular, $\calL j'$ is smooth. On the other hand
    $\calL p$ is \'etale by virtue of \zcref[noname]{eqn: Nisnevich descent} being a Nisnevich square, so $\calL p'$ is
    \'etale as the base change of an \'etale morphism. In particular, $\calL p'$ is smooth. Thus, by \zcref{cor: smooth pullbacks preserve Coh},
    $(\calL p')^!$ and $(\calL j')^!$ each preserve compact objects, and they are $t$-exact by \zcref{lem: etale morphisms have t-exact !-pullbacks}. 
    Finally, $(\calL j')^!$ has a fully faithful $t$-bounded right adjoint by \zcref{lem: open immersion pushforward is fully faithful} and \zcref{lem: open immersion pushfoward is t-bounded}.
	\end{proof}

\begin{lem}
\label{lem: t-structure on fiber product}
	Let $\mathscr{A}\xrightarrow{F} \mathscr{C}\xleftarrow{G}  \mathscr{B}$ be a span in $\bReg_t$, 
    such that $G:\mathscr{B}\to \mathscr{C}$ has a fully faithful right adjoint in $\Pr^L$, and such that
    $F$ and $G$ are $t$-exact and preserve compact objects. 
    Assume that $\scrA \times_{\scrC} \scrB$ is
    compactly generated. 
    Then the lax product $\scrA \vec{\times}_{\scrC} \scrB$\footnotemark
    \footnotetext{For the definition of this, see \cite[Definition 5]{Tamme-excision}.} 
    and the product $\scrA \times_{\scrC} \scrB$ have canonical regular $t$-structures.
\end{lem}

\begin{proof}

    We first prove the claim that the canonical $t$-structure on
    the lax product $\scrA \vec{\times}_{\scrC} \scrB$ is coherent.

    We begin by considering the pointwise $t$-structure on $\Coh(\scrC)^{[1]}$, where $[1]$ is the 
    standard $1$-simplex. By \cite[Lemma 2.1.9(iv)]{HLPreygel}, the products $\Coh(\scrA) \times \Coh(\scrB)$
    and $\Coh(\scrC) \times \Coh(\scrC)$ have canonical $t$-structures such that the projection
    morphisms are $t$-exact. By the same, the lax product $\Coh(\scrA) \vec{\times}_{\Coh(\scrC)} \Coh(\scrB)
    \coloneq (\Coh(\scrA) \times \Coh(\scrB)) \times_{\Coh(\scrC) \times \Coh(\scrC)} \Coh(\scrA) \times \Coh(\scrB)$ 
    possesses a unique $t$-structure such that the projection functors to $\Coh(\scrC)^{[1]}$, $\Coh(\scrA) \times \Coh(\scrB)$,
    and $\Coh(\scrC) \times \Coh(\scrC)$. By \cite[Proposition 2.13]{AntieauGepnerHeller},
    this $t$-structure induces one on the ind-completion $\Ind(\Coh(\scrA) \vec{\times}_{\Coh(\scrC)} \Coh(\scrB))$,
    which is compatible with filtered colimits, and right complete if the $t$-structure on
    $\Coh(\scrA) \vec{\times}_{\Coh(\scrC)} \Coh(\scrB)$ is bounded below.

    Since $\scrA$, $\scrB$, and $\scrC$ belong to $\bCoh_t$,
    their coherent subcategories $\Coh(\scrA)$, $\Coh(\scrB)$, and $\Coh(\scrC)$
    have bounded $t$-structures. Hence, the categories $\Coh(\scrC)^I$,
    $\Coh(\scrA) \times \Coh(\scrB)$, and $\Coh(\scrC) \times \Coh(\scrC)$
    inherit bounded $t$-structures.
    From this, it is clear that the $t$-structure
    on the fiber product $\Coh(\scrA) \vec{\times}_{\Coh(\scrC)} \Coh(\scrB)$ is
    also bounded---in particular, it is bounded below.

    In order to show that the $t$-structure on $\Ind(\Coh(\scrA) \vec{\times}_{\Coh(\scrC)} \Coh(\scrB))$
    is coherent, it therefore remains to show that its heart is a locally coherent abelian category,
    meaning that it is compactly generated (as an abelian category) and that the subcategory of compact
    objects in the heart is abelian.
    From the construction of the $t$-structure, we have 
    \begin{equation}
    \label{eqn: regular t-structure lemma 2}
    \Ind(\Coh(\scrA) \vec{\times}_{\Coh(\scrC)} \Coh(\scrB))^{\heart} \simeq \Ind_1\left(\left(\Coh(\scrA) \vec{\times}_{\Coh(\scrC)} \Coh(\scrB)\right)^{\heart}\right)
    \end{equation}
    where $\Ind_1$ denotes the 1-categorical ind-construction.
    Because $\Coh(\scrA) \vec{\times}_{\Coh(\scrC))} \Coh(\scrB)$ is
    a limit in $\Cat_{\infty}^{\sm}$, it is an $\omega$-small
    stable category. In particular, any of its subcategories is
    also essentially small. Thus, $\left(\Coh(\scrA) \vec{\times}_{\Coh(\scrC)} \Coh(\scrB)\right)^{\heart}$
    is a skeletally small abelian category. Thus, by
    \cite[Theorem 4.1]{PrestAbelianCat}, its $\Ind_1$-category
    is a Grothendieck locally coherent abelian category whose
    subcategory of compact objects is precisely
    $\left(\Coh(\scrA) \vec{\times}_{\Coh(\scrC)} \Coh(\scrB)\right)^{\heart}$,
    which shows that the $t$-structure on
    $\Ind(\Coh(\scrA) \vec{\times}_{\Coh(\scrC)} \Coh(\scrB))$ is coherent.
    To conclude that the $t$-structure on $\scrA \vec{\times}_{\scrC} \scrB$
    is coherent, we note that
    \begin{align*}
    \Ind(\Coh(\scrA) \vec{\times}_{\Coh(\scrC)} \Coh(\scrB)) &\simeq \Ind(\Coh(\scrA)) \vec{\times}_{\Ind(\Coh(\scrC))} \Ind(\Coh(\scrB)) \\
                                                             &\simeq \scrA \vec{\times}_{\scrC} \scrB,
    \end{align*}
    by \cite[Proposition 13]{Tamme-excision}.

    The exact same argument as above shows that the canonical $t$-structure on 
    $\Ind(\Coh(\scrA) \times_{\Coh(\scrC)} \Coh(\scrB))$ is coherent. Since the analogue
    of \cite[Proposition 13]{Tamme-excision} for products is not true in general, it remains to show in our case
    that $\Ind(\Coh(\scrA) \times_{\Coh(\scrC)} \Coh(\scrB)) \xrightarrow{\simeq} \scrA \times_{\scrC} \scrB$ in order
    to prove that the $t$-structure on the latter is coherent.

    To begin, we note that, because $G$ has a fully faithful adjoint, there exists an exact
    sequence $\scrA \times_{\scrC} \scrB \xrightarrow{\i} \scrA \vec{\times}_{\scrC} \scrB \xrightarrow{e} \scrC$
    by \cite[Corollary 12]{Tamme-excision}. Since $F$ and $G$ preserve compact objects
    by assumption, the latter functor restricts to a functor $(\scrA \vec{\times}_{\scrC} \scrB)^{\omega}
    \to \scrC^{\omega}$. At the same time, by \cite[Proposition 13]{Tamme-excision},
    there is an equivalence $(\scrA \vec{\times}_{\scrC} \scrB)^{\omega} \simeq \scrA^{\omega} \vec{\times}_{\scrC^{\omega}} \scrB^{\omega}$.
    Since the $t$-structures on $\scrA$, $\scrB$, and $\scrC$ are regular,
    their subcategories of compact objects coincide with their subcategories
    of coherent objects, so we see that the functor $e^{\omega}: \Coh(\scrA) \vec{\times}_{\Coh(\scrC)} \Coh(\scrB)
    \to \Coh(\scrC)$ is the restriction of the functor $e: \scrA \vec{\times}_{\scrC} \scrB \to \scrC$
    to compact objects. Now, observe that the kernel of $e^{\omega}$ is
    precisely the ordinary product, $\Coh(\scrA) \times_{\Coh(\scrC)} \Coh(\scrB)$.
    We now claim that $\ker(e)^{\omega} \simeq \ker(e^{\omega})$. Indeed,
    since $e$ is a localization, it has a right adjoint $e^R$. Since $e$ preserves
    compact objects (because each of $F$ and $G$ does), $e^R$ preserves filtered colimits.
    Since $\i: \scrA \times_{\scrC} \scrB \to \scrA \vec{\times}_{\scrA} \scrB$ is the
    kernel of the map $e$, its right adjoint $\i^R$ given object-wise by
    $\i^R(X) \simeq \ker(X \to e^R e(X))$. Since $e^R$ preserves filtered colimits,
    and $e$ preserves colimits as a left-adjoint functor, we find that, for $X = \colim_n X_n$,
    \begin{align*}
        \i^R(X) = \i^R(\colim_n X_n) &\simeq \ker(\colim_n X_n \to e^R e(\colim_n X_n)) \\
                                     &\simeq \ker(\colim_n(X_n \to e^R e(X_n))) \\
                                     &\simeq \colim_n \ker(X_n \to e^R e(X_n)),      
    \end{align*}
    where the last equivalence follows from the fact that in a stable category, kernels are both
    limits and colimits. Since $\i^R$ thus preserves filtered colimits, $\i$ preserves compact objects,
    meaning $\ker(e)^{\omega} \simeq \ker(e) \cap (\scrA \vec{\times}_{\scrC} \scrB)^{\omega}$, while
    the latter object is precisely $\ker(e|_{(\scrA \vec{\times}_{\scrC} \scrB)^{\omega}}) = \ker(e^{\omega})$. 
    Since $\ker(e) \simeq \scrA \times_{\scrC} \scrB$, is compactly
    generated by assumption, we finally obtain that $(\scrA \times_{\scrC} \scrB)^{\omega} \simeq \Coh(\scrA) \times_{\Coh(\scrC)} \Coh(\scrB)$,
    which implies that the equivalence $\Ind(\Coh(\scrA) \times_{\Coh(\scrC)} \Coh(\scrB)) \xrightarrow{\simeq} \scrA \times_{\scrC} \scrB$ as desired.

    By \cite[Lemma 5.3.2.9]{HA}, ind-completion
    furnishes an equivalence $\Ind(-): \Cat_{\infty}^{\idm} \to \Pr^L_{\omega}$,
    where $\Cat_{\infty}^{\sm}$ denotes the category of $\omega$-small, idempotent-complete
    $\infty$-categories, and $\Pr^L_{\omega}$ denotes the category of
    $\omega$-presentable, stable $\infty$-categories. In particular,
    this implies that the universal functor,
    \begin{align*}
    \Ind(\Coh(\scrA) \times_{\Coh(\scrC)} \Coh(\scrB)) &\xrightarrow{\simeq} \Ind(\Coh(\scrA)) \times_{\Ind(\Coh(\scrC))} \Ind(\Coh(\scrB)) \\
                                                       &\simeq \scrA \times_{\scrC} \scrB
    \end{align*}
    is an equivalence.

    In order to show that the $t$-structure on $\scrA \vec{\times}_{\scrC} \scrB$
    is regular, we need to show that 
    \[\Coh(\scrA) \vec{\times}_{\Coh(\scrC)} \Coh(\scrB)
    \simeq \Coh(\scrA \vec{\times}_{\scrC} \scrB).\]
    By \cite[Lemma 6.2.5]{BZNP}, since the $t$-structure on
    $\scrA \vec{\times}_{\scrC} \scrB$ is coherent, an object
    $X \in \scrA \vec{\times}_{\scrC} \scrB$ belongs to $\Coh(\scrA \vec{\times}_{\scrC} \scrB)$
    if and only if $H_n X \in \Coh(\scrA \vec{\times}_{\scrC} \scrB)^{\heart}$ for all $n$, 
    and $H_n X = 0$ for all but finitely many $n$.
    By the same, we also know that $\Coh(\scrA \vec{\times}_{\scrC} \scrB)^{\heart}$ consists
    precisely of the compact objects in $(\scrA \vec{\times}_{\scrC} \scrB)^{\heart}$.
    By the characterization \zcref[noname]{eqn: regular t-structure lemma 2} above, 
    we have 
    \[\left((\scrA \vec{\times}_{\scrC} \scrB)^{\heart}\right)^{\omega} \simeq \left(\Coh(\scrA) \vec{\times}_{\Coh(\scrC)} \Coh(\scrB)\right)^{\heart}.\]
    That is, $X \in \scrA \vec{\times}_{\scrC} \scrB$ belongs to $\Coh(\scrA \vec{\times}_{\scrC} \scrB)$
    if and only if 
    \begin{enumerate}
        \item $H_n X \in \left(\Coh(\scrA) \vec{\times}_{\Coh(\scrC)} \Coh(\scrB)\right)^{\heart}$,
        for all $n$,
        \item and $H_n X = 0$ for all but finitely many $n$.
    \end{enumerate} 
    Let $X \in \Coh(\scrA \vec{\times}_{\scrC} \scrB)$.
    Because $\Coh(\scrA \vec{\times}_{\scrC} \scrB)$ has
    a bounded $t$-structure, condition (2) implies that $X$ is equivalent to 
    one of its truncations $\tau_{\leq n} X$, for some $n$. On the other hand,
    if $m$ is the least integer such that $H_m X \neq 0$, then $H_m(X)[m] \simeq \tau_{\leq m} X$.
    Since each truncation is obtained from the previous one by the fiber sequence,
    \[H_{n+1}(X)[n+1] \longrightarrow \tau_{\leq n+1}X \longrightarrow \tau_{\leq n}X,\]
    we see that $X$ is obtained by a series of extensions by shifts of its homology objects.
    Since $H_n X \in \left(\Coh(\scrA) \vec{\times}_{\Coh(\scrC)} \Coh(\scrB)\right)^{\heart}$,
    for all $n$ by condition (1), and since $\Coh(\scrA) \vec{\times}_{\Coh(\scrC)} \Coh(\scrB)$ is a stable subcategory, this
    implies that $X \in \Coh(\scrA \vec{\times}_{\scrC} \scrB)$.
    On the other hand, it is obvious that every object of $\Coh(\scrA) \vec{\times}_{\scrC} \scrB$
    satisfies conditions (1) and (2). Since $\Coh(\scrA \vec{\times}_{\scrC} \scrB)$
    and $\Coh(\scrA) \vec{\times}_{\Coh(\scrC)} \Coh(\scrB)$ are full subcategories
    of $\scrA \vec{\times}_{\scrC} \scrB$ containing the same objects, they are equivalent.

    The argument that the $t$-structure on $\scrA \times_{\scrC} \scrB$ is regular proceeds in exactly the same way.
\end{proof}

\begin{lem}
\label{lem: tTate of fiber product}
	Let $\mathscr{A}\xrightarrow{F} \mathscr{C}\xleftarrow{G}  \mathscr{B}$ be a span in $\bReg_t^{BS^1}$, 
    such that $G:\mathscr{B}\to \mathscr{C}$ has a fully faithful, $t$-bounded right adjoint in $\Pr^L$, 
    and $F$ and $G$ are $t$-exact and preserve compact objects. 
    If $\mathscr{A}\times_{\mathscr{C}}\mathscr{B}$ is compactly generated, then
	\begin{align*}
		\mathscr{A}^{\tT}\times _{\mathscr{C}^{\tT}}\mathscr{B}^{\tT}\simeq(\mathscr{A}\times _{\mathscr{C}}\mathscr{B})^{\tT},
	\end{align*} 
    where the renormalized Tate invariants on the right-hand side 
    are taken with respect to the coherent $t$-structure defined 
    in \zcref{lem: t-structure on fiber product}, and all fiber products
    are taken in $\Pr^L$.
\end{lem} 
\begin{proof}    
    
	We first prove that
	\begin{align*}
		(\mathscr{A}\vec{\times }_{\mathscr{C}}\mathscr{B})^{\tT}\simeq \mathscr{A}^{\tT}\vec{\times }_{\mathscr{C}^{\tT}}\mathscr{B}^{\tT},
	\end{align*}
    where $\tT$ on the left-hand side is taken with respect to the regular $t$-structure 
    defined in \zcref{lem: t-structure on fiber product}.
	Since $F$ and $G$ preserve compact objects and $\mathscr{A}$ and $\mathscr{B}$ are compactly generated, \cite[Proposition 13]{Tamme-excision} gives an equivalence
    \[\Coh(\scrA \vec{\times}_{\scrC} \scrB) \simeq (\scrA \vec{\times}_{\scrC} \scrB)^{\omega} \simeq \scrA^{\omega} \vec{\times}_{\scrC^{\omega}} \scrB^{\omega} \simeq \Coh(\scrA) \vec{\times}_{\Coh(\scrC)} \Coh(\scrB).\]
    Since $S^1$-invariants are defined by a limit operation, and limits commute with one another,
    we obtain an equivalence 
    \begin{equation}
    \label{eqn: established above}
    \Coh(\scrA \vec{\times}_{\scrC} \scrB)^{S^1} \simeq \Coh(\scrA)^{S^1} \vec{\times}_{\Coh(\scrC)^{S^1}} \Coh(\scrB)^{S^1}.\footnotemark
    \end{equation}
    \footnotetext{Recall that an $S^1$-action on a category $\scrE$ restricts to an $S^1$-action on any full
    subcategory of $\scrE$.}
    Now observe that the induced functors $\Ind({F|_{\Coh(\scrA)}}^{S^1}): \Ind(\Coh(\scrA)^{S^1}) \to \Ind(\Coh(\scrC)^{S^1})$
    and $\Ind({G|_{\Coh(\scrA)}}^{S^1}): \Ind(\Coh(\scrB)^{S^1}) \to \Ind(\Coh(\scrC)^{S^1})$ obviously preserve compact objects.
    Thus, we have the sequence of equivalences,
    \begin{align}
         \Ind\left(\Coh(\scrA)^{S^1}\right) \vec{\times}_{\Ind\left(\Coh(\scrC)^{S^1}\right)} \Ind\left(\Coh(\scrB)^{S^1}\right) &\simeq \Ind\left(\Coh(\scrA)^{S^1} \vec{\times}_{\Coh(\scrC)^{S^1}} \Coh(\scrB)^{S^1}\right) \label{eqn: invariant-1} \\
         &\simeq \Ind\left(\left(\Coh(\scrA) \vec{\times}_{\Coh(\scrC)} \Coh(\scrB)\right)^{S^1}\right) \label{eqn: invariant-2} \\
         &\simeq \Ind\left(\Coh\left(\scrA \vec{\times}_{\scrC} \scrB\right)^{S^1}\right) \label{eqn: invariant-3},
    \end{align}
    where \zcref[noname]{eqn: invariant-1} follows from a second application of \cite[Proposition 13]{Tamme-excision}; 
    \zcref[noname]{eqn: invariant-2} is the equivalence \zcref[noname]{eqn: established above} established above; and \zcref[noname]{eqn: invariant-3} follows
    from the regularity of the $t$-structure on $\scrA \vec{\times}_{\scrC} \scrB$.
    Since $\ku\mod$ is dualizable as a $k[u]$-linear category, base change along $k[u]\mod \to \ku\mod$ preserves limits.
    Using this, we deduce from the sequence of equivalences 
    \zcref[noname]{eqn: invariant-1}--\zcref[noname]{eqn: invariant-3} above the equivalence,
    \[(\scrA \vec{\times}_{\scrC} \scrB)^{\tT} \simeq \scrA^{\tT} \vec{\times}_{\scrC^{\tT}} \scrB^{\tT},\]
    as desired.
	
    Now, since $G$ has a fully faithful right adjoint, denoted $G^R$, we 
    apply \cite[Corollary 12]{Tamme-excision} to obtain a localization sequence of
    $S^1$-equivariant functors,
	\begin{align*}
		\mathscr{A}\times _{\mathscr{C}}\mathscr{B}\xlongrightarrow{i}\mathscr{A}\vec \times _{\mathscr{C}}\mathscr{B}\xlongrightarrow{\pi}\mathscr{C}.
	\end{align*} 
	In order to apply \zcref{thm: tTate preserves localization sequences}
    and obtain a localization after applying $(-)^{\tT}$,
    we show that $\pi^R$ and $i^R$ are $t$-bounded.
    Using the decomposition $\pi^R = r \circ G^R$ found therein,\footnotemark
    \footnotetext{For the definition of $r: \scrB \to \scrA \vec{\times}_{\scrC} \scrB$, see the bottom of page
    6 of \cite{Tamme-excision}.}
    we see that $\pi^R$ is $t$-bounded if and only if $G$ and $G^R$ are (as it is clear from the definition
    that $r$ is $t$-bounded if and only if $G$ is). Since both $G$ and $G^R$ are $t$-bounded by
    assumption, it follows that $\pi^R$ is $t$-bounded.
    On the other hand, the functor $i$ is a $t$-exact inclusion, so its right adjoint $i^R$ is left $t$-exact,
    hence left $t$-bounded. To obtain a right $t$-bound on $i^R$, consider the fiber sequence
	\begin{align*}
		ii^R \xlongrightarrow{} \id \xlongrightarrow{} \pi^R\pi.
	\end{align*} 
    Since $\pi^R$ and $\pi$ are $t$-bounded, we see that the fiber $ii^R$ is $t$-bounded.
    Since, $i$ furthermore is $t$-exact, we see that $i^R$ is $t$-bounded. 
Thus, \zcref{thm: tTate preserves localization sequences} applies, and we obtain a localization sequence
	\begin{equation}
    \label{eqn: localization sequence lax Tate}
		\left(\mathscr{A}\times _{\mathscr{C}}\mathscr{B}\right)^{\tT} \xlongrightarrow{} (\mathscr{A}\vec \times _{\mathscr{C}}\mathscr{B})^{\tT} \xlongrightarrow{\pi^{\tT}} \mathscr{C}^{\tT}
	\end{equation}
    Under the identification $(\mathscr{A}\vec \times _{\mathscr{C}}\mathscr{B})^{\tT} \simeq \mathscr{A}^{\tT}\vec\times _{\mathscr{C}^{\tT}}\mathscr{B}^{\tT}$,
    the right-hand functor $\pi^{\tT}$ corresponds to the canonical map 
    \[\pi: \scrA^{\tT} \vec{\times}_{\scrC^{\tT}} \scrB^{\tT} \to \scrC^{\tT},\]
    from the lax product of the span $\scrA^{\tT} \xrightarrow{} \scrC^{\tT} \xleftarrow{} \scrB^{\tT}$
    to the base $\scrC^{\tT}$.
    The fiber of this map is precisely $\scrA^{\tT} \times_{\scrC^{\tT}} \scrB^{\tT}$, so
    we conclude from the localization sequence \zcref[noname]{eqn: localization sequence lax Tate} 
    that $\scrA^{\tT} \times_{\scrC^{\tT}} \scrB^{\tT} \simeq (\scrA \times_{\scrC} \scrB)^{\tT}$, as desired.
\end{proof}

\begin{rem}
    The necessity of using the lax product above boils down to the following.
    If we tried to
    prove \zcref{lem: tTate of fiber product} without the use of the lax product,
    we would need to know that 
    \[\Ind(\Coh(\scrA)^{S^1} \times_{\Coh(\scrC)^{S^1}} \Coh(\scrB)^{S^1})
    \simeq \Ind(\Coh(\scrA)^{S^1}) \times_{\Ind(\Coh(\scrC)^{S^1})} \Ind(\Coh(\scrB)^{S^1})\]---that is,
    we would need to commute $\Ind(-)$ with a limit in $\Pr^L_{\st}$ (not $\Pr^L_{\omega, st}$!),
    which does not hold in general. The existence of \cite[Proposition 13]{Tamme-excision}
    circumvents this need.
\end{rem}

Given a six-functor formalism, we may define the following presheaves
of categories.

\begin{defn}
	\label{defn: six-functor formalism with support}
Let $\mathscr{D}$ be a six-functor formalism, and let $X$ be a stack. Let $Z\hookrightarrow X$ be
a closed immersion, and let $U \hook X$ its open complement.
For each $Y \in \Rep_{/X}$, the assignment of the category
\[\mathscr{D}_Z(Y)\coloneq \fib(j_Y^!: \mathscr{D}(Y)\to \mathscr{D}(U \times_X Y)),\]
where $j_Y: U \times_X Y \to Y$ denotes the canonical open immersion,
defines a presheaf on $\Rep_{/X}$
\end{defn}

\begin{rem}
	\label{rem: value of the presheaf D_Z}
	Note that $Y \times_X U \hook Y$ is the complement
	to the closed immersion $Y \times_X Z \to Y$, so
	the value of the presheaf $\scrD_Z(-)$ on $Y \in \Rep_{/X}$
	is seen to be the category $\scrD_{Y \times_X Z}(Y)$,
	which by definition is the fiber of the functor $j_Y^!: \scrD(Y) \to \scrD(Y \times_X U)$.
\end{rem}

\begin{defn}
	\label{defn: restricted six-functor formalism}
	Let $\scrD$ be a six-functor formalism, and let $X$ be a stack.
	For a fixed $Z\to X$, the assignment to each $Y \in \Rep_{/X}$
	of the category $\scrD(Z \times_X Y)$ defines a presheaf
	which we denote by $\scrD[Z](-): \Rep_{/X} \to \DGCat_k$.
\end{defn}

For a given closed immersion $Z \hook X$, there is an obvious
natural transformation of presheaves, $\scrD[Z](-) \to \scrD_Z(-)$, whose
component at $Y \in \Rep_{/X}$ is given by
pushforward along the inclusion $Z \times_X Y \to Y$.

\begin{lem}\label{lem:cats-with-support-nis-descend}
	Let $X$ be a stack with affine diagonal, and let $\mathscr{D}: \Rep_{/X} \to \DGCat_R$ 
	be a six-functor formalism satisfying Nisnevich descent, where $R$ is any
	$\bbE_{\infty}$-algebra over $k$. 
	For any closed immersion $Z \hook X$,
	the presheaf $\mathscr{D}_Z$ satisfies Nisnevich descent. 
\end{lem}

\begin{proof}
 	As before, it suffices to show that $\scrD_Z(-)$ satisfies Nisnevich excision. 
	Let
	\[
	\begin{tikzcd}
		W\arrow[r]\arrow[d]&V\arrow[d]\\
		U\arrow[r]&S
	\end{tikzcd}
	\] 
	be a Nisnevich square in $\Rep_{/X}$. For $Y \in \Rep_{/X}$, denote by $Y^c$ the
	complement of the closed immersion $Z \times _X Y \hook Y$. Observe that the square
	\[
	\begin{tikzcd}
		W^c\arrow[r]\arrow[d]&V^c\arrow[d]\\
		U^c\arrow[r]&S^c
	\end{tikzcd}
	\] 
	is also Nisnevich. Indeed, it is Cartesian because we may write $W^c \simeq W\times _{S}U^c$. 
	Moreover, the base changes $V^c\to S^c$ and $U^c\to S^c$ are respectively \'etale and open, as these properties are preserved by base change. Similarly, if $V\to S$ was an isomorphism away from $U$, then its base change will also be an isomorphism away from the base change of $U$. Thus, by descent for $\mathscr{D}$, we have pullbacks
	\begin{align*}
		\mathscr{D}(S)=\mathscr{D}(V)\times _{\mathscr{D}(W)}\mathscr{D}(U)
	\end{align*} 
	and
	\begin{align*}
		\mathscr{D}(S^c)=\mathscr{D}(V^c)\times _{\mathscr{D}(W^c)}\mathscr{D}(U^c).
	\end{align*} 
	In order to identify
	\begin{align*}
		\mathscr{D}_Z(S)=\mathscr{D}_Z(V)\times _{\mathscr{D}_Z(W)}\mathscr{D}_Z(U),
	\end{align*} 
	we need only check that the kernel of the map $\mathscr{D}(S)\to \mathscr{D}(S^c)$ 
	induced by $j^!|_V \times j^!|_U$ has an analogous presentation as a fiber product. 
	We now use the following general observation. Let $\Lambda$ denote the category 
	$*\to *\leftarrow  *$. Then $\lim_{\Lambda}$ sends diagrams of shape $\Lambda$ 
	to their pullbacks. Moreover, if we have a diagram $A\to C\leftarrow B$ in $\mathsf{Map}(\Lambda,\Cat_\infty)$, 
	where $A=A_1\to A_3\leftarrow  A_2$ and similarly for $B,C$, then we have an object of 
	$\mathsf{Map}(\Lambda,\mathsf{Map}(\Lambda,\Cat_\infty))\simeq \mathsf{Map}(\Lambda \times \Lambda,\Cat_\infty)$. 
	Writing this diagram out, we have
	\[
	\begin{tikzcd}
		A_1\arrow[r]\arrow[d]&\arrow[d]B_1&\arrow[l]C_1\arrow[d]\\
		A_3\arrow[r]&B_3&\arrow[l]C_3\\
		A_2\arrow[u]\arrow[r]&B_2\arrow[u]&\arrow[l]C_2\arrow[u]
	\end{tikzcd}.
	\] 
	Applying $\lim_{\Lambda \times \Lambda}$ to this diagram and using that limits commute with limits, we find
	\begin{align*}
		\lim_{\Lambda \times \Lambda}(A\to C\leftarrow  B)=(\lim_\Lambda A)\times _{\lim_\Lambda C}(\lim_\Lambda B) = (A_1\times_{C_1}B_1)\underset{A_3\times _{B_3}C_3}{\times }A_2\times _{C_2}B_2.
	\end{align*} 
	Now, let $A$ be the span $\mathscr{D}(V)\to \mathscr{D}(V^c)\leftarrow  0$, 
	$C$ the span $\mathscr{D}(W)\to \mathscr{D}(W^c)\leftarrow  0$, 
	and $B$ the span $\mathscr{D}(U)\to \mathscr{D}(U^c)\leftarrow  0$.
\end{proof}

\subsection{The main theorem}

\begin{thm}
	\label{thm: Kashiwara's lemma for scalloped stacks}
	Let $X$ be a linearly scalloped stack almost of finite type and with affine diagonal, and let $i: Z\hookrightarrow X$ be an almost finitely presented inclusion of
	a closed substack. Then pushforward along $i$ induces
	an equivalence,
	\[i_*: \Pot(Z) \xrightarrow{\simeq} \Pot_Z(X).\]
\end{thm}

\begin{proof}
	Note that $X$ is in particular a quasi-compact algebraic stack with affine diagonal; hence, it is geometric. We may thus use the results of \zcref{ssec:nisnevich descent}. We will proceed by induction on scallop diagrams. Recall that a linearly scalloped stack
	$X$ admits a finite filtration by quasi-compact open substacks
	\[
		\emptyset=X_0\subset X_1\subset \cdots \subset X_n=X
	\]
	along with Nisnevich squares
	\begin{equation}
		\label{diag:nisnevich}
		\begin{tikzcd}
			W_i\arrow[r]\arrow[d]&V_i\arrow[d,"u_i"]\\
			X_{i-1}\arrow[r]&X_i
		\end{tikzcd}
	\end{equation} 
	where the $u_i$ are representable (by algebraic spaces) \'etale morphisms of 
	finite presentation and the $V_i$ are linearly quasi-fundamental, which means that each $V_i$ admits a
	quasi-projective morphism to $BG_i$ for a linearly reductive embeddable group scheme $G_i$.\footnotemark
	\footnotetext{See \cite[\S2.1]{KhanRavi} for the definition of embeddable group.}

	We recall the presheaf $\Pot[Z]$ defined in \zcref{defn: restricted six-functor formalism}
	above, and the natural transformation $\Pot[Z](-) \to \Pot_Z(-)$, whose
	component at $Y \in \Rep_{/X}$ is given by the pushforward,
	$(i_{Z \times_X Y})_*: \Pot(Z \times_X Y) \to \Pot_{Z \times_X Y}(Y)$.
	For a given $Y \in \Rep_{/X}$, let $i_Y: Z \times_X Y \to Y$ denote the canonical
	closed immersion.
	Since $\Pot[Z](-) \to \Pot_Z(-)$ is a map of presheaves, from the Nisnevich square
	\zcref[noname]{diag:nisnevich}, we obtain the following commutative diagram,
	\[\begin{tikzcd}
		{\Pot(Z \times_X X_i)} & {\Pot(Z \times_X V_i)\underset{\Pot(Z \times_X W_i)}{\times }\Pot(Z \times_X X_{i-1})} \\
		{\Pot_Z(X_i)} & {\Pot_Z(V_i)\times _{\Pot_Z(W_i)}\Pot_Z(X_{i-1})}
		\arrow["{(1)}", from=1-1, to=1-2]
		\arrow["{(i_{X_i})_*}", from=1-1, to=2-1]
		\arrow["{(i_{V_i})_* \times_{\left(i_{W_i}\right)_*} (i_{X_{i-1}})_*}", from=1-2, to=2-2]
		\arrow["{(2)}", from=2-1, to=2-2],
	\end{tikzcd}\]
	where the horizontal arrow $(1)$ and $(2)$ are the functors
	induced by the universal property of the fiber products on the right-hand side
	of the diagram.
	By \zcref{lem:cats-with-support-nis-descend}, the presheaf 
	$\Pot_Z(-)$ satisfies Nisnevich excision, so arrow $(2)$ is actually
	an equivalence.
	Similarly, $\Pot(-)$ 
	satisfies Nisnevich excision by \zcref{prop:potent-d-modules-nis-descend},
	so arrow $(1)$ is also an equivalence.
	Thus, in order to prove that $i_*$ is an equivalence,
	it suffices to show that
	$(i_{V_i})_* \times_{(i_{W_i})_*} (i_{X_{i-1}})_*$ is an equivalence.

	We proceed by induction on $i$. For the case $i=1$, the Nisnevich square \zcref[noname]{diag:nisnevich} shows that $X_1$ is quasi-fundamental, 
	since the $u_i$ are isomorphisms away from $X_{i-1}$ by the definition of Nisnevich square. Hence, $X_1 \to BG_1$ is a quasi-projective morphism for $G_1$ a linearly reductive embeddable group. In particular, $X_1 \simeq U_1/G_1$ for some $G_1$-quasi-projective
	scheme $U_1$, so \zcref{thm: Kashiwara's lemma for G-projective quotients} implies that
	\[(i_{X_1})_*: \Pot(Z_1) \xrightarrow{\simeq} \Pot_{Z_1}(X_1) 
	=: \Pot_Z(X_1).\footnotemark
	\footnotetext{See \zcref{rem: value of the presheaf D_Z}.}\]
	
	Now let $n \geq i >1$, and assume as our induction hypothesis that 
	\[(i_{X_{i-1}})_*: \Pot(Z \times_X X_{i-1}) \xrightarrow{} \Pot_Z(X_{i-1})\]
	is an equivalence.
	We now observe that, by the assumption that $X$ is linearly scalloped, $V_i$ admits a quasi-projective morphism $V_i \to BG_i$ for
	some linearly reductive embeddable group scheme $G_i$, hence has the form $U_i/G_i$ for some $G_i$-quasi-projective
	scheme $U_i$, as above. In the same vein, the map $W_i\to V_i$ is an open immersion,\footnotemark
	\footnotetext{It is the base change of the open immersion $X_{i-1} \to X_i$.} 
	so $W_i$ also admits a quasi-projective morphism to $BG_i$, meaning it is also the
	quotient of a $G_i$-quasi-projective scheme. Thus, \zcref{thm: Kashiwara's lemma for G-projective quotients}
	implies that both $(i_{V_i})_*: \Pot(Z \times_X V_i) \xrightarrow{} \Pot_Z(V_i)$ and $(i_{W_i})_*: \Pot(Z \times_X W_i) \xrightarrow{} \Pot_Z(W_i)$
	are equivalences. These two equivalences together with the induction hypothesis imply that
	$(i_{V_i})_* \times_{(i_{W_i})_*} (i_{X_{i-1}})_*$ is an equivalence.

	Taking $i = n$, we therefore obtain the desired
	equivalence, $i_*: \Pot(Z) \to \Pot_Z(X)$.
\end{proof}


\printbibliography

\end{document}